\documentclass[12t,reqno,twoside]{amsart}
\usepackage{enumerate}
\usepackage[english]{babel}
\usepackage{hyperref}
\usepackage{amssymb,amsthm,amsmath,eucal,mathrsfs}
\usepackage{bm}

\allowdisplaybreaks

\usepackage{amsfonts}
\usepackage{graphicx}
\usepackage[utf8]{inputenc}
\usepackage{lscape}
\usepackage{amstext}
\usepackage{xcolor}
\usepackage{float}
\usepackage{cancel}
\usepackage{mathrsfs}
\usepackage{epsfig}
\usepackage{url}
\usepackage{fancyhdr,lipsum}
\usepackage{pspicture}
\usepackage{wrapfig}
\usepackage{graphicx}
\usepackage{multirow}
\usepackage{lmodern}

\usepackage[T1]{fontenc}

\usepackage{cite}

\usepackage{dsfont}

\usepackage{hyperref}

\usepackage{cleveref}

\newtheorem{theorem}{Theorem}[section]

\newtheorem{lemma}[theorem]{Lemma}

\newtheorem{proposition}[theorem]{Proposition}

\theoremstyle{definition}

\newtheorem{remark}[theorem]{Remark}

\numberwithin{equation}{section}

\providecommand{\norm}[1]{\lVert#1\rVert} 

\newcommand\restr[2]{{
  \left.\kern-\nulldelimiterspace 
  #1 
  \vphantom{\big|} 
  \right|_{#2} 
  }}

\usepackage{eucal}

\usepackage{enumitem}
\usepackage[font=it]{caption}
\usepackage{caption}
\usepackage{subcaption}
\setenumerate{leftmargin=*} 

\usepackage{algorithmic}

\usepackage{mathtools}
\newcounter{mysubequations}

\renewcommand{\themysubequations}{(\arabic{mysubequations})}

\newcommand{\mysubnumber}{\refstepcounter{mysubequations}\themysubequations}

\title[Extraction of the mass density using highly dense small inclusions as contrast agents]{Extraction of the mass density using only the ${\mathtt{P}}$-parts of the elastic fields generated by injected highly dense small inclusions}
\address{ Corresponding Author: Divya Gangadaraiah}
\author{ Durga Prasad Challa$^{*}$, Divya Gangadaraiah$^{\dag}$,  Mourad Sini$^{\ddag}$ }
\address{$^{*}$Department of Mathematics and Statistics, IIT Tirupati, India.} 
\email{durga.challa@iittp.ac.in}
\address{$^{\dag}$Department of Mathematics and Statistics, IIT Tirupati, India.}
\email{divya.kavyashree@gmail.com}
\address{$^{\ddag}$RICAM, Austrian Academy of Sciences, Altenberger Strasse 69, A-4040, Linz, Austria.}
\email{mourad.sini@oeaw.ac.at}
\address{M. Sini is Partially supported by the Austrian Science Fund (FWF): P 30756-NBL and
P 32660.} 

\begin{document}
\maketitle

\begin{abstract}
  
We propose a reconstruction method to extract the variable mass density from the elastic farfields, with a single incident direction, measured before and after injecting highly dense small scaled inclusions. We take as a model, the Lam\'e system where the mass density is the unknown in $\Omega$ and the Lam\'e parameters are kept as known constants.
The injected small/dense inclusion,  $D:=z+aB\,(\subset\joinrel\subset\Omega)$  with $z$ as its location, $a\ll1$ as its maximum radius and $B$ of unit volume, generates a sequences of resonant frequencies. These special frequencies are related to the eigenvalues of the Lam\'e volume integral operator defined on the domain of the inclusion. Therefore, these resonant frequencies are, in principle, computable. After injecting the small inclusion at a location point $z$, we send an elastic incident plane wave at an incident frequency close to one of the mentioned resonant frequencies, say $\omega_{n_0}$.
Contrasting the ($\mathtt{p}$-parts of the) farfields generated, at one incident direction, before and after injecting this small inclusion, we provide an explicit formula that allows us to recover the total field $V^{t,\mathtt{p}}(z,-\hat{x})$ corresponding to $\mathtt{p}$-incident waves at the location of the small inclusion $z \in \Omega$. This total field is generated in the absence of the inclusion. Then we repeat the experiment by injecting more inclusions inside $\Omega$.  Using this reconstructed field in the Lam\'e PDE system, via a numerical differentiation, we recover the values of the mass density inside $\Omega$. It is worth mentioning that, we use measurements of dimension 3 to recover a function of 3 dimensions freedom. This makes the inverse problem not over determined. In addition, we use only the pressure wave and the $\mathtt{p}$-part of the farfield, for the reconstruction. To our best knowledge, this is the first result using only one type of elastic waves for the parameter identification.
\bigskip

\end{abstract}

\section{Introduction and statement of the results}
\subsection{Introduction}
~~~~~~~~~~~~~~~~~~~~~~~

Many of the inverse problems appearing in wave imaging, in the broad sense, consist in recovering material properties of the medium, of interest, from responses to external interrogations measured away from this domain. It is well known that reconstructing such materials is highly unstable, see \cite{Isakov-book, Choulli-book} for instance. Such instability is inherent to the problem setting because the corresponding forward problem is highly smoothing. Several approaches have been proposed to handle this instability, as regularization methods \cite{Engl-all-book} or coupling different type of interrogating waves like in hybrid imaging  \cite{ Bal-Review, Bal-all-review-2, Bal-all-Review-3}. Recently, based on engineering ideas of imaging using contrast agents, \cite{Chen-Li-2015}, we proposed an approach for solving the inverse problem at the expense of injecting into the medium of interest, whenever possible, small scaled resonating contrast agents.
Such contrast agents where originally proposed to the purpose of drugs delivery, see \cite{ HAN-all, KIM-all, Mitchel-all}, then they were used also to the imaging and therapy, \cite{Chen-Li-2015, HAN-all, KIM-all, Mitchel-all}. These small scaled resonators can be found in optics, as the plasmonic and the all-dielectric nanoparticles, in acoustics as micro-scaled bubbles and in mechanics as micro or nano-scaled inclusions. Therefore, these contrast agents could be used for the purpose of imaging based on the acoustic, optic or elastic waves based imaging (and also their related combinations). 
\bigskip

Let us focus now on imaging in the time-harmonic regime. The main idea is that under some critical scales of their size and contrasts, these agents will play the role of local resonators. Namely these resonators are related to the volume integral operators or surface integral operators appearing in the related Lippmann-Schwinger system modeling the scattering phenomenon under consideration. The resonant frequencies are then generated by the eigenvalues of those operators. As such, when the coupled system, given by the background and the injected contrast, is interrogated by incident frequencies close to those mentioned resonant frequencies, the generated fields are, additively, perturbed by local spots. These local spots model the interaction between the background medium and the small contrast. 
Our strategy is then to measure the generated field before and after injecting these contrast agents. Contrasting the two fields, we recover the field modeling the local spot we mentioned above. From this local spot, we recover the total field {(or other useful quantities) that we will use to reconstruct the needed material properties}. 
\bigskip

This approach was tested in acoustic, see \cite{DGS-2021}, in time harmonic and then time domain setting, see \cite{SW-2022, SSW-2023}, and in optics and photo-acoustic, see \cite{GS-2022, GS-2022-2, GS-2022-3}, where the underlying inverse problem was fully solved. In this current work, we extend this approach to the elastic tomography problem using contrasting inclusions. 
\bigskip

In acoustics, the bubbles can be modeled by contrast on the mass density or the bulk modulus. The useful resonating contrasting agents are the 'Minnaert' soft bubble, \cite{Ammari-F-G-L-Z, Ammar-all-bubble}, which is due to the balanced, but high, contrast between the mass density and the bulk modulus. It generates a resonant frequency, i.e. the Minnaert one, throughout the surface Double layer operator. Another option is the non-soft but bulky bubbles, see \cite{Ammari-all-Volume-potential, MPS-2022-volume-potential}, with which we can generate a sequences of resonant frequencies due to the scalar Newtonian operator. In optics, the modeling of the used nanoparticles is related to the permittivity or the permeability. We distinguish a sequences of all-dielectric resonant frequencies which are due to the vector-Newtonian operator projected on the divergence free subspace, \cite{Ammari-all-2023, CGS-2023}, or a sequences of plasmonic resonant frequencies, due to the surface double layer operator, or equivalently the vector-Newtonian operator projected on the subspace of the grad harmonic fields, \cite{GS-2022-2}. The situation in elasticity is richer than in acoustics or optics, as we have two types of speeds. To give a fortaste, we recall that the linearized and isotropic elastodynamic model is related to the mass density $\rho$  and the Lam\'e parameters 
$$
\lambda =\frac{v E}{(1+v)(1-2v)} \mbox{ and } \mu=\frac{E}{(1+v)}
$$
where $E$ is the Young's modulus, which is positive, and $v$ is the Poisson's ratio. Most of the natural solids have the Poisson's ratio $v$ in $[-1, \frac{1}{2}]$, see \cite{Lim}. Therefore $\lambda$ might be negative, as the auxetic solids, see \cite{Lakes, Lim}, and both $\lambda$ and $\mu$ can be large, as the saturated Clay or Rubber material see \cite{Lim}. 

We are interested in small scaled inclusions enjoying few of the following properties.

\begin{enumerate}
\item The mass density $\rho$ is high as compared to the background's one. In addition, the Poisson ration $v$ is in one of the following ranges: \label{item-1-for-nu}
\item []
\item  $v$ is in $(-1, \frac{1}{2})$ and away from $-1$ and $\frac{1}{2}$. In this case, the inclusion enjoys slow Shear and Pressure speeds. \label{item-2-for-nu}
\item []
\item  $v$ is close to $\frac{1}{2}$.  It enjoys a slow Shear speed and a moderate Pressure speed.\label{item-3-for-nu}
\item []
\item  $v$ is close to $-1$. It enjoys a moderate Shear and Pressure speed.\label{item-4-for-nu}
\end{enumerate}

In the current work, we only focus on contrasts coming from the mass density $\rho$, i.e., case \eqref{item-1-for-nu}-\eqref{item-2-for-nu}. In addition, for simplicity of the exposition, we assume the inclusions to have the Lam\'e constants $\lambda$ and $\mu$ the same as those of the background. Inspired by previous results using ultrasound measurements with bubbles as contrast agents, see \cite{DGS-2021, SSW-2023}, we do believe that with inclusions enjoying high mass density but moderate Lam\'e parameters would allow us to reconstruct only the mass density of the background medium, as it is stated and discussed in sections \ref{The model}, \ref{section-results} and \ref{section-Application}. However, using inclusions enjoying a high mass density with $v$ close to $\frac{1}{2}$, i.e., \eqref{item-1-for-nu} and \eqref{item-3-for-nu} above, then one could reconstruct, in addition to the mass density, also the Pressure speed. Likewise, using inclusions enjoying high mass density with $v$ close to $-1$, i.e., \eqref{item-1-for-nu} and \eqref{item-4-for-nu} above, then one could reconstruct, in addition to the mass density, both the Shear and the Pressure speeds. Therefore, one could solve the whole problem for this Lam\'e model.

\subsection{The Mathematical Model}\label{The model}\hfill
\par Let $D_j\subset \mathbb{R}^3$, $j=1,2,\cdots, M$ denote small inclusions defined by  $D_j:=z_j+aB_{{j}}$, where $B_{{j}}$ is an open, simply connected, bounded set with a Lipschitz boundary containing the origin. The parameter $a(>0)$ characterises the smallness of inclusion, and $z_j$ denotes the location of the inclusion $D_j$. 
Let $\Omega$ be an open, bounded and smooth domain in $\mathbb{R}^3$ that compactly contains the small inclusions $D_m,\, m = 1, \cdots, M$. The minimum distance between the small inclusions $D_m,\, m = 1, \cdots, M$, assumed to be positive, is denoted by $d$, i.e.,. 
\begin{equation}\label{def-dis-d-mp}
d:=\min\limits_{ 1\leq i,j\leq M,\, i\neq j } d_{ij} \quad [>0],  \mbox{ where }d_{ij}:=dist(D_i,D_j).
\end{equation} 
        
 We assume that the Lam\'e parameters $\lambda$ and $\mu $ are constants in $\mathbb{R}^3 $ satisfying $\mu>0$ and $3\lambda+2\mu >0$. Consider the mass density, denoted by $\rho(x)$, of the form 
\begin{equation}\label{def-of-rho}
 \rho(x) =\left\{ \begin{array}{c c c c}
  \rho_{0}(x)=\tilde{\rho_0} & \mbox{ in} & \mathbb{R}^3\setminus\bar{\Omega}
 \\
 \rho_{0}(x) & \mbox{ in} & \Omega  \\
 \rho_j:=c_ja^{-2}&\mbox{ in} & D_j, \; j=1,2,\cdots, M
\end{array}
\right.,
\end{equation}
where $\rho_{0}(x)$ is background mass density assumed to be $C^1$-smooth inside $\Omega$ and is constant $\tilde{\rho_0}$, independent of $a$, outside the $\Omega$.  Thus $\rho_j$ denotes the mass density of the inclusion $D_j$ with the scaling $\rho_{j}=c_ja^{-2}$,  where $j=1,2,\cdots,M$, and $c_j$ is a positive constant independent of $a$.

We are interested in the following problem, describing the time-harmonic elastic scattering by {inclusions}, 
\begin{equation}\label{elastic-problem-with-bubble-variable-density}
\left\{ \begin{array}{c c c } 
\Delta^e U^t+\omega^2\rho U^t = 0 & \mbox{in} &\mathbb{R}^3 \\ 
U^t|_{+}=U^t|_{-}\quad \mbox{and} \quad   T_{\nu_j}{U^t}|_{+}=T_{\nu_j}{U^t}|_{-}& \text{on} &  \partial D_j, j=1,2,\cdots,M
\end{array}\right.,
\end{equation}
where, $\omega>0$ is a given frequency,  $\nu_j$ denotes the outward unit normal to $\partial D_j$, $\Delta^e(\star)$ is the Lam\'e operator defined by $\Delta^e(\star):=(\lambda+\mu)\nabla(\nabla\cdot \star)+\mu\Delta(\star)$ and $T_{\nu}$ denotes the  conormal derivative defined as $T_{\nu}(*): =\lambda(\nabla\cdot (*))\nu+\mu(\nabla (*)+\nabla (*)^{\top})\nu$. Here, $U^t:=U^i+U^s$ is the total field, with $U^i$ denoting the incident wave (we restrict to plane incident waves) and $U^s$ denoting the elastic field scattered by the inclusion $D$ due to the incident field $U^i$. The incident field $U^i$ satisfies the Navier equation $\Delta^e U^i+\omega^2 \tilde{\rho_0} U^i=0 \,\, \text{in}\,\, \mathbb{R}^3$. Making use of the Helmholtz decomposition (see \cite{LOW-FREQUENCY}), the scattered field $U^s$ can be decomposed as $U^s=U_{\mathtt{p}}^s+U_{\mathtt{s}}^s$, where $U_{\mathtt{p}}^s:=-\frac{\nabla(\nabla\cdot U^s)}{{\kappa}_{\mathtt{p}}^2}$ denotes the pressure (or the longitudinal) part of the scattered field, and $U_{\mathtt{s}}^s:=\frac{\nabla\times(\nabla \times U^s)}{{\kappa}_{\mathtt{s}}^2}$ denotes the shear (transversal) part of the scattered field. Further, $U_{\mathtt{p}}^s$ and $U_{\mathtt{s}}^s$ satisfy the Sommerfeld Radiation Condition (\textit{S.R.C.}) separately, which together are called Kupradze Radiation Conditions (\textit{K.R.C.}) and are given below 
\begin{eqnarray}
\lim_{|x|\rightarrow \infty}|x|\left(\frac{\partial {U}_{\mathtt{p}}^{s}(x)}{\partial |x|}-\mathbf{\mathtt{i}}{\kappa}_{\mathtt{p}}{U}_{\mathtt{p}}^{s}(x)\right)=0 \quad \mbox{and}\quad
\lim_{|x|\rightarrow \infty}|x|\left(\frac{\partial {U}_{\mathtt{s}}^{s}(x)}{\partial |x|}-\mathbf{\mathtt{i}}{\kappa}_{\mathtt{s}} {U}_{\mathtt{s}}^{s}(x)\right)=0, \label{KRC-U-s-U-p-variable-rho}
\end{eqnarray} 
where the two limits are uniform in all the directions $\hat{x}:=\frac{x}{|x|}$ on the unit sphere $\mathbb{S}^2$. Here, ${\kappa}_{\mathtt{p}}$ and ${\kappa}_{\mathtt{s}}$ are longitudinal and the transversal wave numbers, respectively, i.e., ${\kappa}_{\mathtt{p}}:=\frac{\omega}{{c}_{\mathtt{p}}}$ and ${\kappa}_{\mathtt{s}}:=\frac{\omega}{{c}_{\mathtt{s}}}$, where ${c_{\mathtt{s}}}$ and ${c_{\mathtt{p}}}$ denotes the transversal and longitudinal phase velocities respectively and is given by ${c}_{\mathtt{s}}^2:=\frac{\mu}{\tilde{\rho_0}}$ and ${c}_{\mathtt{p}}^2:=\frac{\lambda+2\mu}{\tilde{\rho_0}}$. \\

By writing $D:=\cup_{j=1}^M D_j$, the scattering problem (\ref{elastic-problem-with-bubble-variable-density}--\ref{KRC-U-s-U-p-variable-rho}) can be equivalently formulated as
 \begin{align}
 & \Delta^e U^t|_{-}+\omega^2\rho_{j} U^t|_{-} = 0 \quad \mbox{in} \quad D_j, j=1,\cdots,M,\label{internal-elastic-variable-rho}\\ 
&\Delta^e U^t|_{+}+\omega^2\rho_{0}(x) U^t|_{+} = 0 \quad \mbox{in} \quad \mathbb{R}^3\setminus \bar{D},\label{external-elastic-variable-rho} \\
&U^t|_{+}=U^t|_{-} \quad \text{and } \quad T_{\nu}{U^t}|_{+}=T_{\nu}{U^t}|_{-}\quad \text{on} \quad \partial D,\label{Transmission-SM-variable-rho}\\
 &\lim_{|x|\rightarrow \infty}|x|\left(\frac{\partial {U}_{\mathtt{p}}^{s}(x)}{\partial |x|}-\mathbf{\mathtt{i}}{\kappa}_{\mathtt{p}}{U}_{\mathtt{p}}^{s}(x)\right)=0 \quad \text{and} \quad
\lim_{|x|\rightarrow \infty}|x|\left(\frac{\partial {U}_{\mathtt{s}}^{s}(x)}{\partial |x|}-\mathbf{\mathtt{i}}{\kappa}_{\mathtt{s}} {U}_{\mathtt{s}}^{s}(x)\right)=0\label{KRC-U-s-variable-rho}.
 \end{align}
The scattering problem (\ref{internal-elastic-variable-rho}--\ref{KRC-U-s-variable-rho}) is well posed in the H\"{o}lder or Sobolev spaces, see \cite{COLTON-IE,COLTON-KRESS,kupradze1980three,kupradze1967potential} for instance.\\

 We restrict to plane incident waves $U^i$ of the form, 
$
 U^i({x},\theta):=\beta_1\theta e^{\mathbf{\mathtt{i}}\kappa_{\mathtt{p}}\theta\cdot x}+\beta_2\theta^{\perp}e^{\mathbf{\mathtt{i}}\kappa_{\mathtt{s}}\theta\cdot x}$,
where $\beta_j$, $j=1,2$ are scalars, $\theta$ is the incident direction in $\mathbb{S}^2$ and $\theta^{\perp}$ is any direction in $\mathbb{S}^2$ perpendicular to $\theta$.  We denote $\mathtt{p}$ and $\mathtt{s}$-parts of the incident waves  by $U_{\mathtt{p}}^i({x},\theta)$ and $U_{\mathtt{s}}^i({x},\theta)$ respectively, i.e., in our case $U_{\mathtt{p}}^i({x},\theta):=\beta_1\theta e^{\mathbf{\mathtt{i}}\kappa_{\mathtt{p}}\theta\cdot x}$ and $U_{\mathtt{s}}^i({x},\theta):=\beta_2\theta^{\perp}e^{\mathbf{\mathtt{i}}\kappa_{\mathtt{s}}\theta\cdot x}$. Observe that in three-dimensional setting,  any perpendicular direction $\theta^{\perp}$ is combination of two orthogonal directions, namely horizontal and vertical perpendicular directions $\theta^{\perp_h},\theta^{\perp_v}$ respectively, and so $\theta^{\perp}:=\frac{\alpha_1}{\sqrt{\alpha_1^2+\alpha_2^2}}\theta^{\perp_h}+\frac{\alpha_2}{\sqrt{\alpha_1^2+\alpha_2^2}}\theta^{\perp_v}$, where $\alpha_1,\alpha_2$ are scalars, define $\beta_{2_h}:=\beta_2 \frac{\alpha_1}{\sqrt{\alpha_1^2+\alpha_2^2}}$ and $\beta_{2_v}:=\beta_2\frac{\alpha_2}{\sqrt{\alpha_1^2+\alpha_2^2}} $. Given a $\theta$, one can compute $\theta^{\perp_h},\theta^{\perp_v}$ explicitly, for instance using the rotation matrix which transforms $\theta$ to $e_3=(0,0,1)^\top$, see \cite[(3.5, 3.6)]{DPC-MS-InvProblems-2012}. Hence, the shear waves also have two orthogonal components called horizontal and vertical shear directions denoted by $\mathtt{s}_h$ and $\mathtt{s}_v$ in the direction of $\theta^{\perp_h}$ and $\theta^{\perp_v}$ respectively.

 The scattered field $U^s$ has an asymptotic behaviour of the form \cite{alves2002far}
\begin{eqnarray}
U^s(x,\theta)=\dfrac{e^{\mathbf{\mathtt{i}}\kappa_{\mathtt{p}}|x|}}{|x|}U_{\mathtt{p}}^\infty(\hat{x},\theta)+\dfrac{e^{\mathbf{\mathtt{i}}\kappa_{\mathtt{s}}|x|}}{|x|}U_{\mathtt{s}}^\infty(\hat{x},\theta)+O\bigg(\dfrac{1}{|x|^2}\bigg),\;\; |x|\to \infty. \label{Us-assymptotic-expansion}
\end{eqnarray}
 We denote the elastic farfield  by $U^\infty(\hat{x},\theta)$, $(\hat{x},\theta)\in \mathbb{S}^2\times\mathbb{S}^2$ corresponding to the scattered field $U^s(x,\theta)$, and is defined as the sum of the $\mathtt{p}$-farfield $U_{\mathtt{p}}^\infty(\hat{x},\theta)$ and $\mathtt{s}$-farfield $U^\infty_{\mathtt{s}}(\hat{x},\theta)$, which are defined as the coefficients of $\frac{e^{\mathbf{\mathtt{i}}\kappa_{\mathtt{p}}|x|}}{|x|}$ and $\frac{e^{\mathbf{\mathtt{i}}\kappa_{\mathtt{s}}|x|}}{|x|}$ respectively in the asymptotic expansion corresponding to incident and propagation unit directions $\theta$ and $\hat{x}$. Also, we use $U_\mathtt{q}^{\infty,\mathtt{j}}(\hat{x},\theta)$, for $\mathtt{j},\mathtt{q}=\mathtt{p},\mathtt{s} $ to denote the $\mathtt{q}$ part of the farfield associated to $\mathtt{j}$ part of the incident wave. Similarly, we use $U_\mathtt{q}^{s,\mathtt{j}}({x},\theta)$ and $U_\mathtt{q}^{t,\mathtt{j}}({x},\theta)$, for $ \mathtt{j},\mathtt{q}=\mathtt{p},\mathtt{s} $ to respectively denote the $\mathtt{q}$ part of the scattered and total fields associated to $\mathtt{j}$ part of the incident wave.
\\

Let $V^t$ the total field  satisfying the following background problem, i.e., in the absence of the inclusion $D$:
{
\begin{equation}\label{background-problem-without-bubble-variable-density}
\left\{ \begin{array}{c c c }
 (\Delta^e+\omega^2\rho_0(x))V^t(x)=0, & x\mbox{ in } &  \mathbb{R}^3\\
 &&\\
\underset{|x|\to \infty}{ \lim}|x|\left(\dfrac{\partial V^s_{\mathtt{s}}}{\partial |x|}-\mathbf{\mathtt{i}}{\kappa}_{\mathtt{s}}V_{\mathtt{s}}^s\right)=0 &  &  \\
  &&\\
  \underset{|x|\to \infty}{\lim}|x|\left(\dfrac{\partial V^s_{\mathtt{p}}}{\partial |x|}-\mathbf{\mathtt{i}}{\kappa}_{\mathtt{p}}V_{\mathtt{p}}^s\right)=0&  & 
\end{array},
\right.
\end{equation}
where $V^t:=V^i+V^s$, with $V^i$, denoting the incident wave, which we considered as a plane incident wave $U^i$, and $V^s$ denoting the scattered field, and $V_{\mathtt{p}}^s$ and $V_{\mathtt{s}}^s$ denoting the corresponding $\mathtt{p}$ and $\mathtt{s}$-parts of the scattered field $V^s$ respectively. We can observe that $V_{\mathtt{p}}^s$ and $V_{\mathtt{s}}^s$ are satisfying the \textit{K.R.C.}, in which the limits are uniform in all the directions $\hat{x}:=\dfrac{x}{|x|}$ of the unit sphere $\mathbb{S}^2$. 
\medskip\\
Let us set $G^\omega(x,y)$ to be the Green's matrix of the elastic scattering problems \eqref{background-problem-without-bubble-variable-density}. Hence, it satisfies the following equation in the distribution sense
\begin{eqnarray}
\mu\Delta_x G^\omega(x,y)+(\lambda + \mu)\nabla_x(\nabla_x\cdot G^\omega(x,y))+\omega^2\rho_{0}(x) G^\omega(x,y) = -\delta (x,y){I}, \quad \mbox{in} \quad \mathbb{R}^3,
\label{Greens-function-variable-rho}
\end{eqnarray}
and it satisfies the asymptotic expansion.
 {The asymptotic expansion of $G^\omega$ is given by
\begin{eqnarray}
G^\omega(x,y)=G^{\omega,\infty}_{\mathtt{p}}(\hat{x},y)\frac{e^{\mathtt{i}\kappa_\mathtt{p}|x|}}{|x|}+G^{\omega,\infty}_{\mathtt{s}}(\hat{x},y)\frac{e^{\mathtt{i}\kappa_\mathtt{s}|x|}}{|x|}+O(|x|^{-2}),\; |x|\to\infty\label{assymptotic-expansion-Green}
\end{eqnarray}
}
where $I$ is $3\times 3$ identity matrix.
\medskip\\
Let $\Gamma^\omega(x,y)$ denote the fundamental matrix, also called as Kupradze matrix, to the Navier-equation associated with the frequency $\omega$ and the fixed background mass density $\tilde{\rho_0}$, i.e., $\Gamma^\omega(x,y)$ satisfies the following problem 
\begin{eqnarray}
\Delta^e \Gamma^\omega(x,y)+\omega^2\tilde{\rho_0}\Gamma^\omega(x,y)=-\delta(x,y)I, & \text{in } \mathbb{R}^3,\label{fundamental-homogeneous-background-massdensity-tildrho}
\end{eqnarray}
with the related Kupradze radiation conditions.
\medskip

In addition, let $\Gamma^\omega_z(x,y)$ denote the fundamental matrix to the Navier-equation associated with the frequency $\omega$ and fixed mass density $\rho_0(z)$.

The explicit form of the Kupradze matrix $\Gamma^{\omega}(x,y)=(\Gamma_{ij}^\omega(x,y))$ is given by (see \cite{AMMARI-BIOMEDICAL-IMAGING} for instance)
\begin{eqnarray}
\Gamma^\omega(x,y)=\dfrac{1}{4\pi\mu}\dfrac{e^{\mathbf{\mathtt{i}}\kappa_{\mathtt{s}}|x-y|}}{|x-y|}I+\dfrac{1}{4\pi\mu\kappa_{\mathtt{s}}^2}\nabla\nabla^{\top}\left(\dfrac{e^{\mathbf{\mathtt{i}}\kappa_{\mathtt{s}}|x-y|}}{|x-y|}-\dfrac{e^{\mathbf{\mathtt{i}}\kappa_{\mathtt{p}}|x-y|}}{|x-y|}\right),\end{eqnarray}
and it satisfies the  asymptotic expansion, see \cite{LOW-FREQUENCY} for instance,
\begin{eqnarray}
\Gamma^\omega({x},y)&=& {\dfrac{1}{4\pi(\lambda+2\mu)}\hat{x}\,\hat{x}^{\top}\dfrac{e^{\mathbf{\mathtt{i}}\kappa_{\mathtt{p}}|x|}}{|x|} e^{-\mathbf{\mathtt{i}}\kappa_{\mathtt{p}}\hat{x}\cdot y}+\dfrac{1}{4\pi\mu}(I-\hat{x}\,\hat{x}^{\top})\dfrac{e^{\mathbf{\mathtt{i}}\kappa_\mathtt{s}|x|}}{|x|}e^{-\mathbf{\mathtt{i}}\kappa_{\mathtt{s}}\hat{x}\cdot y}}+O(|x|^{-2}),\;\; |x|\to \infty.\quad\label{assymptotic-expansion-Gamma-omega}
\end{eqnarray}

We can also observe that the $ik^{th} $ entry of $\Gamma^{\omega}(x,y)$ is given by \cite{AMMARI},
\begin{equation}
\Gamma^{\omega}_{ik}(x,y)=\frac{\delta_{ik}}{4\pi \mu|x-y|}e^{\mathbf{\mathtt{i}}\kappa_{\mathtt{s}}|x-y|}+\frac{1}{4\pi\omega^2\tilde{\rho_0}}\partial_i\partial_k \frac{e^{\mathbf{\mathtt{i}}\kappa_{\mathtt{s}}|x-y|}-e^{\mathbf{\mathtt{i}}\kappa_{\mathtt{p}}|x-y|}}{|x-y|},
\end{equation}
where $\delta_{ik}$ denotes the Kronecker symbol.
\\
Hence, we can write the  series representation of  each entry $\Gamma_{ik}^\omega(x,y)$  as follows, see  \cite{LAYER-POTENTIAL-TECHNIQUES, AMMARI-MATHEMATICAL-METHODS-IN-ELASTIC-IMAGING}:
\begin{eqnarray}\label{entrywise-FM}
\Gamma^{\omega}_{ik}(x,y)&=&\frac{1}{4\pi\tilde{\rho_0}}\sum_{n=0}^{\infty}\frac{\mathbf{\mathtt{i}}^n}{(n+2)n!}\left( \frac{n+1}{c_{\mathtt{s}}^{n+2}}+\frac{1}{c_{\mathtt{p}}^{n+2}}\right)\omega^n\delta_{ik}|x-y|^{n-1}\nonumber\\
&&- \frac{1}{4\pi\tilde{\rho_0}}\sum_{n=0}^{\infty}\frac{\mathbf{\mathtt{i}}^n(n-1)}{(n+2)n!}\left(\frac{1}{c_{\mathtt{s}}^{n+2}}-\frac{1}{c_{\mathtt{p}}^{n+2}}\right)\omega^n|x-y|^{n-3}(x-y)_i(x-y)_k.
\end{eqnarray}
Further, for the zero frequency, the associated fundamental matrix $\Gamma^0$, also called as Kelvin matrix, can be simplified as below, and observe that it is symmetric, i.e., $\Gamma^0(x,y)={\Gamma^0(x,y)}^{\top}$ and $\Gamma^0(x,y)=\Gamma^0(y,x)$:
\begin{equation}\label{entrywise-FM-zerof}
\left\{
\begin{array}{c c c}
\Gamma_{ik}^0(x,y)=\dfrac{\gamma_1}{4\pi}\dfrac{\delta_{ik}}{|x-y|}+\gamma_2\dfrac{(x-y)_i(x-y)_k}{4\pi |x-y|^3},
\\
\\
\mbox{where,}\;\;\gamma_1=\frac{1}{2}\left(\frac{1}{\mu}+\frac{1}{2\mu+\lambda}\right) \;\; \mbox{and} \;\; \gamma_2=\frac{1}{2}\left(\frac{1}{\mu}-\frac{1}{2\mu+\lambda}\right).
\end{array}\right.
\end{equation}

\subsection{ {Statement  of the results}}\label{section-results}
\hfill

\par We consider incident frequencies $\omega$  satisfying}

\begin{equation}\label{omega-Minnaert}
\vert \omega^2-\omega_{n_{0_{(j)}}}^2\vert \sim a^{h_j}, \mbox{ with } 0<h_j<1  \quad\text{and }\quad
\omega_{n_{0_{(j)}}}^2:=\dfrac{1}{\rho_j\lambda_{n_0}^{D_j}},
\end{equation}
and $\frac{1}{2}\max\{\kappa_{\mathtt{s}},\kappa_{\mathtt{p}}\}\,diam(D)<1$.
Here, $(\lambda_n^{D_j},e_n^{D_j})_{n\in\mathbb{N}}$  is the complete orthonormal eigen system  of  the Navier operator $N_{D_j}^0:(L^2(D_j))^3\to (L^2(D_j))^3$ defined by $N_{D_j}^0(U)(x):=\int_{D_j} \Gamma_z^0(x,y)\cdot U(y)dy$, which is compact and self-adjoint. In addition, we set the polarization tensor
\begin{equation}\label{def-of-mat-EnB-and -omegano-First theorem}
E_{n_0}^{D_j}:= \sum_{l=1}^{l_{\lambda_{n_0}^j}}\int_{D_j}{{e}_{n_{0_l}}^{D_j}}(\xi)\,d\xi\otimes\int_{D_j}{e}_{n_{0_l}}^{D_j}(\xi)\,d\xi,
\end{equation}
where $e_{{n_0}_l}^{D_j}$, for $n_0$  {fixed}, are the eventual $l_{\lambda_{n_0}^j}$-multiple linearly independent eigenfunctions corresponding to the eigenvalue $\lambda_{n_0}^{D_j}$.
\bigskip

\noindent Observe that $\omega_{n_{0_{(j)}}}^{-2}= c_j \lambda^{B_j}_{{n_0}}$ and $E_{n_0}^{D_j}=a^3E_{n_0}^{B_j}$, because $\rho_j = c_j a^{-2}$, $\lambda^{D_j}_{n_0}=a^2 \lambda^{B_j}_{n_0}$ and $\int_{D_j}e_n^{D_j}(y)dy=a^{\frac{3}{2}}e_n^{B_j}(\xi)d\xi$, where $(\lambda_n^{B_j},e_n^{B_j})_{n\in\mathbb{N}}$  denotes the complete orthonormal eigen system  of  the Navier operator ${N}_{B_j}^0:(L^2(B_j))^3\to (L^2(B_j))^3$ defined by ${N}_{B_j}^0(U)(\zeta):=\int_{B_j} \Gamma^0(\zeta,\eta) \cdot U(\eta)d\eta$ and $E_{n_0}^{B_j}:=\sum_{l=1}^{l_{\lambda_{{n_0}_j}}}\int_{B_j}{{e}_{n_{0_l}}^{B_j}}(\eta)\,d\eta\otimes\int_{B_j}{e}_{n_{0_l}}^{B_j}(\eta)\,d\eta$. Therefore $\omega_{{n_0}_{(j)}} \sim 1$, as $0<a \ll 1$. For simplicity, we assume all $B_j$ are identical and denote them as $B$. However, the results also hold if $B_j$ are different.
\bigskip

\noindent In the following theorem, we state the approximate estimations of the scattered field and the farfield in terms of the parameter $a$.
\begin{theorem}\label{theorem}
Let $D_j:=z_j+a B$, $j=1,2,\cdots, M$, be small inclusions with Lipschitz boundary in $\mathbb{R}^3$ and $\Omega$ be a bounded domain which contains all $D_j$'s compactly. Let $\rho_0$ be a positive function of class $C^1$ in $\Omega$, equals a constant $\tilde{\rho_0}$ outside $\Omega$. The Lam\'e parameters $\lambda$ and $\mu $ are assumed to be constants throughout. In addition, let $\rho_j$ be the inclusion $D_j$'s density  and behaves as $\rho_j=c_ja^{-2}$, where $c_j>0$ is a constant independent of $a$. Then the scattering problem (\ref{internal-elastic-variable-rho}--\ref{KRC-U-s-variable-rho}) has a unique solution. In addition, under the condition 
\begin{equation}\label{condition-M}
(M-1)\frac{a}{d}<a^h, \; a\ll 1,
\end{equation}
the corresponding scattered and farfields respectively satisfies the following asymptotic expansions, uniformly for $x$ in a bounded domain in the exterior of $\Omega$ and  for all the directions  $\theta$ and $\hat{x}$ in $\mathbb{S}^2$;
\begin{enumerate}
\item The scattered field $U^s$ satisfies:
\begin{align}
U^s(x,\theta,\omega)&{=}V^s(x,\theta,\omega) +\hspace{-0.05cm}\sum_{j=1}^M\hspace{-0.05cm}\frac{\rho_j\omega^2\omega_{n_{0_{(j)}}}^2}{\omega_{n_{0_{(j)}}}^2\hspace{-0.4cm}-\omega^2} 
G^\omega(x,z_j)\hspace{-0.05cm}\cdot \hspace{-0.05cm}(E_{n_0}^{D_j}\hspace{-0.05cm}\cdot\hspace{-0.05cm} V^t(z_j,\theta,\omega)
 )
\hspace{-0.05cm}+O\left(\max\{Ma,M^2a^{1-2h}\frac{a}{d}\}\right).\label{theorem-scatteredfield-mp}
\end{align}
\item The pressure and shear farfields $U^\infty_{\mathtt{p}}$ and $U^\infty_{\mathtt{s}}$, corresponding to the scattered field $U^s$ satisfies:
\begin{align}
U^\infty_{\mathtt{p}}\hspace{-0.05cm}(\hat{x},\theta)& \hspace{-0.1cm}  = \hspace{-0.1cm} V^\infty_{\mathtt{p}}\hspace{-0.05cm}(\hat{x},\theta)\hspace{-0.09cm}+\hspace{-0.1cm}\frac{1}{4\pi(\lambda\hspace{-0.05cm}+\hspace{-0.05cm}2\mu)\beta_1}\hspace{-0.1cm}\sum_{j=1}^M\hspace{-0.1cm}\frac{\rho_j\omega^2\omega_{n_{0_{(j)}}}^2}{\omega_{n_{0_{(j)}}}^2\hspace{-0.45cm}-\hspace{-0.1cm}\omega^2}\hat{x}\hspace{-0.05cm}\cdot\hspace{-0.05cm}\left(\hspace{-0.05cm}E_{n_0}^{D_j}\hspace{-0.08cm}\cdot\hspace{-0.05cm}  V^{t,\mathtt{p}}(z_j,\theta)\hspace{-0.05cm}\right)\hspace{-0.05cm}V^{t,\mathtt{p}}(z_j,\hspace{-0.05cm}-\hat{x})
\hspace{-0.1cm}+O\hspace{-0.05cm}\left(\max\{Ma,M^2a^{1-2h}\frac{a}{d} \}\right) \label{theorem-p-part-farfield-mp}\\
\mbox{ and } 
\nonumber\\
     U^\infty_{\mathtt{s}}(\hat{x},\theta)&\hspace{-0.05cm}=\hspace{-0.05cm}V^\infty_{\mathtt{s}}(\hat{x},\theta)\hspace{-0.05cm}+\hspace{-0.05cm}\sum_{j=1}^M\hspace{-0.05cm}\frac{\rho_j}{4\pi\mu}\frac{\omega^2\omega_{n_{0_{(j)}}}^2}{\omega_{n_{0_{(j)}}}^2\hspace{-0.35cm}-\omega^2} \hspace{-0.05cm}\left[\hspace{-0.05cm}\frac{(E_{n_0}^{D_j}\hspace{-0.05cm}\cdot\hspace{-0.05cm} V^{t}(z_j,\theta))\hspace{-0.05cm}\cdot \hat{x}^{\perp_h}}{\beta_{2_h}} V^{t,\mathtt{s}_h}(z_j,-\hat{x})\hspace{-0.05cm}+\hspace{-0.05cm}\frac{(E_{n_0}^{D_j}\cdot V^{t}(z_j,\theta))\cdot \hat{x}^{\perp_v}}{\beta_{2_v}} V^{t,\mathtt{s}_v}(z_j,-\hat{x})\right] \nonumber\\
     &\qquad
     +O\left(\max\{Ma,M^2a^{1-2h}\frac{a}{d} \}\right).\label{theorem-farfield-s-part-mp}
\end{align}
\end{enumerate}
In particular, for the incident direction $\theta=-\hat{x}$, considering the $\mathtt{p}$-incident wave $U_{\mathtt{p}}^i$, \eqref{theorem-p-part-farfield-mp} simplifies to

\begin{align}
U^{\infty,\mathtt{p}}_{\mathtt{p}}(\hat{x},-\hat{x})  =
&V^{\infty,\mathtt{p}}_{\mathtt{p}}(\hat{x},-\hat{x}) + \hspace{-0.05cm}\sum_{j=1}^M\hspace{-0.05cm}\frac{\rho_j}{4\pi(\lambda+2\mu)\beta_1}\frac{\omega^2\omega_{n_{0_{(j)}}}^2}{(\omega_{n_{0_{(j)}}}^2-\omega^2)}{\left( \hat{x}\cdot  \left(E_{n_0}^{D_j} \cdot V^{t,\mathtt{p}}(z_j,-\hat{x})\right) \right)}V^{t,\mathtt{p}}(z_j,-\hat{x})
\nonumber\\
&\quad +O\left(\max\{Ma,M^2a^{1-2h}\frac{a}{d} \}\right).\label{farfield-variable-density-backscattered-field}
\end{align}
Similarly, for $\mathtt{s}_h$ and $\mathtt{s}_v$ incident waves $U^{i}_{\mathtt{s_h}}$ and  $U^{i}_{\mathtt{s}_v}$,  \eqref{theorem-farfield-s-part-mp} simplifies to the following respectively:
\begin{align}
U^{\infty,\mathtt{s}_h}_{\mathtt{s}}(\hat{x},-\hat{x})=&V^{\infty,\mathtt{s}_h}_{\mathtt{s}}(\hat{x},-\hat{x})+\hspace{-0.05cm}\sum_{j=1}^M\hspace{-0.05cm}\frac{\rho_j}{4\pi\mu\beta_2}\frac{\omega^2\omega_{n_{0_{(j)}}}^2}{\omega_{n_{0_{(j)}}}^2\hspace{-0.35cm}-\omega^2} \hspace{-0.05cm}\bigg[(E_{n_0}^{D_j}\cdot V^{t,\mathtt{s}_h}(z_j,-\hat{x}))\cdot \hat{x}^{\perp_h}V^{t,\mathtt{s}_h}(z_j,-\hat{x})
     \nonumber\\
     &+(E_{n_0}^{D_j}\cdot V^{t,\mathtt{s}_h}(z_j,-\hat{x}))\cdot \hat{x}^{\perp_v} V^{t,\mathtt{s}_v}(z_j,-\hat{x})\bigg]+O\left(\max\{Ma,M^2a^{1-2h}\frac{a}{d} \}\right),\label{theorem-s-h-incident-farfield-backscat}
     \\
U^{\infty,\mathtt{s}_v}_{\mathtt{s}}(\hat{x},-\hat{x}) = &V^{\infty,\mathtt{s}_v}_{\mathtt{s}}(\hat{x},-\hat{x}) + \hspace{-0.05cm}\sum_{j=1}^M\hspace{-0.05cm}\frac{\rho_j}{4\pi\mu\beta_2}\frac{\omega^2\omega_{n_{0_{(j)}}}^2}{\omega_{n_{0_{(j)}}}^2\hspace{-0.35cm}-\omega^2} \hspace{-0.05cm}\bigg[(E_{n_0}^{D_j} \cdot  V^{t,\mathtt{s}_v}(z_j,-\hat{x})) \cdot  \hat{x}^{\perp_h}V^{t,\mathtt{s}_h}(z_j,-\hat{x}) 
     \nonumber\\
     &+ (E_{n_0}^{D_j} \cdot  V^{t,\mathtt{s}_v}(z_j,-\hat{x})) \cdot  \hat{x}^{\perp_v} V^{t,\mathtt{s}_v}(z_j,-\hat{x})\bigg] + O\left(\max\{Ma,M^2a^{1-2h}\frac{a}{d} \}\right).\label{theorem-s-v-incident-farfield-backscat}
\end{align}
In the above expansions, { $V^{t,\mathtt{p}}(x,\theta)$ is the solution to the problem \eqref{background-problem-without-bubble-variable-density}, representing the total field corresponding to the $\mathtt{p}$-incident wave, and $V^{\infty,\mathtt{p}}(\hat{x},\theta,\omega)$ represents its farfield for the incident frequency $\omega$}. 
\end{theorem}

\subsection{An application: Reconstruction of the mass density $\rho_0$ using highly dense injected inclusions}\label{section-Application}\hfill
\par As an application of Theorem \ref{theorem}, we propose to solve the inverse problem of reconstructing the background mass density using the following measurements. 
\par 
{\bf{Measurements.}}  We assume to have access to  the $\mathtt{p}$-part of the farfield data, corresponding to $\mathtt{p}$-incident waves before and after injecting an inclusion at a point $z\in\Omega$, corresponding to a single incident direction $\theta\in\mathbb{S}^2$ and the corresponding back-scattering direction $\hat{x}=-\theta$ at a frequency $\omega$  satisfying $|\omega^2-\omega_{n_0}^2|\sim a^h$, i.e., measure $V_p^{\infty,p}(\hat{x},-\hat{x},\omega)$ and
$U_p^{\infty,p}(\hat{x},-\hat{x},\omega,z) $ respectively for single incident direction $\theta$ and few observation directions $\hat{x}$. We repeat these measurements by injecting the inclusions $D_j$, one after another, to different locations $z \in \Omega$. Precisely, we inject the first inclusion, perform the measurement, then inject the second inclusion, and perform the similar measurements,  continuing this process until the entire domain $\Omega$ is scanned. Our main assumption is that these inclusions are well separated. Precisely,  let the minimum distance $d$ between the injected inclusions satisfies
 $
 (M-1)\frac{a}{d}<a^h, \, a \ll 1.
$ 
\newline
\noindent The last condition, on the number and separation of the inclusions, allows us to derive expansions of the fields in the Born-regime, avoiding multiple interaction between the inclusions. This condition is important in the deriving the algorithm below.
\bigskip

\noindent Before describing the reconstruction scheme, let us first discuss the dimensionality of these measurements. In terms of dimensionality, we observe that in terms of measurements, we have only $3$ degrees of freedom ($3D$) equal to the $3D$ degrees of freedom for the unknown mass density $\rho(\cdot)$. This makes the related inverse problem equally determined (i.e. not overdetermined).  In addition, we use only the $\mathtt{p}$-parts of the farfield measurements, i.e. only the Pressure waves.
\medskip

{\bf{Reconstruction scheme.}} Our reconstruction procedureis divided into a few steps. 
 {
\begin{enumerate}
\item[\textbf{Step 1}]  \label{Vt-dot-int-e_n0-D-using-farfield-in-steps}Make use of the backscattered  farfield data  $V^{\infty,\mathtt{p}}_{\mathtt{p}}(\hat{x},-\hat{x},\omega)$ and $U^{\infty,\mathtt{p}}_{\mathtt{p}}(\hat{x}, -\hat{x},\omega)$ in the farfield expansion \eqref{farfield-variable-density-backscattered-field} {for $M=1$, i.e., by sending single inclusion} . Multiply \eqref{farfield-variable-density-backscattered-field} from left side by matrix $E_{n_0}^{B}$  then by unit vector $\hat{x}$ to recover  {$\hat{x}\cdot  \left(E_{n_0}^{B} \cdot V^{t,\mathtt{p}}(z_1,-\hat{x})\right)$, for $z_1\in\Omega$ upto a sign}.
\\
\item [\textbf{Step 2}]
Make use of   {$\hat{x}\cdot  \left(E_{n_0}^B \cdot V^{t,\mathtt{p}}(z_1,-\hat{x})\right)$, for $z_1\in\Omega$} in \eqref{farfield-variable-density-backscattered-field} to recover  {$V^{t,\mathtt{p}}(z_1)$ upto a sign}. 
\\
\item [\textbf{Step 3}]  {Measure the backscattered farfields by sending the second inclusion and make use of \eqref{farfield-variable-density-backscattered-field} for $M=2$. From Step 1, we have $ \left(E_{n_0}^{B} \cdot V^{t,\mathtt{p}}(z_1,-\hat{x})\right)$, for $z_1\in\Omega$ upto a sign. By substituting known terms in \eqref{farfield-variable-density-backscattered-field}, we recover  {$ \left(E_{n_0}^{B} \cdot V^{t,\mathtt{p}}(z_2,-\hat{x})\right)$, for $z_2\in\Omega$ upto a sign}.  Hence we can recover {$V^{t,\mathtt{p}}(z_2)$ upto a sign}, following Step 2.
\\
\item [\textbf{Step 4}] Similarly, measure the backscattered farfield measurements by sending the third inclusion into the target region.  Using these farfield measurements, $\hat{x}\cdot  \left(E_{n_0}^{B} \cdot V^{t,\mathtt{p}}(z_1,-\hat{x})\right)$ from Step 1, and $\hat{x}\cdot  \left(E_{n_0}^{B} \cdot V^{t,\mathtt{p}}(z_2,-\hat{x})\right)$ from Step 3, in the expansion \eqref{farfield-variable-density-backscattered-field} for $j=3$, recover the value {$ \hat{x}\cdot  \left(E_{n_0}^{B} \cdot V^{t,\mathtt{p}}(z_3,-\hat{x})\right)$ upto a sign}. Further, following Step 2, recover the total field $V^t$ at the location of $D_3$, i.e., $V^t(z_3)$.
\\
\item [\textbf{Step 5}] Similarly, by sending one by one inclusion, recover {$\hat{x}\cdot  \left(E_{n_0}^{B} \cdot V^{t,\mathtt{p}}(z_j,-\hat{x})\right)$ upto a sign and hence $V^t(z_j,-\hat{x})$, $j=1,2,\cdots, M$}.
\\
\item [\textbf{Step 6}] Reconstruct $\rho_0(z_j)$ using the equation \eqref{background-problem-without-bubble-variable-density} for $\mathtt{p}$-incident wave, i.e.,
\begin{eqnarray}
 {\Delta^e V^{t,\mathtt{p}}(z_j)+\omega^2\rho_0(z)V^{t,\mathtt{p}}(z_j)=0.}
\end{eqnarray}
 This reconstruction is possible with the knowledge of  $\Delta^e V^t(z_j)$ for $z_j\in\Omega$, which can be done by applying an effective numerical differentiation technique. Hence, the mass density $\rho_0$ at the location $z_j$ can be recovered
 {as,
 \begin{eqnarray}
\rho_0(z_j)= -\dfrac{(\Delta^e V^{t,\mathtt{p}}(z_j))\cdot \overline{V^{t,\mathtt{p}}}(z_j)
}{\omega^2|V^{t,\mathtt{p}}(z_j)|^2},\quad j=1,2,3,\cdots, M.\label{theorem-rho-0-expression-at-z}
\end{eqnarray}}}
\end{enumerate}}
\bigskip
\bigskip
Let us discuss a little more the proposed scheme. 
\begin{enumerate}
\item In terms of dimensionality,  we have only $3$ degrees of freedom of our measurements ($3D$) which is equal to the $3D$ degrees of freedom for the unknown mass density $\rho(\cdot)$. This makes the related inverse problem equally determined (i.e. not overdetermined). 
 
\item As described above, the amount of data needed for the reconstruction is reasonably optimal. One more appreciable advantage of this approach is its \textit{locality}. Indeed, we need to inject the inclusions only in the vicinity of the domain where we want to do the imaging.
 
\item We need to scan the domain $\Omega$ with the needed small inclusion. This issue is discussed largely in the context of acoustic and optic imaging using contrast agents (as bubbles and nanoparticles). There is a huge literature devoted to the transport of such contrast agents in the framework of drugs delivery and imaging, see for instance \cite{C-F-Q2009, F-M-S2003, Chen-Li-2015, Q2007, Z-L-L-S-W19} and the references therein. The second point is related to the numerical differentiation needed to be performed on the reconstructed fields. But, as it is known, the instability generated by this step is a mild one (of H\"{o}lder type). It is well known that the reconstruction methods proposed by using the traditional measurements (as the Dirichlet-Neumann map or the farfield map) suffer from an extremely sever instability. Indeed, the inverse problems related to these types of measurements are at least of log-type (and even of double log (or loglog)-type) for the elasticity. Therefore, with our types of measurements, we reduce the Log-type stability to a H\"{o}lder-type stability. This is an appreciable step that makes our approach competitive as soon as we can have access to the formentioned measurements.
\item The proposed schemes, based on the derived asymptotic expansions, are stable against the perturbation of the inclusion positions and the noise level as soon as the perturbations and the noise level are small enough. This issue is already clarified in another context, see \cite{alsaedi2016extraction}. A similar analysis can be extended to our present context without essential difficulties apart from the needed technicalities.
\item Finally, observe that our reconstruction schemes is based only on the $\mathtt{p}$-parts of the farfields. Reconstructing {\it{shapes}} with the $\mathtt{p}$-parts of the farfields corresponding to all the p-incident waves is known already in the literature, see \cite{Drosso-Mourad, Kar-Mourad}. However, regarding the {\it{medium}} recovery (as the mass density), it is known that using those types of measurements, one can even not prove the unique identification (based on the complex geometric optics solutions for instance). Here, using our type of measurements, it turns out to be possible to reconstruct the density (and not only to show its uniqueness).  
\end{enumerate}
\bigskip

The rest of the paper is organised as follows.  The detailed justification of the results stated in Section \ref{section-results} is given in Section \ref{section-proofs-forward-problem-j-1} for a single inclusion and then extended in Section \ref{section-proofs-forward-problem-j-M} to multiple inclusions. The justification of the scheme for reconstructing the mass density is provided in Section \ref{section-Reconstruction-mass-density}. Then, several important properties of the Navier operator and apriori estimates of the total fields, which are used in proving our results, are derived in Section \ref{Technical-lemmas}. Finally, we discuss the eventuality of reconstructing the mass density using $\mathtt{s}$-parts (instead of $\mathtt{p}$-parts) of the fields and recall few estimations of singularities of the Green's matrix in Appendix \ref{section-appendix}.

\section{  Proof of Theorem \ref{theorem}\label{section-proofs-forward-problem-j-1} for $M=1$}
The dominant elastic scattered field and farfield stated in Theorem \ref{theorem} is derived by making use of the Lippmann Schwinger equation and apriori estimates associated with the total field of the problem (\ref{internal-elastic-variable-rho}--\ref{KRC-U-s-variable-rho}). 

The Lippmann Schwinger equation associated to the problem (\ref{internal-elastic-variable-rho}--\ref{KRC-U-s-variable-rho}) is

{\begin{equation}\label{lippmann-in-R3-variable-rho-lemma}
U^t(x)=V^t(x)+\int_{D_1}{\alpha(y)\omega^2G^\omega(x,y)\cdot U^t(y)\, dy},\quad x\in\mathbb{R}^3,
\end{equation}}
where $\alpha$ is the real valued function defined on $D_1$ as
{\begin{equation}\label{def-alpha}
{\alpha(y):=\rho_1-\rho_0(y).}
\end{equation}}

\subsection{Apriori estimates}\hfill\\
In this section, we state and prove apriori estimates satisfied by the total field $U^t$. 
First, let us introduce the operators $T^\omega_{D_1}:(L^2(D_1))^3\to (L^2(D_1))^3$ and the elastic Navier operator $N_{D_1}^0:(L^2(D_1))^3\to (L^2(D_1))^3$, which are going to play a pivotal role in deriving the apriori estimates:
\begin{align} \label{def-T-omega-T-z-0}
T^\omega_{D_1} U^t(x):=\int_{D_1}{\alpha(y)G^\omega(x,y)\cdot U^t(y)\,dy}, \,
\end{align}
\begin{align} \label{def-Newton-operator}
N_{D_1}^0 U^t(x)&:=\int_{D_1}{\Gamma^0(x,y)\cdot U^t(y)\,dy}.
\end{align}
The Navier operator, $N_{D_1}^0$, is compact and self-adjoint on the separable Hilbert space $(L^2(D_1))^3$ and hence it has a complete orthonormal eigensystem that we denote $\{(\lambda_n^{D_1}, e_n^{D_1})\}_{n\in\mathbb{N}}$. Fix a natural number $n_0$ and let $\omega^2$ be the frequency  such that
\begin{equation}\label{w-wn0-choosen}
 \omega^2-\omega_{{n_0}_{(1)}}^2\sim a^{h_1} \;\mbox{ where }\; \omega_{{n_0}_{(1)}}^2:=\dfrac{1}{\rho_{1}\lambda_{n_0}^{D_1}} \mbox { and }0< h_1\leq 1.
 \end{equation}
This condition is equivalent to $
\omega^2:=\dfrac{1\pm b a^{h_1}}{\rho_{1}\lambda_{n_0}^{D_1}},$
with some constant $b$. In addition, due to the scales of $\rho_1$ and $\rho_0$, we have the property
$
|(1-(\rho_{1}-\rho_{0}(z_1))\omega^2\lambda_{n_0}^{D_1})|\sim a^{h_1},\, h_1>0$, that we will be also using in the sequel.

\begin{lemma}\label{lemma-properties}We have the following properties:
 \begin{enumerate}
 \item \label{lemma-enD-scaling-lambda-scaling-to-B} 
The  eigensystem  $\{(\lambda_n^{D_1}, e_n^{D_1})\}$ satisfies the following scaling properties;
 \begin{eqnarray}\label{scaling-prop-c-o-s}
\int_{D_1}{e_n^{D_1}(x)\,dx}={a}^{\frac{3}{2}}\int_{B}{e_n^{B}(\eta)\,d\eta} \quad\mbox{ and }\quad \lambda_n^{D_1}={a}^2\lambda_n^{B}.
 \end{eqnarray}
 \item \label{infimum-exist-sigma-greater-than-0}
Under the condition \eqref{w-wn0-choosen}, we have
\begin{itemize}
\item [(i)]\label{def-sigma}
$\sigma_1 :=\underset{n(\neq n_0)}{\inf} \{|1-\alpha(z_1)\omega^2\lambda_n^{D_1}|^2\}  >0,$
\item [(ii)] the operator $(I-\alpha(z_1)\omega^2N_{D_1}^0):(L^2(D_1))^3\to (L^2(D_1))^3$ is invertible. 
\end{itemize}
 \end{enumerate}
\end{lemma}
See section \ref{Technical-lemmas} for a detailed justification, and one can show that these properties hold over $D_j$, $j=2,3,\cdots, M$.
\medskip
\\
Using the properties mentioned in the above lemma, we derive the following proposition stating the needed estimates satisfied by the total field $U^t$.
\begin{proposition}\label{lemma-apriori-estimates}
The total field $U^t$,  (solution of (\ref{internal-elastic-variable-rho}-\ref{KRC-U-s-variable-rho})),  enjoys the following estimates:
\begin{eqnarray}\setcounter{mysubequations}{0}
     \hspace{-7.55cm}\mysubnumber& \norm{U^t}_{(L^2(D_1))^3}^2\lesssim { \dfrac{1}{|1-\alpha(z_1)\omega^2\lambda_{n_0}^{D_1}|^2} \norm{V^t}_{(L^2(D_1))^3}^2, \label{norm-ut-variiable-rho-lemma}}
      \\
   \hspace{-7.55cm}   &\nonumber
     \\
     \hspace{-7.55cm} \mysubnumber& \int_{D_1}{U^t(x)\,dx}= \int_{D_1}{W \,dx} \cdot V^t(z_1)+O(a^{4-2h}),\qquad\qquad\label{final-iny-Ut-app-variable-rho-lemma}
\end{eqnarray}
where $W:=[W_1\, W_2\, W_3]^{\top}$ and each $W_k$, $k=1,2,3$ is defined as
 \begin{eqnarray}\label{def-Wj}
  W_k:=(I-\alpha(z_1)\omega^2N_{D_1}^0)^{-1}\mathtt{e_k},
  \end{eqnarray} 
  with $\mathtt{e_k}$ denoting the standard unit vector in $\mathbb{R}^3$, for $k=1,2,3$. 
\end{proposition}

\begin{proof}
\setcounter{mysubequations}{0}
      \mysubnumber
\textbf{ Estimation of $\norm{U^t}_{(L^2(D_1))^3}$}

Using the operators $T^\omega_{D_1}$ and $N_{D_1}^0$, the Lippmann Schwinger equation \eqref{lippmann-in-R3-variable-rho-lemma}, associated with the problem (\ref{internal-elastic-variable-rho}--\ref{KRC-U-s-variable-rho}), can be expressed as
\begin{eqnarray}
(I-\omega^2\alpha(z_1)N^0_{D_1})U^t(x)=V^t(x)+\omega^2(T^\omega_{D_1}-\alpha(z_1)N_{D_1}^0)U^t(x).\label{operator-form-for-variable-rho}
\end{eqnarray}
Since $(\lambda_{n}^{D_1}, e_n^{D_1})_{n\in\mathbb{N}}$ is a complete orthonormal eigensystem of the Navier operator $N_{D_1}^0$, making use of Parseval's identity and the self-adjoint property of $N_{D_1}^0$, we can compute the $L^2$ norm of the left hand side of \eqref{operator-form-for-variable-rho} as\footnote{Here $\langle\, ,\,\rangle$ stands for the $(L^2(D_1))^3$ inner product.}
\begin{align}
\norm{(I- & \omega^2\alpha(z_1)N_{D_1}^0)U^t}^2_{(L^2(D_1))^3}
\,=\,\sum_{n}|\,\langle\,(I-\omega^2\alpha(z_1)N_{D_1}^0)U^t,\,e_n^{D_1}\,\rangle\,|^2
\nonumber\\
&=\,\sum_{n}|\,\langle\,U^t,\,(I-\omega^2\alpha(z_1)N^0_{D_1})e_n^{D_1}\,\rangle\,|^2
\,=\,\sum_{n}{|\,(1-\alpha(z_1)\omega^2\lambda_n^{D_1})\,|^2\;|\,\langle\,U^t,\,e_n^{D_1}\,\rangle\,|^2}
\nonumber\\
&
=\,|\,(1-\alpha(z_1)\omega^2\lambda_{n_0}^{D_1})\,|^2\; {\sum_{l=1}^{l_{\lambda_{n_0}}}|\,\langle\,U^t,\,e_{n_{0_l}}^{D_1}\,\rangle\,|^2 } 
+\sum_{n\neq n_0}{|\,(1-\alpha(z_1)\omega^2\lambda_n^{D_1})\,|^2\;|\,\langle\,U^t,\,e_n^{D_1}\,\rangle\,|^2}.\label{without-sigma-variable-rho}
 \end{align}

Making use of the positivity of 
$\sigma_1 \,(:=\inf_{n\neq n_0}|1-\alpha(z_1)\omega^2\lambda_n^{D_1}|^2)$, from Lemma \ref{lemma-properties}, in \eqref{without-sigma-variable-rho} and by making use of Parseval's identity, we deduce that
\begin{align}
\norm{U^t}^2_{(L^2(D_1))^3}&= {\sum_{l=1}^{l_{\lambda_{n_0}}}|\,\langle\,U^t,\,e_{n_{0_l}}^{D_1}\,\rangle\,|^2}+\sum_{n\neq n_0}|\,\langle U^t,\,e_n^{D_1}\rangle\,|^2
\nonumber\\
&\leq\left[ \dfrac{1}{|1-\alpha(z_1)\omega^2\lambda_{n_0}^{D_1}|^2}+\dfrac{1}{\sigma_1}\right] \norm{U^t-\alpha(z_1)\omega^2N_{D_1}^0(U^t)}_{(L^2(D_1))^3}^2.\label{norm-ut-parseval} 
\end{align}

Hence, to estimate $\|U^t\|_{(L^2(D_1))^3}$ the estimate of $\norm{U^t-\alpha(z_1)\omega^2N_{D_1}^0(U^t)}_{(L^2(D_1))^3}$ is needed. We have from \eqref{operator-form-for-variable-rho} that
\begin{eqnarray}
\norm{U^t-\alpha(z_1)\omega^2N_{D_1}^0(U^t)}_{(L^2(D_1))^3}&=&\norm{V^t+\omega^2(T^\omega_{D_1}-\alpha(z_1)N_{D_1}^0)(U^t)}_{(L^2(D_1))^3}\nonumber
\\
&\leq & \norm{V^t}_{(L^2(D_1))^3}+|\omega^2|\,\,\norm{(T^\omega_{D_1}-\alpha(z_1)N^0_{D_1})U^t}_{(L^2(D_1))^3}.\label{term_in_norm_u-variable-rho}
\end{eqnarray}

Consider, 
\begin{align}
\norm{(T^\omega_{D_1}-\alpha(z_1)N^0_{D_1})U^t&}_{(L^2(D_1))^3}^2\underset{\underset{\eqref{def-Newton-operator}}{\eqref{def-T-omega-T-z-0}}}{=}
 \norm{\int_{D_1}
 {\bigg(\alpha(y)G^\omega(\cdot,y)-\alpha(z_1)\Gamma^0(\cdot,y)\bigg)\cdot U^t(y)\,dy}}^2_{(L^2(D_1))^3}
 \nonumber\\
& \hspace{-1.1cm}=\int_{D_1}{\sum_{i=1}^3{|\int_{D_1}{\bigg(\alpha(y)G^\omega(x,y)-\alpha(z_1)\Gamma^0(x,y)\bigg)_i\cdot U^t(y)\,dy}|^2}dx}
\nonumber\\
& \hspace{-1.2cm}
\underset{CSI}{\leq} \norm{U^t}_{(L^2(D_1))^3}^2 \sum_{i=1}^3\sum_{k=1}^{3}\int_{D_1}\int_{D_1}{|(\alpha(y)G^\omega(x,y)-\alpha(z_1)\Gamma^0(x,y))_{ki}|^2\, dy}\,dx
\nonumber\\
&\hspace{-1.2cm}\leq 2\norm{U^t}_{(L^2(D_1))^3}^2 {\sum_{i=1}^3 }{\sum_{k=1}^{3}}\underset{D_1}{\int}\hspace{-0.07cm}\underset{D_1}{\int}\hspace{-0.07cm}\bigg(|\rho_1[G^\omega(x,y)-\Gamma^0(x,y)]_{ki}|^2+|[\rho_0(y)G^\omega(x,y)-\rho_0(z_1){\Gamma^0}(x,y)]_{ki}|^2\bigg)dy\,dx
\nonumber\\
&\hspace{-1.6cm}\underset{\underset{\eqref{G-omega-Gamma-omega-bounded-for-x,y-in-D},\eqref{def-of-rho}}{\eqref{int-D-int-D-|gamma-0-x,y-square}}}{\leq}\hspace{-0.25cm}\norm{U^t}_{(L^2(D_1))^3}^2{\sum_{i=1}^3 }{\sum_{k=1}^{3}}\left[a^2c_1^2(H_3+H_4)^2|B|^2+2C_1^2(H_3+H_4)^2a^6|B|^2+2C_2^2{C_3}\,diam(B)^2a^6\right]
\nonumber\\
&\hspace{-1.2cm}= O\left(a^2\norm{U^t}_{(L^2(D_1))^3}^2\right),\label{Nw-N0_approximation-variable-rho}
\end{align}
where {$CSI$ stands for the Cauchy-Schwarz Inequality}, $C_2$ is a Lipschitz constant associated to the function $\rho_0$ and  $C_1$ is a constant satisfying $|\rho_0(y)|\leq C_1$ in $D_1$ considering that $\rho_0$ is a $C^1$ function. Here, we also used the points that $(G^\omega-\Gamma_z^0)$ bounded entry wise by constant $H_3+H_4$ in $D_1$ and $\underset{D_1}{\int}\,\underset{D_1}{\int}|y-z|^2|\Gamma_z^0(x,y)|^2\,dy\,dx$ is bounded by $diam(B)^2C_3a^6$, see Lemma \ref{lemma-fundamental-Gree-properties-inside-D} and Remark \ref{lemma-int-int-D-|y-z|-Gamma-bounded} in the appendix for the details.

Making use of \eqref{Nw-N0_approximation-variable-rho} in \eqref{term_in_norm_u-variable-rho}, and the substitution of the resulting one in \eqref{norm-ut-parseval} will give us the following estimate:
\begin{eqnarray}
\norm{U^t}_{(L^2(D_1))^3}^2&\leq& \bigg(1+\dfrac{|1-\alpha(z_1)\omega^2\lambda_{n_0}^{D_1}|^2}{\sigma_1}\bigg) \dfrac{1}{|1-\alpha(z_1)\omega^2\lambda_{n_0}^{D_1}|^2}\left(\norm{V^t}_{(L^2(D_1))^3}^2+O(|\omega|^2a^2\norm{U^t}_{(L^2(D_1))^3}^2)\right).\nonumber
 \end{eqnarray}
Hence, using the fact that $1-\frac{a^2}{|1-\alpha(z_1)\omega^2\lambda_{n_0}^{D_1}|^2}$ is uniformly bounded below for $h_1<1$, the estimate \eqref{norm-ut-variiable-rho-lemma} follows.
\bigskip
\\
\mysubnumber\textbf{ Estimation of $\int_{D_1}U^t(y)\,dy$}
\par Applying the first order Taylor series representation to $V^t$ around the point $z_1$ (observe that $V^t$ is $C^3$ as $\rho_0$ is taken to be of class $C^1$), the integral equation \eqref{operator-form-for-variable-rho} for $x\in D_1$ becomes
\begin{eqnarray}
(I-\alpha(z_1)\omega^2N_{D_1}^0)U^t(x)=V^t(z_1)+\hspace{-0.05cm}\int_{0}^{1}{\hspace{-0.25cm} \nabla V^t(z_1+t(x-z_1))\cdot (x-z_1)\,dt}+
\omega^2(T^\omega_{D_1}-\alpha(z_1)N_{D_1}^0)U^t(x) .\label{before-dot-with-W-variable-rho}
\end{eqnarray}
{Knowing that $(I-\alpha(z_1)\omega^2N_{D_1}^0)$ is invertible from Lemma \ref{lemma-properties}, $W_j$ defined in \eqref{def-Wj} is well defined. Hence by taking the dot product with $W$ and integrating  over $D_1$, and using the property that the operator $N_{D_1}^0$  is self-adjoint,   \eqref{before-dot-with-W-variable-rho} becomes }
\begin{align}
\int_{D_1}{U^t(x)} \, dx\,=&
\int_{D_1}{\hspace{-0.2cm}W\cdot V^t(z_1)\,dx}+\mathcal{L}_1+
\mathcal{L}_2,
 \label{int-ut-app-variable-rho}
\end{align}
where
\begin{align}\label{int-ut-app-variable-rho-def-u1-u2}
{\mathcal{L}_1}:=&{\int_{D_1}{\hspace{-0.1cm}W\cdot \hspace{-0.1cm}\int_{0}^{1}{\hspace{-0.26cm} \nabla V^t(z_1+t(x-z_1))\cdot (x-z_1)\,dt}\, dx}}\quad\mbox{ and }\quad
{ \mathcal{L}_2}:=\omega^2{\int_{D_1}\hspace{-0.05cm}W\cdot (T^\omega_{D_1}-\alpha(z_1) N_{D_1}^0)U^t(x)\, dx}.
\end{align}
In order to estimate $\mathcal{L}_1$ and $\mathcal{L}_2$, first  we derive the a priori estimation of $\int_{D_1}{W_i(x)}dx$ and $\norm{W_i}_{(L^2(D_1))^3}
$ for $i=1,2,3$. In this direction, first observe that $\int_{D_1}{e_n^{D_1}}(x)\,dx$ can be expressed as  
\begin{align}
\int_{D_1}\hspace{-0.23cm}{e_n^{D_1}(x)}dx=\int_{D_1}\hspace{-0.23cm}{I\,\cdot e_n^{D_1}(x)dx}
\hspace{-0.2cm}\underset{\eqref{def-Wj}}{=}
\hspace{-0.2cm}\int_{D_1}{\hspace{-0.25cm} (I-\alpha(z_1)\omega^2 N_{D_1}^0)W\cdot e_n^{D_1}(x) dx}
= (1-\alpha(z_1)\omega^2 \lambda_n^{D_1})\int_{D_1}\hspace{-0.25cm}{W\cdot e_n^{D_1}(x) dx}\nonumber,
\end{align}
 and hence we have,
\begin{eqnarray}\label{INP-WI-EN}
\dfrac{1}{(1-\alpha(z_1)\omega^2 \lambda_n^{D_1})}\int_{D_1}{\mathtt{e_i}\cdot e_n^{D_1}(x)\,dx}=\int_{D_1}{W_i\cdot e_n^{D_1}(x)\, dx}=\,\langle\,W_i,\,{e_n^{D_1}}\,\rangle, \, i=1,2,3.
\end{eqnarray}

 Now, we use of Parseval's identity to estimate $\int_{D_1}{W_i\, dx}$ and $\norm{W_i}_{(L^2(D_1))^3}$.
\begin{itemize}
\item[1)] First, we derive
\begin{align}
\vspace{-1cm}
 \int_{D_1}{W_i\, dx}\,=&\sum_{n\in \mathbb{N}}{\langle\,W_i,e_n^{D_1}\,\rangle\,\int_{D_1}{e_n^{D_1}(x)\,dx}}
\underset{\eqref{INP-WI-EN}}{=}\sum_{n}{\dfrac{1}{(1-\alpha(z_1)\omega^2 \lambda_n^{D_1})}\left(\int_{D_1}{\mathtt{e_i}\cdot {e_n^{D_1}}(x)\,dx}\right)\int_{D_1}{e_n^{D_1}(x)\,dx}}\nonumber\\
=\,&\dfrac{1}{(1-\alpha(z_1)\omega^2 \lambda_{n_0}^{D_1})} {\sum_{l=1}^{l_{\lambda_{n_0}}}\left(\mathtt{e_i}\cdot\int_{D_1}{{e_{n_{0_l}}^{D_1}}(x) \,dx}\right)\int_{D_1}{e_{n_{{0}_l}}^D(x)\,dx}}+O(a^3),\, i=1,2,3,\label{int-W1-app-variable-density}
\end{align}
since, $ \int_{D_1}W_{im}\,dx=
\int_{D_1}W_{i}\,dx\cdot \mathtt{e_m}$ and the estimate (see Lemma \ref{lemma-series-convergent-for-n-not-n0-single-paricle} in the appendix) 
$$\sum_{n\neq {n_0}}\dfrac{1}{(1-\alpha(z_1)\omega^2 \lambda_n^{D_1})}\left(\mathtt{e_i}\cdot\int_{D_1}{{e_n^{D_1}}(x)\, dx}\right)\int_{D_1}{e_n^{D_1}(x)\,dx}\cdot \mathtt{e_m}=O( a^3),\, m=1,2,3.$$
\medskip
 
\item [2)] Secondly, we obtain
\begin{align}
\hspace{1.2cm}\norm{W_i}_{(L^2(D_1))^3}^2=&\sum_{n}{|\,\langle\,W_i,\,e_n^{D_1}\,\rangle\,|}^2\underset{\eqref{INP-WI-EN}}{=}\sum_{n}{\dfrac{1}{|1-\alpha(z_1)\omega^2 \lambda_{n}^{D_1}|^2}\;|\,\langle\,\mathtt{e_i},\,e_n^{D_1}\,\rangle\,|^2}
\nonumber\\
=&\dfrac{1}{|1-\alpha(z_1)\omega^2 \lambda_{n_0}^{D_1}|^2}\; {\sum_{l=1}^{l_{\lambda_{n_0}}}|\,\langle\,\mathtt{e_i},\,e_{n_{0_l}}^{D_1}\,\rangle\,|^2}+\sum_{n\neq n_0}{\dfrac{1}{|1-\alpha(z_1)\omega^2 \lambda_{n}^{D_1}|^2}\;|\,\langle\,\mathtt{e_i},\,e_n^{D_1}\,\rangle\,|^2}
\nonumber\\
\underset{CSI}{\lesssim} & \dfrac{1}{|1-\alpha(z_1)\omega^2 \lambda_{n_0}^{D_1}|^2}\,\norm{\mathtt{e_i}}^2_{L^2(D_1)^3}+\sum_{n\neq n_0}{\dfrac{1}{|1-\alpha(z_1)\omega^2 \lambda_{n}^{D_1}|^2}\;|\,\langle\,\mathtt{e_i},e_n^{D_1}\,\rangle\,|^2},
\nonumber
\end{align}
and hence,
\begin{eqnarray}
 \norm{W_i}^2_{(L^2(D_1))^3}\underset{\underset{ \text{Lemma} \,\ref{lemma-series-convergent-for-n-not-n0-single-paricle}}{\eqref{w-wn0-choosen},\eqref{def-sigma}}}{=}& O(a^{3-2h_1})+O(a^3)\implies \norm{W_i}_{(L^2(D_1))^3}=O(a^{\frac{3}{2}-h_1})\mbox{ for } i=1,2,3.\label{norm-W-variable-rho}
\end{eqnarray}
\end{itemize}

Now, with the help of the estimate \eqref{norm-W-variable-rho}, we can estimate $\mathcal{L}_1$ and $\mathcal{L}_2$, defined in \eqref{int-ut-app-variable-rho-def-u1-u2} as follows.

\begin{itemize}
\item[$\star$]  Estimate of $\mathcal{L}_1$:
\begin{align}
|\mathcal{L}_1|^2=\;\,&\sum_{i=1}^3|\int_{D_1}{W_i\cdot\left( \int_{0}^{1}{ \nabla V^t(z_1+t(x-z_1))\cdot (x-z_1)\,dt} \right)\,dx}\,|^2
\nonumber\\
\underset{CSI}{\leq}& \sum_{i=1}^3\norm{W_i}_{(L^2(D_1))^3}^2\norm{\int_{0}^{1}{ \nabla V^t(z_1+t(\cdot-z_1))\cdot (\cdot-z_1)\,dt}}_{(L^2(D_1))^3}^2
\nonumber\\
\leq\;\,& [\sum_{i=1}^3\norm{W_i}_{(L^2(D_1))^3}^2]\;\sum_{k=1}^3\int_{D_1}{\left(a\,diam(B)\,\int_{0}^{1}{| \nabla V_k^t(z_1+t(x-z_1))|\,dt}\right)^2 \,dx}
\nonumber\\
|\mathcal{L}_1|{\underset{\eqref{norm-W-variable-rho}}{=}} &O(a^{4-h_1}). 
\label{I_1-app-variable-rho}
\end{align}

\item[$\star$]  Estimate of $\mathcal{L}_2$:
\begin{align}
|\mathcal{L}_2|^2&=|\omega^2|^2\sum_{i=1}^3|\int_{D_1}W_i\cdot (T^\omega_{D_1}-\alpha(z_1)N_{D_1}^0)U^t(x)\, dx|^2
\nonumber\\
&\underset{\underset{\eqref{Nw-N0_approximation-variable-rho}}{CSI}}{\leq}  |\omega^2|^2\norm{U^t}_{(L^2(D_1))^3}^2\sum_{i=1}^3\norm{W_i}^2_{(L^2(D_1))^3}3^2\bigg(2a^2c_1^2(H_3+H_4)^2|B|^2+4C_1^2(H_3+H_4)^2a^6|B|^2+4C_2^2\,diam(B)^2a^6{C_3}
\bigg)
\nonumber\\
&\underset{\eqref{norm-W-variable-rho}}{=} O(|\omega^2|^2\norm{U^t}_{(L^2(D_1))^3}^2 a^{5-2h})
\nonumber\\
|\mathcal{L}_2|&\underset{\eqref{norm-ut-variiable-rho-lemma}}{=}O\left(\dfrac{\omega^2}{|1-\alpha(z_1)\omega^2\lambda_{n_0}^{D_1}|}\norm{V^t}_{(L^2(D_1))^3} a^{\frac{5}{2}-h_1}\right)\underset{\eqref{w-wn0-choosen}}{=}O(a^{\frac{5}{2}-2h_1}{\norm{V^t}_{(L^2(D_1))^3}})
=O(a^{4-2h_1}).\label{I2-approximation-variable-rho}
\end{align}
\end{itemize}
Finally, substitution of \eqref{I_1-app-variable-rho} and \eqref{I2-approximation-variable-rho} in \eqref{int-ut-app-variable-rho} gives the required result \eqref{final-iny-Ut-app-variable-rho-lemma}.
\end{proof}

\subsection{  {Estimation of scattered field}}\hfill\\
In this section, we show the {proof of \eqref{theorem-scatteredfield-mp} stated in} Theorem \ref{theorem} using the Lippmann Schwinger equation and the a priori estimates mentioned in Proposition \ref{lemma-apriori-estimates}.

From the Lippmann Schwinger equation \eqref{lippmann-in-R3-variable-rho-lemma} of the problem (\ref{internal-elastic-variable-rho}--\ref{KRC-U-s-variable-rho}) and making use of the Taylor series expansion of $G^\omega(x,y)$ about $y$ near $z_1$, the scattered field $U^s(x)$ for $x$ away from $D_1$ can be expressed as,
\begin{align}
U^s(x)=&V^s(x)+\omega^2\int_{D_1}{\ \alpha(y)G^\omega(x,y)\cdot U^t(y)\,dy}
\nonumber\\
=&V^s(x)+\omega^2\int_{D_1}{ \alpha(z_1)G^\omega(x,z_1)\cdot U^t(y)\,dy}+\omega^2\mathcal{A}_1,
\label{final-step-after-taylor-series-variable-rho}
\end{align}
with
\begin{align}
\mathcal{A}_1(k)&:=\int_{D_1}\int_{0}^{1}{\hspace{-0.2cm} \nabla_y\left[\alpha(z_1+t(y-z_1))G_k^\omega(x,z_1+t(y-z_1))\right]\cdot (y-z_1)dt}\cdot U^t(y)dy
= \mathcal{Q}_1(k) + \mathcal{Q}_2(k),\;\, k=1,2,3,\nonumber
\end{align}
where $\mathcal{Q}_l(k)$, $l=1,2$ represents the entries of column vectors $\mathcal{Q}_l$, $l=1,2$, given by 
\begin{align}
\mathcal{Q}_1(k)&:=
\int_{D_1}{\left(\int_{0}^{1}{ [\alpha(z_1+t(y-z_1))\nabla_y G_k^\omega(x,z_1+t(y-z_1))] \cdot (y-z_1)\,dt}\right) \cdot U^t(y)\,dy},
 \label{def-S1} \\
\mbox{and  }\;\mathcal{Q}_2(k)&:=
\int_{D_1}{\left(\int_{0}^{1}{ [G_k^\omega(x,z_1+t(y-z_1))\nabla_y\alpha(z_1+t(y-z_1))]\cdot (y-z_1)dt}\right)\cdot U^t(y)\,dy},
 \label{def-S2} \; \mbox{ for } k=1,2,3.
\end{align}
Knowing that $|\alpha(y)|=|\rho_1-\rho_0(y)|=O(a^{-2})$, $ \int_{0}^{1}|\nabla_yG^{\omega}(x,z_1+t(y-z_1))_{ij}|dt$ and $G^\omega(x,y)$ are  bounded by constants for well seperated $x$ and $y$, (see \eqref{def-of-rho},\eqref{def-alpha} and Remark \ref{lemma-grad-G-for-x-away-from-omega}), we can approximate $\mathcal{A}$ as;
\begin{eqnarray}
|\mathcal{A}_1|^2&\leq& 2(|\mathcal{Q}_1|^2+|\mathcal{Q}_2|^2)\nonumber\\
& \leq&  2 \left[\sum_{i=1}^3|\int_{D_1}{\left(\int_{0}^{1}{ [\alpha(z_1+t(y-z_1))\nabla_y G_i^\omega(x,z_1+t(y-z_1))] \cdot (y-z_1)\,dt}\right) \cdot U^t(y)\,dy}|^2\right.
\nonumber\\
&& +\left.\sum_{i=1}^3|\int_{D_1}{\left(\int_{0}^{1}{ [G_i^\omega(x,z_1+t(y-z_1))\nabla\alpha(z_1+t(y-z_1))]\cdot (y-z_1)\,dt}\right)\cdot U^t(y)\,dy}|^2\right]
\nonumber\\
&=&O(a\norm{U^t}^2_{(L^2(D_1))^3}+a^{5}\norm{U^t}^2_{(L^2(D_1))^3})\nonumber\\
\implies |\mathcal{A}_1|&=& O(a^{\frac{1}{2}}\norm{U^t}_{(L^2(D_1))^3}).
\label{variable-density-final-B1-app}
\end{eqnarray}
Making use of the behavior of $\mathcal{A}_1$ from \eqref{variable-density-final-B1-app} and the a priori estimates of $\norm{U^t}_{(L^2(D_1))^3}$ and $\int_{D_1}{U^t(y)dy}$, stated in \eqref{norm-ut-variiable-rho-lemma} and \eqref{final-iny-Ut-app-variable-rho-lemma} of Proposition \ref{lemma-apriori-estimates}, we can rewrite \eqref{final-step-after-taylor-series-variable-rho} as 
\begin{eqnarray}
U^s(x)\hspace{-0.35cm}&=&\hspace{-0.35cm}V^s(x)+\omega^2\alpha(z_1)G^\omega(x,z_1)\cdot\left( \int_{D_1}{\hspace{-0.23cm}W \,dx\cdot V^t(z_1)}+O(a^{4-2h})\right)+O\left( \dfrac{a^{\frac{1}{2}}\omega^2}{|1-\alpha(z_1)\omega^2\lambda_{n_0}^{D_1}|} {\norm{V^t}}_{(L^2(D_1))^3}\right)
\nonumber\\
&\underset{\eqref{w-wn0-choosen}}{=}&\hspace{-0.35cm} V^s(x)+\omega^2\alpha(z_1)G^\omega(x,z_1)\cdot\int_{D_1}{W\,dx}\cdot V^t(z_1)+O(a^{4-2h_1}|\alpha(z_1)G^\omega(x,z_1)|)+\omega^2O(a^{\frac{1}{2}-h_1}a^{\frac{3}{2}}).\label{us-app-step-1-variable-rho}
\end{eqnarray}
In addition, plugging the expression \eqref{int-W1-app-variable-density} and the scaling property \eqref{scaling-prop-c-o-s} related to eigenvectors in the above approximation \eqref{us-app-step-1-variable-rho}, we deduce the following approximation
\begin{align}
U^s(x)=&V^s(x)+\frac{\omega^2}{(1-\alpha(z_1)\omega^2\lambda_{n_0}^{D_1})}\alpha(z_1)G^\omega(x,z_1)\cdot E_{n_0}^{D_1}\cdot  V^t(z )+\omega^2V^t(z_1)O(a^3|\alpha(z_1)G^\omega(x,z_1)|)
\nonumber\\
& +O(a^{4-2h_1}|\alpha(z_1)G^\omega(x,z_1)|)+\omega^2O(a^{2-h_1}), \nonumber
\end{align}
where the matrix $E_{n_0}^{D_1}:={\sum_{l=1}^{l_{\lambda_{n_0}}}\int_{D_1}{e_{n_{0_l}}^{D_1}}(x)\,dx\otimes\int_{D_1}e_{n_{0_l}}^{D_1}(x)\,dx}$.
Again, using the facts that $\omega_{n_0}^{-2}= \rho_1\lambda_{n_0}^{D_1}$, $|\alpha(y)|=|\rho_1-\rho_0(y)|=O(a^{-2})$ and $G^\omega(x,y)$ is bounded by constant for $y\in D$ and $x$ away from $\Omega$ (see \eqref{w-wn0-choosen} and \eqref{def-of-rho}) 
and Remark \ref{lemma-grad-G-for-x-away-from-omega},
 we get the following approximation of the scattered field $U^s(x)$, which is nothing but the required estimation \eqref{theorem-scatteredfield-mp}:
\begin{align}
U^s(x,\theta)=V^s(x,\theta)+\frac{\omega^2\omega_{n_0}^2}{(\omega_{n_0}^2-\omega^2)} \rho_1G^\omega(x,z_1)\cdot  E_{n_0}^{D_1}\cdot V^t(z_1,\theta)+{O(a)}{ +O\left(a^{\min\{1, 2-2h_1\}}\right)} , \label{scatteredfield-variable-mass-density-Proof}
\end{align}
uniformly for the incident direction $\theta\in\mathbb{S}^2$.
\subsection{{End of the proof of Theorem \ref{theorem}} (Estimation of farfields)}\label{section-mass-density-reconstruction}
\hfill\\
{In this section, firstly we show the proof of \eqref{theorem-p-part-farfield-mp} in Theorem \ref{theorem} using Lippmann Schwinger equation \eqref{lippmann-in-R3-variable-rho-lemma}, apriori estimates mentioned in Proposition \ref{lemma-apriori-estimates} and mixed reciprocity relation mentioned in Lemma \ref{lemma-mixed-reciprocity-elastic}. Secondly, by following the same arguments, we prove the result \eqref{theorem-farfield-s-part-mp} in Theorem \ref{theorem}.
\subsubsection{Proof of \eqref{theorem-p-part-farfield-mp}}\hfill\\
Consider the Lippmann Schwinger equation \eqref{lippmann-in-R3-variable-rho-lemma},
\begin{equation}\label{scattered-field-L-E-far-est}
U^s(x)=V^s(x)+\int_{D_1}{\alpha(y)\omega^2G^\omega(x,y)\cdot U^t(y)\, dy}.
\end{equation}
Using asymptotic expansions of $U^s(\cdot)$ and $G^\omega(\cdot,y)$ i.e., \eqref{Us-assymptotic-expansion} and \eqref{assymptotic-expansion-Green} respectively in \eqref{scattered-field-L-E-far-est} we get
\begin{align}
    &U^\infty_{\mathtt{s}}(\hat{x},\theta)=V^\infty_{\mathtt{s}}(\hat{x},\theta)+\int_{D_1}\alpha(y)\omega^2G_\mathtt{s}^{\omega,\infty}(\hat{x},y)\cdot U^t(y,\theta)dy\label{s-part-Ut-infty-IE}
\\
\mbox{and }\,&U^\infty_{\mathtt{p}}(\hat{x},\theta)=V^\infty_{\mathtt{p}}(\hat{x},\theta)+\int_{D_1}\alpha(y)\omega^2G_\mathtt{p}^{\omega,\infty}(\hat{x},y)\cdot U^t(y,\theta)dy\label{p-part-Ut-inft-IE}
    \end{align}
make use of Taylor series expansion of $\alpha(y)G^{\omega,\infty}_k(y,-\hat{x})$, $k=1,2,3$, about $y$ near $z$ in \eqref{p-part-Ut-inft-IE}, we get
\begin{align}
   U^\infty_{\mathtt{p}}(\hat{x},\theta)=& V^\infty_{\mathtt{p}}(\hat{x},\theta)+\int_{D_1}\omega^2\alpha(z_1)G_\mathtt{p}^{\omega,\infty}(\hat{x},z)\cdot U^t(y,\theta)\,dy+\mathtt{Ep}\label{p-part-farfield-expression} 
\end{align}
where 
\begin{align}
    \mathtt{Ep}(k):={\omega^2}\int_{D_1}\int_{0}^1\nabla_y \left(\alpha(z_1+t(y-z_1))(G^{\omega,\infty}_{\mathtt{p}}(\hat{x},z_1+t(y-z_1)))_k\right)\cdot (y-z_1) dt\cdot U^t(y,\theta) dy,\; k=1,2,3
\end{align}
\begin{align}
|\mathtt{Ep}|=&\left(\sum_{k=1}^3|{\omega^2}\int_{D_1}\int_{0}^1\nabla_y \left(\alpha(z_1+t(y-z_1))(G^{\omega,\infty}_{\mathtt{p}}(\hat{x},z_1+t(y-z_1)))_k\right)\cdot (y-z_1) dt\cdot U^t(y,\theta) dy|^2\right)^{\frac{1}{2}}
        \nonumber\\
        \leq &\left(\sum_{k=1}^3|{\omega^2}|\norm{\int_{0}^1\nabla\left(\alpha(z_1+t(\cdot-z_1))(G^{\omega,\infty}_{\mathtt{p}}(\hat{x},z_1+t(\cdot-z_1)))_k\right)\cdot (\cdot-z_1) dt}_{(L^2(D_1))^3}^2\norm{ U^t}_{(L^2(D_1))^3}^2 \right)^{\frac{1}{2}}
        \nonumber\\
        \leq &\norm{ U^t}_{(L^2(D_1))^3}\left(\sum_{k=1}^3|{\omega^2}|\left[\norm{\int_{0}^1[\alpha(z_1+t(\cdot-z_1))\nabla(G^{\omega,\infty}_{\mathtt{p}}(\hat{x},z_1+t(\cdot-z_1)))_k]\cdot (\cdot-z_1)dt}_{(L^2(D_1))^3}\right.\right.
    \nonumber\\
&\left.\left.+\norm{\int_{0}^1[(G^{\omega,\infty}_{\mathtt{p}}(\hat{x},z_1+t(\cdot-z_1)))_k\otimes\nabla \alpha(z_1+t(\cdot-z_1))]\cdot (\cdot-z_1) dt}_{(L^2(D_1))^3}
\right]^2 \right)^{\frac{1}{2}}
\nonumber\\
= &O(\norm{ U^t}_{(L^2(D_1))^3}(a^{\frac{3}{2}}a^{-2}a^{1}+a^{\frac{3}{2}}a^1))
\nonumber\\
\underset{\eqref{norm-ut-variiable-rho-lemma}}{=}&O(\frac{1}{|1-\alpha(z_1)\omega^2\lambda_{n_0}^{D_1}|}\norm{V^t}_{(L^2(D_1))^3}a^{\frac{1}{2}})=O(a^{2-h_1}).\label{E-r-estimation-in-terms-of-a}
\end{align}
 By using \eqref{E-r-estimation-in-terms-of-a} and mixed reciprocity relation \eqref{lemma-mixed-reciprocity-elastic-p} in \eqref{p-part-farfield-expression} we get 
\begin{align}
   U^\infty_{\mathtt{p}}(\hat{x},\theta){=}& V^\infty_{\mathtt{p}}(\hat{x},\theta)+\omega^2\alpha(z_1)G_\mathtt{p}^{\omega,\infty}(\hat{x},z)\cdot \int_{D_1}U^t(y,\theta)\,dy+O(a^{2-h_1})
\nonumber\\
\underset{\eqref{lemma-mixed-reciprocity-elastic-p}}{=}& V^\infty_{\mathtt{p}}(\hat{x},\theta)+\frac{\omega^2}{4\pi(\lambda+2\mu)\beta_1}\alpha(z_1)(\hat{x}\cdot \int_{D_1}U^t(y,\theta)dy)V^{t,\mathtt{p}}(z_1,-\hat{x})+O(a^{2-h_1})
\end{align}
\begin{align}
   & U^\infty_{\mathtt{p}}(\hat{x},\theta) \hspace{-0.05cm} \underset{\eqref{final-iny-Ut-app-variable-rho-lemma}}{=}\hspace{-0.07cm}V^\infty_{\mathtt{p}}(\hat{x},\theta)+\frac{\omega^2}{4\pi(\lambda+2\mu)\beta_1}\left(\hat{x}\cdot \left( \int_{D_1}{W \,dx} \cdot V^t(z_1,\theta)\right)+O(a^{4-2h})\right)\alpha(z_1)V^{t,\mathtt{p}}(z_1,-\hat{x})+O(a^{2-h_1})
    \nonumber\\
    &\underset{\eqref{int-W1-app-variable-density}}{=}V^\infty_{\mathtt{p}}(\hat{x},\theta)\hspace{-0.07cm}+\hspace{-0.07cm}\frac{\omega^2\alpha(z_1)}{4\pi(\lambda+2\mu)\beta_1}\,\hat{x}\hspace{-0.07cm}\cdot \hspace{-0.1cm} \left(\hspace{-0.03cm}\left[\hspace{-0.07cm}\frac{1}{(1-\alpha(z_1)\omega^2\lambda_{n_0}^{D_1})}E_{n_0}^{D_1}\hspace{-0.07cm}+\hspace{-0.07cm}O(a^3)\right]\hspace{-0.1cm} \cdot \hspace{-0.1cm}V^t(z_1,\theta)\hspace{-0.1cm}\right)V^{t,\mathtt{p}}(z_1,-\hat{x})
\hspace{-0.07cm}+\hspace{-0.07cm}O(a^{2-2h})\hspace{-0.07cm}\hspace{-0.07cm}+\hspace{-0.07cm}O(a^{2-h_1})
    \nonumber\\
     &{=}V^\infty_{\mathtt{p}}(\hat{x},\theta)+\frac{\omega^2}{4\pi(\lambda+2\mu)\beta_1}\frac{\alpha(z_1)}{(1-\alpha(z_1)\omega^2\lambda_{n_0}^{D_1})}\left(\hat{x}\cdot \left( E_{n_0}^{D_1} \cdot V^t(z_1,\theta)\right)\right)V^{t,\mathtt{p}}(z_1,-\hat{x})+O(a^1)
    +O(a^{2-2h_1})
     \label{farfield-variable-mass-density-proof}
\end{align}
which is nothing but our required farfield approximation \eqref{theorem-p-part-farfield-mp} for $j=1$, by using \eqref{def-of-rho}, \eqref{def-alpha} and \eqref{w-wn0-choosen}, uniformly in all the directions of $\hat{x},\theta\in\mathbb{S}^2$.
\\
 One can observe that for $\mathtt{p}$-incident wave $U_{\mathtt{p}}^i$, and making use of \eqref{def-of-rho}, \eqref{def-alpha} and \eqref{w-wn0-choosen}, the equation \eqref{farfield-variable-mass-density-proof} becomes
\begin{align}
  U^{\infty,\mathtt{p}}_{\mathtt{p}}(\hat{x},\theta)   {=}&V^{\infty,\mathtt{p}}_{\mathtt{p}}(\hat{x},\theta)\hspace{-0.05cm}+\hspace{-0.05cm}\frac{\rho_1}{4\pi(\lambda+2\mu)\beta_1}\frac{\omega^2\omega_{n_0}^2}{(\omega_{n_0}^2-\omega^2)}\left(\hspace{-0.05cm}\hat{x}\cdot \hspace{-0.05cm}\left( E_{n_0}^{D_1} \hspace{-0.05cm}\cdot V^{t,\mathtt{p}}(z,\theta)\right)\hspace{-0.1cm}\right)V^{t,\mathtt{p}}(z_1,-\hat{x})\hspace{-0.05cm}+\hspace{-0.05cm}O(a^{\min\{1,2-2h_1\}}).\label{farfield-variable-mass-density-p-part-incident}
\end{align}
uniformly for the incident and observation directions $\theta,\hat{x}\in\mathbb{S}^2$ respectively.
 }
 
\subsubsection{Proof of \eqref{theorem-farfield-s-part-mp}}
\hfill\\
 {
To derive the farfield estimation \eqref{theorem-farfield-s-part-mp}, first we consider the equation \eqref{s-part-Ut-infty-IE} and make use of Taylors series expansion of $\alpha(y)G_\mathtt{s}^{\omega,\infty}(x,y)$ about $y$ near $z_1$ to get 
\begin{eqnarray}
   U^\infty_{\mathtt{s}}(\hat{x},\theta)=V^\infty_{\mathtt{s}}(\hat{x},\theta)+\int_{D_1}\alpha(z_1)\omega^2G_\mathtt{s}^{\omega,\infty}(\hat{x},z)\cdot U^t(y)dy  + \mathtt{Es}\label{s-part-Ut-IE-after-Taylor-exp}
\end{eqnarray}
where 
\begin{eqnarray}
    \mathtt{Es}(k):= {\omega^2}\int_{D_1}\int_{0}^1\nabla_y \left(\alpha(z_1+t(y-z_1))(G^{\omega,\infty}_{\mathtt{s}}(\hat{x},z_1+t(y-z_1)))_k\right)\cdot (y-z_1) dt\cdot U^t(y,\theta) dy,\; k=1,2,3
\end{eqnarray}
\begin{align}
        |\mathtt{Es}| \leq &\left(\sum_{k=1}^3|{\omega^2}|\norm{\int_{0}^1\nabla\left(\alpha(z_1+t(\cdot-z_1))(G^{\omega,\infty}_{\mathtt{s}}(\hat{x},z_1+t(\cdot-z_1)))_k\right)\cdot (\cdot-z_1) dt}_{(L^2(D_1))^3}^2\norm{ U^t}_{(L^2(D_1))^3}^2 \right)^{\frac{1}{2}}
        \nonumber\\
        \leq &\norm{ U^t}_{(L^2(D_1))^3}\left(\sum_{k=1}^3|{\omega^2}|\left[\norm{\int_{0}^1([\alpha(z_1+t(\cdot-z_1))\nabla(G^{\omega,\infty}_{\mathtt{p}}(\hat{x},z_1+t(\cdot-z_1)))_k]\cdot (\cdot-z_1)dt}_{(L^2(D_1))^3}\right.\right.
    \nonumber\\
&\left.\left.+\norm{\int_{0}^1[(G^{\omega,\infty}_{\mathtt{p}}(\hat{x},z_1+t(\cdot-z_1)))_k\otimes\nabla \alpha(z_1+t(\cdot-z_1))]\cdot (\cdot-z_1) dt}_{(L^2(D_1))^3}
\right]^2 \right)^{\frac{1}{2}}
\nonumber\\
&\underset{\eqref{norm-ut-variiable-rho-lemma}}{=}O(\frac{1}{|1-\alpha(z_1)\omega^2\lambda_{n_0}^{D_1}|}\norm{V^t}_{(L^2(D_1))^3}a^{\frac{1}{2}})=O(a^{2-h_1}).\label{E-s-estimation-in-terms-of-a}
\end{align}
Substituting \eqref{E-s-estimation-in-terms-of-a} in \eqref{s-part-Ut-IE-after-Taylor-exp} and using mixed reciprocity relation \eqref{lemma-mixed-reciprocity-relation-s-incident-wave} we get
\begin{align}
   U^\infty_{\mathtt{s}}(\hat{x},\theta)=&V^\infty_{\mathtt{s}}(\hat{x},\theta)+\int_{D_1}\alpha(z_1)\omega^2G_\mathtt{s}^{\omega,\infty}(\hat{x},z_1)\cdot U^t(y)dy  + O(a^{2-h_1})
   \nonumber\\
\underset{\eqref{lemma-mixed-reciprocity-relation-s-incident-wave}}{=}&  V^\infty_{\mathtt{s}}(\hat{x},\theta)+ \alpha(z_1)\omega^2\left[ \frac{(\int_{D_1}U^t(y)dy\cdot\hat{x}^{\perp_h})}{4\pi\mu\beta_{2_h}} V^{t,\mathtt{s}_h}(z_1,-\hat{x})+\frac{(\int_{D_1}U^t(y)dy\cdot\hat{x}^{\perp_v})}{4\pi\mu\beta_{2_v}} V^{t,\mathtt{s}_v}(z_1,-\hat{x})\right].\label{s-part-Ut-after-int-Ut-expansion}
\end{align}
Make use of apriori estimate \eqref{final-iny-Ut-app-variable-rho-lemma}, definitions \eqref{def-alpha} and \eqref{omega-Minnaert} in \eqref{s-part-Ut-after-int-Ut-expansion}, we get the desired farfield estimation \eqref{theorem-farfield-s-part-mp}, for $j=1$,
\begin{align}
     U^\infty_{\mathtt{s}}(\hat{x},\theta)=&V^\infty_{\mathtt{s}}(\hat{x},\theta)+\hspace{-0.05cm}\frac{\rho_1}{4\pi\mu}\frac{\omega^2\omega_{n_0}^2}{\omega_{n_0}^2-\omega^2} \left[\frac{(E_{n_0}^{D_1}\hspace{-0.05cm}\cdot\hspace{-0.05cm} V^{t}(z,\theta))\hspace{-0.05cm}\cdot \hat{x}^{\perp_h}}{\beta_{2_h}} V^{t,\mathtt{s}_h}(z_1,-\hat{x})\hspace{-0.05cm}+\hspace{-0.05cm}\frac{(E_{n_0}^{D_1}\cdot V^{t}(z,\theta))\cdot \hat{x}^{\perp_v}}{\beta_{2_v}} V^{t,\mathtt{s}_v}(z_1,-\hat{x})\right] \nonumber\\
     &+ O(a)+O(a^{2-2h_1}).\label{s-part-farfield-expansion-in-combination-withs-h-and-s-v}
\end{align}
For the incident wave $U^{i}_{\mathtt{s}_h}$, equation \eqref{s-part-farfield-expansion-in-combination-withs-h-and-s-v} becomes
\begin{align}
U^{\infty,\mathtt{s}_h}_{\mathtt{s}}(\hat{x},\theta)\hspace{-0.07cm}=&V^{\infty,\mathtt{s}_h}_{\mathtt{s}}(\hat{x},\theta)\hspace{-0.07cm}+\hspace{-0.1cm}\frac{\rho_1}{4\pi\mu\beta_{2}}\frac{\omega^2\omega_{n_0}^2}{\omega_{n_0}^2\hspace{-0.18cm}-\omega^2} \hspace{-0.08cm}\left[(E_{n_0}^{D_1}\hspace{-0.1cm}\cdot\hspace{-0.05cm} V^{t,\mathtt{s}_h}(z_1,\theta))\hspace{-0.05cm}\cdot\hspace{-0.05cm} \hat{x}^{\perp_h}V^{t,\mathtt{s}_h}(z_1,-\hat{x})\hspace{-0.07cm}+\hspace{-0.07cm}(E_{n_0}^{D_1}\hspace{-0.1cm}\cdot\hspace{-0.05cm} V^{t,\mathtt{s}_h}(z_1,\theta))\hspace{-0.05cm}\cdot\hspace{-0.05cm} \hat{x}^{\perp_v} V^{t,\mathtt{s}_v}(z_1,-\hat{x})\right] \nonumber\\
     &+ O(\max\{a,a^{2-2h_1})\label{s-part-farfield-exp-in-h-incident}
\end{align}
and for $U^{i}_{\mathtt{s}_v}$
\begin{align}
U^{\infty,\mathtt{s}_v}_{\mathtt{s}}(\hat{x},\theta)\hspace{-0.1cm}=&V^{\infty,\mathtt{s}_v}_{\mathtt{s}}(\hat{x},\theta)\hspace{-0.07cm}+\hspace{-0.07cm}\frac{\rho_1}{4\pi\mu\beta_{2}}\frac{\omega^2\omega_{n_0}^2}{\omega_{n_0}^2\hspace{-0.15cm}-\omega^2} \hspace{-0.07cm}\left[(E_{n_0}^{D_1}\hspace{-0.07cm}\cdot \hspace{-0.05cm}V^{t,\mathtt{s}_v}(z_1,\theta))\hspace{-0.05cm}\cdot\hspace{-0.05cm} \hat{x}^{\perp_h}V^{t,\mathtt{s}_h}(z_1,-\hat{x})\hspace{-0.07cm}+\hspace{-0.07cm}(E_{n_0}^{D_1}\hspace{-0.07cm}\cdot \hspace{-0.05cm}V^{t,\mathtt{s}_v}(z_1,\theta))\hspace{-0.05cm}\cdot \hspace{-0.05cm}\hat{x}^{\perp_v} V^{t,\mathtt{s}_v}(z_1,-\hat{x})\right] 
     \nonumber\\
    & + O(\max\{a,a^{2-2h_1})\label{s-part-farfield-exp-in-v-incident},
\end{align}
uniformly in all the directions of $\hat{x}$ and $\theta$ in $\mathbb{S}^2$.
By choosing $\theta=-\hat{x}$, and using \eqref{w-wn0-choosen} and \eqref{def-alpha}, in \eqref{s-part-farfield-exp-in-h-incident} and \eqref{s-part-farfield-exp-in-v-incident} we get \eqref{theorem-s-h-incident-farfield-backscat} and \eqref{theorem-s-v-incident-farfield-backscat} respectively. 
\hfill \qed
}
\section{ Proof of Theorem \ref{theorem}\label{section-proofs-forward-problem-j-M} for general $M$}
 The corresponding Lippmann Schwinger equation for the problem (\ref{internal-elastic-variable-rho}--\ref{KRC-U-s-variable-rho}) is
 given by
    \begin{align}\label{Lippmann-Multipleinclusions-Variable-density-mp}
U^s(x)=V^s(x)+\sum_{j=1}^M\int_{D_j}\omega^2\alpha(y)G^\omega(x,y)\cdot U^t(y)dy,\quad x\in\mathbb{R}^3.
\end{align}
 

The apriori estimates satisfied by the total field associated with this case are stated below as Lemma \ref{lemma-apriori-multiple-inclusions-normU-int-u-abydsmall}, for which the proof is discussed briefly in Section \ref{Technical-lemmas}. With the help of the estimates stated in this lemma, we prove the results in Theorem \ref{theorem}.
    \begin{lemma}[Apriori estimates]\label{lemma-apriori-multiple-inclusions-normU-int-u-abydsmall}
    
  The total field $U^t$  enjoys the following estimates:
\begin{align}\setcounter{mysubequations}{0}
     \mysubnumber&\sum_{j=1}^M\norm{U^t}_{(L^2(D_j))^3}^2\lesssim a^{-2h}\sum_{j=1}^M\norm{V^t}_{(L^2(D_j)^3}^2,\label{norm-u-app-mp}
\\
\mysubnumber& \int_{D_j}U^t(x)dx= \dfrac{1}{(1-\alpha(z_j)\omega^2 \lambda_{n_0}^j)}E_{n_0}^{D_j}\cdot V^t(z_j)+O\left(\max\{a^3,Ma^{3-2h}\frac{a}{d} \}\right),\label{int-ut-app-Dj-multiple-particle-lemma-mp}
 \end{align}
   whenever $(M-1)\frac{a}{d}< a^h$, where $0< h[:=\underset{j=1,\cdots,M}{\max}\{h_j\}] <1$ and for the incident frequency $\omega$ satisfying
   \begin{align}
|\omega^2\hspace{-.1cm}-\omega_{n_{0_{(j)}}}^2\hspace{-.07cm}|\sim a^{h_j},\mbox{ with }  \omega_{n_{0_{(j)}}}^2\hspace{-.2cm}:=\hspace{-.1cm}\frac{1}{\rho_j\lambda_{n_{0}}^j},\,  j=1,2,\cdots,M,\mbox{ and for } E_{n_0}^{D_j}\hspace{-.1cm}:=\hspace{-.1cm}\sum_{l=1}^{l_{\lambda_{n_0}^j}}\hspace{-.1cm}{\int_{D_j}\hspace{-.2cm}{e_{n_0,l}^j(x)dx}}\otimes\hspace{-.2cm}\int_{D_j}\hspace{-.2cm}{e_{n_0,l}^j(x)dx}.\label{omega-n0-j-chosen-E-matrix-scaling}
\end{align} 
Here, $(\lambda_n^j, e^j_{n,l})_{n\in\mathbb{N},\; l=1,\cdots l_{\lambda_{n}^j}}$ is the complete orthonormal eigensystem of the compact, self adjoint operator $N_{D_j}^0$\footnote{Here, $N_{D_j}^0$ is same as the operator defined  in \eqref{def-Newton-operator} but associated to the inclusion $D_j$ instead of $D$.} with $l_{\lambda_{n}^j}$ indicating the geometric multiplicity of the eigen value $\lambda_n^j$ and $n_0$ being a chosen fixed natural number.
\end{lemma}
\subsection{Estimation of scattered field}
 To estimate the scattered field, we start with considering the Lippmann Schwinger equation \eqref{Lippmann-Multipleinclusions-Variable-density-mp} and employ the Taylor series expansion of $\alpha(y)G^\omega(x,y)$ about $y$ near $z_j$. For $x$ away from $\cup_{j=1}^M D_j$, we have
\begin{align}
U^s(x)=V^s(x)+\sum_{j=1}^M\left[\int_{D_j}\alpha(z_j)\omega^2G^\omega(x,z_j)\cdot U^t(y)dy+\mathcal{A}_j\right]\label{scattered-field-eqn-with-A-j-term-mp}
\end{align}
where  $\mathcal{A}_j$ is defined as
\begin{align}
\mathcal{A}_j(k):=\omega^2\int_{D_j}\int_{0}^1\nabla_y(\alpha(z_j+t(y-z_j))G^\omega(x,z_j+t(y-z_j))_k)\cdot(y-z_j)\cdot U^t(y)dy,\; k=1,2,3\nonumber
\end{align}
and can be estimated as
\begin{align}
|\mathcal{A}_j|&\underset{\eqref{variable-density-final-B1-app}}{=}O(a^{\frac{1}{2}}\norm{U^t}_{(L^2(D_1))^3}).\label{A-j-estimation-mp}
\end{align}
Substituting \eqref{A-j-estimation-mp} in \eqref{scattered-field-eqn-with-A-j-term-mp}, we obtain
\begin{align}
U^s(x)&\underset{\eqref{scattered-field-eqn-with-A-j-term-mp},\eqref{norm-u-app-mp}}
{=}V^s(x)+\sum_{j=1}^M\omega^2\alpha(z_j)G^\omega(x,z_j)\cdot\int_{D_j} U^t(y)dy+O(M a^{2-h})
\nonumber\\
&\underset{\eqref{int-ut-app-Dj-multiple-particle-lemma-mp}}{=}\hspace{-0.15cm}V^s(x)\hspace{-0.05cm}+\hspace{-0.1cm}\sum_{j=1}^M\hspace{-0.1cm}\dfrac{\omega^2\alpha(z_j)}{(1-\alpha(z_j)\omega^2 \lambda_{n_0}^j)}G^\omega(x,z_j)\hspace{-0.05cm}\cdot\hspace{-0.05cm}E_{n_0}^{D_j}\cdot V^t(z_j)
 \hspace{-0.05cm}+\hspace{-0.05cm}O\left(\max\{Ma, M^2a^{1-2h}\frac{a}{d}  \}\right)
\end{align}
which is nothing but the required scattered field expansion by using \eqref{def-alpha} and \eqref{omega-n0-j-chosen-E-matrix-scaling}.
\subsection{Estimation of farfields}
 To estimate the farfield, consider the Lippmann Schwinger equation \eqref{Lippmann-Multipleinclusions-Variable-density-mp} and  employ the asymptotic expansion of $G^{\omega}$  and the Taylor series expansion of $\alpha(y)G_k^{\omega,\infty}(\hat{x},y)$
 about $y$ near $z_j$, for $k=1,2,3$ and $j=1,2,\cdots, M$. Following the similar procedure as  in Section \ref{section-mass-density-reconstruction}, we get the shear and pressure part of the farfield as follows:
 \begin{align}
   U^\infty_{\mathtt{s}}(\hat{x},\theta)=V^\infty_{\mathtt{s}}(\hat{x},\theta)+\sum_{j=1}^M\int_{D_j}\omega^2\alpha(z_j)G^{\omega,\infty}_{\mathtt{s}}(\hat{x},z_j)\cdot U^t(y,\theta)dy +\sum_{j=1}^M\mathtt{Es}_{j}\label{s-part-Ut-IE-mp}
\\
\mbox{and }\;\,   U^\infty_{\mathtt{p}}(\hat{x},\theta)=V^\infty_{\mathtt{p}}(\hat{x},\theta)+\sum_{j=1}^M\int_{D_j}\omega^2\alpha(z_j)G^{\omega,\infty}_{\mathtt{p}}(\hat{x},z_j)\cdot U^t(y,\theta)dy +\sum_{j=1}^M\mathtt{Ep}_{j}\label{scat-field-estimation-withEpj-term}
 \end{align}
 where $\mathtt{Es}_{j}$ and $\mathtt{Ep}_{j}$, for $j=1,2,\cdots,M$ are defined as
 \begin{align}
  \mathtt{Es}_{j}(k):=\int_{D_j}\hspace{-0.2cm}\omega^2\int_{0}^1\hspace{-0.2cm}\nabla_y\bigg(\alpha(z_j+t(y-z_j))G^{\omega,\infty}_{\mathtt{s}}(\hat{x},z_j+t(y-z_j))\bigg)_k\cdot (y-z_j)dt\cdot U^t(y,\theta)dy, k=1,2,3
  \nonumber\\
     \mathtt{Ep}_{j}(k):=\int_{D_j}\hspace{-0.2cm}\omega^2\int_{0}^1\hspace{-0.2cm}\nabla_y\bigg(\alpha(z_j+t(y-z_j))G^{\omega,\infty}_{\mathtt{p}}(\hat{x},z_j+t(y-z_j))\bigg)_k\cdot (y-z_j)dt\cdot U^t(y,\theta)dy, k=1,2,3\nonumber
 \end{align}
 satisfying the estimate
 \begin{align}
  \sum_{j=1}^M\mathtt{Ep}_{j}\underset{\eqref{E-r-estimation-in-terms-of-a}}{=}O(\sum_{j=1}^M\norm{U^t}_{(L^2(D_j))^3}a^{\frac{1}{2}})\underset{\eqref{norm-u-app-mp}}{=}O(Ma^{2-h}).\label{E-p-j-est-mp}
 \end{align}
 By applying  the mixed reciprocity relation \eqref{lemma-mixed-reciprocity-elastic-p}, i.e., $G^{\omega,\infty}(\hat{x},y)\cdot U=\frac{(U\cdot\hat{x})}{4\pi(\lambda+2\mu)\beta_1}V^{t,\mathtt{p}}(y,-\hat{x})$, and using \eqref{E-p-j-est-mp} and \eqref{int-ut-app-Dj-multiple-particle-lemma-mp} in  \eqref{scat-field-estimation-withEpj-term}, we obtain
 \begin{align}
     U^\infty_{\mathtt{p}}(\hat{x},\theta)&\underset{\eqref{lemma-mixed-reciprocity-elastic-p},\eqref{E-p-j-est-mp}}{=}V^\infty_{\mathtt{p}}(\hat{x},\theta)+\sum_{j=1}^M\frac{\omega^2}{4\pi(\lambda+2\mu)\beta_1}\alpha(z_j)\left(\hat{x}\cdot\int_{D_j} U^t(y,\theta)dy \right)V^{t,\mathtt{p}}(z_j,-\hat{x})+O(Ma^{2-h})
 \nonumber\\
& \underset{\eqref{int-ut-app-Dj-multiple-particle-lemma-mp}}{=}    V^\infty_{\mathtt{p}}(\hat{x},\theta)+\frac{\omega^2}{4\pi(\lambda+2\mu)\beta_1}\sum_{j=1}^M\dfrac{\alpha(z_j)}{(1-\alpha(z_j)\omega^2 \lambda_{n_0}^j)}\hat{x}\cdot\left(E_{n_0}^{D_j}\cdot  V^t(z_j,\theta)\right)V^{t,\mathtt{p}}(z_j,-\hat{x})
\nonumber\\
&\qquad+O\left(\max\{Ma, M^2a^{1-2h}\frac{a}{d} \}\right),
 \end{align}
and hence for $\mathtt{p}$-incident wave $U_{\mathtt{p}}^i$, the pressure part of the farfield satisfies the following asymptotic expansion, uniformly in all the directions of $\hat{x},\theta\in\mathbb{S}^2$,
\begin{align}
      U^\infty_{\mathtt{p}}(\hat{x},\theta)&   =  V^\infty_{\mathtt{p}}(\hat{x},\theta)+\frac{\omega^2}{4\pi(\lambda+2\mu)\beta_1}\sum_{j=1}^M\dfrac{\alpha(z_j)}{(1-\alpha(z_j)\omega^2 \lambda_{n_0}^j)}\hat{x}\cdot\left(E_{n_0}^{D_j}\cdot  V^{t,\mathtt{p}}(z_j,\theta)\right)V^{t,\mathtt{p}}(z_j,-\hat{x}) \nonumber\\
&\qquad+O\left(\max\{Ma, M^2a^{1-2h}\frac{a}{d} \right).\label{j-inclusions-farfield}
\end{align}
The above asymptotic expansion is nothing but the desired expansion \eqref{theorem-p-part-farfield-mp}. In a similar manner, by appropriately making use of mixed reciprocity relation \eqref{lemma-mixed-reciprocity-relation-s-incident-wave}, apriori estimates \eqref{norm-u-app-mp} and \eqref{int-ut-app-Dj-multiple-particle-lemma-mp} in \eqref{s-part-Ut-IE-mp}, one can derive the asymptotic expansion for the shear part of the farfield as stated in \eqref{theorem-farfield-s-part-mp}.

\section{Inverse Problem: Extraction of the  mass density}\label{section-Reconstruction-mass-density}
In this section, we justify the scheme proposed in Section \ref{section-Application} to address the inverse problem of recovering the mass density from the farfield measurements taken before and after injecting the inclusions  {one by one}. We show that reconstruction of mass density is possible by using $\mathtt{p}$ parts of the farfields corresponding to the scattered fields due to the $\mathtt{p}$-incident wave $U_{\mathtt{p}}^i$.

Firstly, by injecting the first inclusion $D_1$ to the target region, we have the estimations of elastic fields corresponding to the $\mathtt{p}$ incident wave by applying Theorem \ref{theorem} for $M=1$.  Multiplying  both sides of equation \eqref{farfield-variable-density-backscattered-field} first by matrix $E_{n_0}^B$, and then by $\hat{x}$, we get:
\begin{align}
\hat{x}\cdot \left(E_{n_0}^B\cdot U^{\infty,\mathtt{p}}_{\mathtt{p}}(\hat{x},\theta)\right) =\;&\hat{x}\cdot \left(E_{n_0}^B\cdot V^{\infty,\mathtt{p}}_{\mathtt{p}}(\hat{x},\theta)\right)+\frac{\rho_1}{4\pi(\lambda+2\mu)\beta_1}\frac{\omega^2\omega_{n_0}^2}{(\omega_{n_0}^2-\omega^2)}\bigg(\hat{x}\cdot  \left(E_{n_0}^{D_1} \cdot V^{t,\mathtt{p}}(z,\theta)\right)\bigg)
\nonumber\\
&\bigg(\hat{x}\cdot \left(E_{n_0}^B\cdot V^{t,\mathtt{p}}(z_1,-\hat{x})\right)\bigg)+O(a^{\min\{1,2-2h\}}). \label{farfield-variable-mass-density-extracion-density}
\end{align}
 By making use of the backscattered direction and using \eqref{scaling-prop-c-o-s} in \eqref{farfield-variable-mass-density-extracion-density}, we get
\begin{align}
\hat{x}\cdot \left(E_{n_0}^B\cdot U^{\infty,\mathtt{p}}_{\mathtt{p}}(\hat{x},-\hat{x})\right) =&\hat{x}\cdot \left(E_{n_0}^B\cdot V^{\infty,\mathtt{p}}_{\mathtt{p}}(\hat{x},-\hat{x})\right)+\frac{\rho_1\,a^3}{4\pi(\lambda+2\mu)\beta_1}\frac{\omega^2\omega_{n_0}^2}{(\omega_{n_0}^2-\omega^2)}\bigg(\hat{x}\cdot  \left(E_{n_0}^B \cdot V^{t,\mathtt{p}}(z_1,-\hat{x})\right)\bigg)^2
\nonumber
\\
&+O(a^{\min\{1,2-2h\}}). \label{farfield-variable-extraction-mass-density-2-step}
\end{align}
 For the frequency $\omega$ satisfying \eqref{w-wn0-choosen} and using the farfield measurements associated with the back-scattering, in the expansion \eqref{farfield-variable-extraction-mass-density-2-step}, we obtain  {$\left(\hat{x}\cdot  \left(E_{n_0}^B \cdot V^{t,\mathtt{p}}(z_1,-\hat{x})\right)\right)$ up to a sign}. Further, by using $\left(\hat{x}\cdot  \left(E_{n_0}^B \cdot V^{t,\mathtt{p}}(z_1,-\hat{x})\right)\right)$ in \eqref{farfield-variable-density-backscattered-field}, we deduce the total field  {$V^{t,\mathtt{p}}(z_1,-\hat{x})$ up to a sign}.

Next, we inject the second inclusion $D_2$ after performing the previous step, and measure the backscattered farfields corresponding to the $\mathtt{p}$-incident wave. Again, using the estimations of elastic fields stated in Theorem \ref{theorem} for the case $M=2$, we can recover {$V^{t,p}(z_2, -\hat{x})$ up to a sign}. To achieve this, we multiply both sides of equation \eqref{farfield-variable-density-backscattered-field} for $M=2$ by the matrix $E_{n_0}^B $, followed by the vector $ \hat{x} $, and use the backscatteredfarfield measurements, along with the already obtained $\left(\hat{x}\cdot \left(E_{n_0}^B \cdot V^{t,\mathtt{p}}(z_1,-\hat{x})\right)\right)$ to deduce 
{$\left(\hat{x} \cdot (E_{n_0}^B \cdot V^{t,\mathtt{p}}(z_2,-\hat{x}))\right)$
 up to a sign}. 

Further, by using $\left(\hat{x}\cdot  \left(E_{n_0}^B \cdot V^{t,\mathtt{p}}(z_j,-\hat{x})\right)\right)$, $j=1,2$  and $V^{t,\mathtt{p}}(z_1,-\hat{x})$ in \eqref{farfield-variable-mass-density-p-part-incident}, we can determine the total field {$V^{t,\mathtt{p}}(z_2,-\hat{x})$ up to a sign.}

Continuing with a similar procedure, by injecting the inclusions $D_j$ one after  another and using Theorem \ref{theorem} and farfield expansion \eqref{farfield-variable-density-backscattered-field} for $M=j$, we can determine {$\left(\hat{x} \cdot (E_{n_0}^B \cdot V^{t,\mathtt{p}}(z_j,-\hat{x}))\right)$ up to a sign and hence $V^t(z_j,-\hat{x})$ for $j=1,2,\cdots,M$}, associated to  $\mathtt{p}-incident$ wave. 

Since $\rho$ is assumed to be continuous (actually of class $C^1$) everywhere,  then $V^t$ is of class $C^2$ in $\Omega$ and hence $V^{t,\mathtt{p}}$ is of class $C^2$ in $\Omega$.

Let $\delta>0$ be a positive real number. Select all the points $z_i\in\Omega$, for $i=1,2,\cdots, {M}$, for which there exists a $ k=1,2,3,$ such that $|V_k^{t,\mathtt{p}}(z_i)|>\delta$, where $V_k^{t,\mathtt{p}}$ is the $k^{th}$ entry of $V^{t,\mathtt{p}}$.

Recall that we have reconstructed $V^{t,\mathtt{p}}(z_i)$ up to a sign. Indeed as we reconstruct $(V^{t,\mathtt{p}}(z_i))^2$, then we deduce the modulus of $V^{t,\mathtt{p}}(z_i)$ and it argument up to $\pi$. Thus we get  {$V^{t,\mathtt{p}}_k(z_i)$ up to the same sign}, for $k=1,2,3$, and then for $i=1,2,\cdots  {M}$. Let $z_1$ be the first selected point for which the above point is valid, i.e. $|V_k^{t,\mathtt{p}}(z_1)|>\delta$ for a $ k=1,2,3$. Now, consider other points $z_i$ (where $i\neq 1$), near $z_1$. By continuity, we know that all $V^{t,\mathtt{p}}(z_i)$ will share the same sign. Otherwise, $V^{t,\mathtt{p}}_k$ for the same $k=1,2,3$ would vanish at some point near $z_1$,  which contradicts the original assumption of $|V_k^{t,\mathtt{p}}(z_i)|>\delta$ for a $k=1,2,3$.

By using  numerical differentiation, we can reconstruct $\Delta^eV^{t,\mathtt{p}}(z_1)$ from $V^{t,\mathtt{p}}(z_i)$, $i=1,2,3,\cdots, M$. Finally, by using the fact that $V^{t,\mathtt{p}}$ satisfies the background equation, i.e., \eqref{background-problem-without-bubble-variable-density} corresponding to the $\mathtt{p}$-incident wave, we can extract the background mass density at $z_1$, as mentioned in Step 6 of the reconstruction Scheme.
{\begin{align}
    &\Delta^e V^{t,\mathtt{p}}+\omega^2\rho V^{t,\mathtt{p}}=0.
    \end{align}
    Thus, by taking the dot product on both sides of the above equation with $\overline{V^{t,\mathtt{p}}}(z_1)$, we get 
    \begin{align}
    & \Delta^eV^{t,\mathtt{p}}(z_1)\cdot \overline{V^{t,\mathtt{p}}}(z_1)+\rho_0(z_1)\omega^2V^{t,\mathtt{p}}(z_1)\cdot \overline{V^{t,\mathtt{p}}}(z_1)=0
    \\
    \mbox{ and then } &\rho_0(z_1)=-\frac{\Delta^eV^{t,\mathtt{p}}(z_1)\cdot \overline{V^{t,\mathtt{p}}}(z_1)}{\omega^2|V^{t,\mathtt{p}}(z_1)|^2}.
\end{align}}


\section{Proofs of Lemma \ref{lemma-properties}, Lemma \ref{lemma-apriori-multiple-inclusions-normU-int-u-abydsmall},  Lemma \ref{lemma-series-convergent-for-n-not-n0-single-paricle} and Lemma \ref{lemma-mixed-reciprocity-elastic}} \label{Technical-lemmas} 


\begin{proof}[{\textbf{Proof of  Lemma \ref{lemma-properties}:} }]\hfill
\begin{enumerate}
\item  {The family} $\{(\lambda_n^{D_1}, e_n^{D_1})\}$ is the complete orthonormal eigensystem of the operator
 $N_{D_1}^0:(L^2(D_1))^3\to (L^2(D_1))^3$ defined by $N_{D_1}^0(U)(x):=\int_{D_1}{\Gamma^0_z(x,y)\cdot U(y)\,dy}$. For any $\xi \in B$, define  $\widetilde{e_n^{D_1}}{(\xi)}$ as  $\widetilde{e_n^{D_1}}{(\xi)}:=e
_n^{D}(z_1+{a}\xi)$ and $e_n^{B}(\xi):=\dfrac{\widetilde{e_n^{D_1}}{(\xi)}}{\norm{\widetilde{e_n^{D_1}}}}_{(L^2(B))^3 }$. We prove this result in two steps.

\begin{itemize}
\item Firstly, due to the observation that
\begin{align}
\norm{\widetilde{e_n^{D_1}}}_{(L^2(B))^3}^2
\hspace{-0.1cm}=\hspace{-0.2cm}\int_{B}{\hspace{-0.05cm}\sum_{i=1}^3|(\widetilde{e_n^{D_1}}(\xi))_i|^2\,d\xi}
=\hspace{-0.2cm}\int_{B}{\hspace{-0.05cm}\sum_{i=1}^3|(e_n^{D_1}(z_1+{a}\xi))_i|^2\,d\xi}
\hspace{-0.1cm}=\hspace{-0.1cm} \frac{1}{{a}^3}\int_{D_1}\hspace{-0.1cm}{\sum_{i=1}^3|(e_n^{D_1}(x))_i|^2\,dx}
\mbox { as }\norm{e_n^{D_1}}_{(L^2(D_1))^3}=1,
\label{tilde-norm-end-scaling}
\end{align}
we get the scaling properties of the eigen functions as 
\begin{eqnarray}
e_n^{B}(\xi):=a^{\frac{3}{2}}{\widetilde{e_n^{D_1}}{(\xi)}}=a^{\frac{3}{2}}{{e_n^{D_1}}{(z_1+{a}\xi)}}
\label{enB-in-terms-of-enD} 
\end{eqnarray}
\begin{eqnarray*}
\mbox{ and }\int_{D_1}{e_n^{D_1}(x)dx}\hspace{-0.2cm}&=&\hspace{-0.2cm}\int_{B}{e_n^{D_1}(z_1+{a}\xi){a}^3\, d\xi}
\,=\, {a}^3\int_{B}{\widetilde{e_n^{D_1}}{(\xi)}\,d\xi} 
\,=\, {a}^3\int_{B}e_n^{B}(\xi)\norm{\widetilde{e_n^{D_1}}}_{(L^2(B))^3}\,d\xi={{a}^{\frac{3}{2}}\int_{B}{e_n^{B}(\xi)\,d\xi} }.
\label{step-in-end-scaling} 
\end{eqnarray*}
\item Secondly, to get the scaling property of the eigenvalues, making use that $e_n^{D_1}$ is an eigen function of $N_{D_1}^0$ corresponding to the eigenvalue $\lambda_n^{D_1}$, we observe 
\begin{align*}
\lambda_n^{D_1} e_n^{D_1}(x)&=N_{D_1}^0( e_n^{D_1})(x)\,\underset{\eqref{def-Newton-operator}, \eqref{entrywise-FM-zerof} }{=}\,\int_{D_1}{ \left[\left(\dfrac{-\gamma_1}{4\pi}\dfrac{\delta_{kl}}{|x-y|}-\gamma_2\dfrac{(x_k-y_k)(x_l-y_l)}{4\pi|x-y|^3}\right)\right]_{k,l=1,2,3}\hspace{-1.2cm}\cdot e_n^{D_1}(y)\,dy}
 \nonumber\\
&=\int_{B}{ \left[\left(\dfrac{-\gamma_1}{4\pi}\dfrac{\delta_{kl}}{|x-z-{a}\eta|}-\gamma_2\dfrac{(x_k-z_k-{a} \eta_k)(x_l-z_l-{a}\eta_l)}{4\pi|x-z-{a} \eta|^3}\right)\right]_{k,l=1,2,3}\hspace{-1cm}\cdot e_n^{D_1}(z_1+{a}\eta){a}^3 \,d\eta}
\nonumber\\
&={a}^3 \int_{B}{ \left[\left(\dfrac{-\gamma_1}{4\pi}\dfrac{\delta_{kl}}{{a} |\frac{x-z}{{a}}-\eta|}-\gamma_2\dfrac{{a}(\frac{x_k-z_k}{{a}}- \eta_k){a}(\frac{x_l-z_l}{{a}}-\eta_l)}{4\pi{a}^3|\frac{x-z}{{a}}- \eta|^3}\right)\right]_{k,l=1,2,3}\hspace{-1cm}\cdot \widetilde{e_n^{D_1}}{(\eta)}\,d\eta}
\nonumber\\
&\underset{\eqref{entrywise-FM-zerof}}{=} {a}^2\int_{B}{ \Gamma_z^0(\xi,\eta)\cdot \widetilde{e_n^{D_1}}(\eta)\,d\eta}
\qquad[\because x\in D(=z+{a}B),\,\exists \,\xi\in B  ,\,\mbox{ s.\,t}, \,{x}=z+a\xi] 
\nonumber\\
&\underset{\underset{\eqref{enB-in-terms-of-enD}}{\eqref{def-Newton-operator} }}{=}{a}^2a^{-\frac{3}{2}}(N^0_{B}){e_n^{B}}{(\xi)}
\nonumber\\
&={a}^2a^{-\frac{3}{2}}\lambda_n^{B}{e_n^{B}}{(\xi)}
\nonumber\\
&\underset{\eqref{def-Newton-operator}}{=}{a}^2\lambda_n^{B}e_n^{D_1}(x)
\nonumber\\
\mbox{and hence }\lambda_n^{D_1}&={a}^2\lambda_n^{B},
\end{align*}
\end{itemize}
\medskip
\item (i)\;\, Making use of the properties of chosen frequency $\omega$ 
from (\ref{w-wn0-choosen}) and the scaling properties of the eigenvalues of $N^0_{D_1}$  from \eqref{scaling-prop-c-o-s},  $|1-\alpha(z_1)
\omega^2\lambda_n^{D_1}|$ for $n\neq n_0$ can be estimated;
\begin{eqnarray}
|1-\alpha(z_1)
\omega^2\lambda_n^{D_1}|
&=& |1-\alpha(z_1)
\omega^2\lambda_{n_0}^{D_1}-\alpha(z_1)
\omega^2(\lambda_n^{D_1}-\lambda_{n_0}^{D_1})|
\nonumber\\
&=&|\alpha(z_1)
\omega^2(\lambda_n^{D_1}-\lambda_{n_0}^{D_1})
)\left(1-\frac{1-\alpha(z_1)
\omega^2\lambda_{n_0}^{D_1})}{\alpha(z_1)
\omega^2(\lambda_n^{D_1}-\lambda_{n_0}^{D_1})
}\right)|
\nonumber\\
&\underset{\underset{\eqref{scaling-prop-c-o-s}}{\eqref{def-of-rho}}}{\geq}&|C_4|\,|{\lambda}_n^{B}-{\lambda}_{n_0}^{B}
|\left[1- \frac{{b\,}a^h}{|C_4|\, |{\lambda}_n^{B}-{\lambda}_{n_0}^{B}|
}\right]
\nonumber\\
&>&|C_4||({\lambda}_{n_0}^{B}-{\lambda}_{\tilde{n_0}}^{B})
|\left[1- \frac{{b\,}a^h}{|C_1|\, |{\lambda}_{n_0}^{B}-{\lambda}_{\tilde{n_0}}^{B}|
}\right]
\nonumber\\
&>&|C_4|\,|{\lambda}_{n_0}^{B}-{\lambda}_{\tilde{n_0}}^{B}|-{b\,}a^h, \nonumber
\end{eqnarray} 
 where  $C_4:=(c_1-\rho_0(z_1)a^{2})\omega^2\simeq O(1)$ and $\lambda_{\tilde{n}_0}^{D (B)}$, for some $\tilde{n}_0$, is an eigen value of $N^0_{D(B)}$ satisfying $|\lambda_n^{D (B)}-\lambda_{n_0}^{D (B)}|>|\lambda_{\tilde{n}_0}^{D (B)}-\lambda_{n_0}^{D (B)}| $,  $\forall n\neq n_0$, which is possible as $N^0_D$ is a compact self-adjoint operator.

 Hence, the sequence $\{|1-\alpha(z_1)
\omega^2\lambda_n^{D_1}|^2\}_{n(\neq n_0)}$  is bounded below by a positive quantity which leads to the existence and positiveness of the $\sigma:=\inf_{n\neq n_0}|1-\alpha(z_1)
\omega^2\lambda_n^{D_1}|^2$.
\medskip

\item [(ii)]
{
Since $N_{D_1}^0:(L^2(D_1))^3 \to (L^2(D_1))^3$ is a compact self-adjoint operator, the only singular values of it are its eigenvalues and zero as their unique accumulation point. Hence, to prove $(I-(\rho_{1}-\rho_{0}(z_1))\omega^2N_{D_1}^0)$ is invertible,  it is sufficient to prove that $1\slash[{(\rho_{1}-\rho_{0}(z_1))\omega^2}]$ is not an eigenvalue of $N^0_D$ $\left( \mbox{as }  {1}/{[(\rho_{1}-\rho_{0}(z_1))\omega^2]}\neq 0\right)$. 
\medskip\\ 
 First, one can observe clearly from \eqref{w-wn0-choosen} that $\dfrac{1}{(\rho_{1}-\rho_{0}(z_1))\omega^2}$ is different from $\lambda_{n_0}^{D_1}$. 
 Indeed, making use of \eqref{w-wn0-choosen}, $\dfrac{1}{(\rho_1-\rho_0(z_1))\omega^2}$ can be rewritten as $\lambda_{n_0}^{D_1}\mp \dfrac{b\,a^h}{\rho_1\omega^2}+O(a^4)$.
\medskip \\ 
  It remains to show that $\dfrac{1}{(\rho_{1}-\rho_{0}(z_1))\omega^2}$ is different from $\lambda_{n}^{D_1}$ for all $n(\neq n_0)$. We can achieve this by rewriting $\lambda_{n}^{D_1}$ for $n\neq {n_0}$ as 
$$\lambda_n^{D_1}=\lambda_{n}^{D_1}-\lambda_{n_0}^{D_1}+\lambda_{n_0}^{D_1}\underset{\eqref{w-wn0-choosen}}{=}(\lambda_{n}^{D_1}-\lambda_{n_0}^{D_1})+\dfrac{1}{(\rho_{1}-\rho_{0}(z_1))\omega^2}\pm \dfrac{b\,a^h}{\rho_1\omega^2}$$
and then using the observation that 
\begin{eqnarray}
(\lambda_{n}^{D_1}-\lambda_{n_0}^{D_1})\pm \dfrac{b\,a^h}{\rho_1\omega^2}\underset{\underset{Lemma \,\ref{lemma-properties}}{\eqref{def-of-rho}}}
{=} {a}^2[({\lambda_n^{B}}-{\lambda_{n_0}^{B}})\pm {b \,}c_1^{-1}\omega^{-2}{a}^{h}]\neq 0,
\end{eqnarray}
considering ${a}\ll 1$ and $({\lambda_n^{B}}-{\lambda_{n_0}^{B}})$ is independent of ${a}$.
}
\end{enumerate}
\end{proof}

\begin{proof}[{\textbf{Proof of  Lemma \ref{lemma-apriori-multiple-inclusions-normU-int-u-abydsmall}:}}]
 We prove this result mainly using the Lippmann Schwinger equation \eqref{Lippmann-Multipleinclusions-Variable-density-mp} and the Parseval's identity and the scaling of the operators. First, we prove the estimate \eqref{norm-u-app-mp} and then the estimate \eqref{int-ut-app-Dj-multiple-particle-lemma-mp}.
\begin{enumerate} 
\item   \textbf{ Estimation of $\mathbf{\sum_{j=1}^M\norm{U^t}^2_{(L^2(D_j))^3}}$}
\\
For $x\in D_j$, we consider the Lippmann Schwinger equation \eqref{Lippmann-Multipleinclusions-Variable-density-mp} and rewrite it by introducing the operators $T_{D_j}^\omega:(L^2(D_j))^3\to (L^2(D_j))^3$ and the elastic Navier operator $N_{D_j}^0:(L^2(D_j))^3\to (L^2(D_j))^3$, as follows\footnote{Here, $N_{D_j}^0$ and $T_{D_j}^\omega$ are same as the operator defined  in (\ref{def-T-omega-T-z-0}-\ref{def-Newton-operator}) but associated to the inclusion $D_j$ instead of $D_1$.}:
\begin{align}
 U^t(x)-\alpha(z_j)\omega^2 N^0_{D_j}(U^t)(x)
&=V^t(x)+\omega^2 (T^\omega_{D_j}-\alpha(z_j)N^0_{D_j})(U^t)(x)+\hspace{-0.1cm}\sum_{\substack{m=1 \\ m\neq j}}^M\int_{D_m}\hspace{-0.34cm}\alpha(y)\,\omega^2G^\omega(x,y)\cdot U^t(y)dy.\label{lippman-schwinger-equation-sending -multiple-inclusions}
\end{align}
Making use of the complete orthonormal eigensystem $(\lambda_n^j, {e^j_{n,l})_{n\in\mathbb{N},\; l=1,\cdots l_{\lambda_{n_0}^j}}}$ of the compact, self-adjoint operator $N_{D_j}^0$, the Parseval's identity and the positivity of $\sigma_j :=\underset{n(\neq n_0)}{\inf} \{|1-\alpha(z_j)\omega^2\lambda_n^{j}|^2\}$,  by following the procedure as it was done in Proposition \ref{lemma-apriori-estimates} to derive the estimate \eqref{norm-ut-parseval}, we obtain
\begin{align}
\norm{U^t}_{(L^2(D_j))^3}\leq \;\;&\left(1+\frac{|1-\alpha(z_j)\omega^2\lambda_{n_0}^{j}|^2}{\sigma_j}\right)^{\frac{1}{2} } \dfrac{1}{|1-\alpha(z_j)\omega^2\lambda_{n_0}^{j}|} \norm{U^t-\alpha(z_j)\omega^2N^0_{D_j}(U^t)}_{(L^2(D_j))^3}.\label{norm_u_term-multiple-particle} 
\end{align}
Hence, to estimate $\norm{U^t}_{(L^2(D_j))^3}$, the estimate  of $\norm{U^t-\alpha(z_j)\omega^2N^0_{D_j}(U^t)}_{(L^2(D_j))^3}$ is needed and in this connection we obtain from \eqref{lippman-schwinger-equation-sending -multiple-inclusions} that
\begin{align}
\norm{(I-\alpha(z_j)\,\omega^2N^0_{D_j})U^t}_{(L^2(D_j))^3}
=\norm{V^t}_{(L^2(D_j))^3}+O\left(a\norm{U^t}_{(L^2(D_1))^3}\right)+O\left(\hspace{-0.1cm}(M-1)^{\frac{1}{2}}\frac{a}{d}  \bigg(\sum_{\substack{m=1 \\ m\neq j}}^M \norm{U^t}^2_{L^2(D_m)^3}\bigg)^{\frac{1}{2}}\right).
\label{norm-i-N-0-u-t-estimation-with}
\end{align}
For,  the normed behavior of the last two terms of \eqref{lippman-schwinger-equation-sending -multiple-inclusions} can be observed by using \eqref{Nw-N0_approximation-variable-rho}, estimates of $\Gamma^\omega(z_m,z_j)$, $j\neq m$ (see \cite[Lemma 5.2]{challa-divya-sini-2024} for instance), and the Cauchy-Schwarz Inequality($CSI$) as follows:

\begin{align}
    \setcounter{mysubequations}{0}
 \mysubnumber &\;\,\norm{(T^\omega_{D_j}-\alpha(z_j)N^0_{D_j})U^t}_{(L^2(D_j))^3}^2\underset{\eqref{Nw-N0_approximation-variable-rho}}{=} O\left(a^2\norm{U^t}_{(L^2(D_1))^3}^2\right), \label{N-0-difference-T-omega-estimation}
 \\
 \mysubnumber& \sum_{\substack{m=1 \\ m\neq j}}^M \hspace{-0.05cm}\norm{\hspace{-0.1cm}\int_{D_m}\hspace{-0.48cm}\alpha(y)\,\omega^2 G^{\omega}(\cdot,\hspace{-0.05cm}y)\hspace{-0.05cm}\cdot\hspace{-0.05cm} U^t(y)dy}_{(L^2(D_j))^3}\underset{\eqref{G-omega-in-terms-of-gamma-omega}}{=} \sum_{\substack{m=1 \\ m\neq j}}^M \hspace{-0.05cm}\norm{\hspace{-0.1cm}\int_{D_m}\hspace{-0.48cm}\alpha(y)\,\omega^2 (\Gamma^{\omega}+\mathcal{H}^\omega)(\cdot,\hspace{-0.05cm}y)\hspace{-0.05cm}\cdot\hspace{-0.05cm} U^t(y)dy}_{(L^2(D_j))^3}
    \nonumber\\
 & \hspace{-0.25cm} \underset{CSI}{\leq} \hspace{-0.15cm}\sum_{\substack{m=1 \\ m\neq j}}^M \norm{U^t}_{L^2(D_m)^3}\hspace{-0.15cm}\left(\sum_{k,i=1}^3\int_{D_j}\hspace{-0.15cm}\int_{D_m}\hspace{-0.3cm}|\alpha(y)\omega^2|^2|(\Gamma^\omega+\mathcal{H}^\omega)_{ki}(x,y)|^2dydx\right)^{\frac{1}{2}}
     \hspace{-0.35cm}=\hspace{-0.1cm}O\hspace{-0.15cm}\left(\hspace{-0.15cm}a^{-2}(M-1)^{\frac{1}{2}}\frac{1}{d}a^3  \hspace{-0.15cm}\left(\sum_{\substack{m=1 \\ m\neq j}}^M \norm{U^t}^2_{L^2(D_m)^3}\hspace{-0.15cm}\right)^{\hspace{-0.15cm}\frac{1}{2}}\right).
\end{align}
Finally, using \eqref{norm-i-N-0-u-t-estimation-with} in \eqref{norm_u_term-multiple-particle}, we obtain:
{
\begin{equation}\label{norm-u-app-piecewise-density-multiple-particle}
\sum_{j=1}^M\norm{U^t}_{L^2(D_j)^3}^2\lesssim \sum_{j=1}^M\dfrac{1}{|1-\alpha(z_j)\omega^2\lambda_{{n_0}}^{j}|^2}\norm{V^t}_{(L^2(D_j)^3}^2
\end{equation}}
which is valid for {$(M-1)\frac{a}{d}a^{-h}<1$} and is nothing but our desired estimate \eqref{norm-u-app-mp}.
\medskip

\item \textbf{Estimation of $\mathbf{\int_{D_j}U^t(y)dy}$}

For $x\in D_j$, consider the modified Lippmann Schwinger equation \eqref{lippman-schwinger-equation-sending -multiple-inclusions}, and apply the Taylor series expansion of incident field $U^i$ about $z_j$, to obtain
\begin{align}
(I-\alpha(z_j)\omega^2N^0_{D_j})U^t(x)=&V^t(z_j)+\int_{0}^{1}{\nabla_x V^t(z_j+t(x-z_j))\cdot (x-z_j)dt}+\omega^2 (T^\omega_{D_j}-\alpha(z_j)N^0_{D_j})(U^t)(x)\nonumber\\
&+\sum_{\substack{m=1 \\ m\neq j}}^M\int_{D_m}\hspace{-0.2cm}\alpha(y)\,\omega^2G^\omega(x,y)\cdot U^t(y)dy.\label{find-int-ut-before-with-W-mp}
\end{align}
Set $W^j_k$ as 
$ W_k^j:=(I-\alpha(z_j)\omega^2N^0_{D_j})^{-1}\mathtt{e_k}$
and $W^j:=[W_1^j\; W_2^j\; W_3^j]^\top$ 
 with $\mathtt{e_k}$, $k=1,2,3,$ denoting the standard unit vectors in $\mathbb{R}^3$.
 Taking the dot product with $W^j$ from the left in \eqref{find-int-ut-before-with-W-mp} and integrating over $D_j$, and making use of the self adjoint property of $I-\alpha(z_j)\omega^2N^0_{D_j}$, we get
\begin{align}
\int_{D_j}{U^t(x)dx}=\int_{D_j}{ W^j dx}\cdot V^t(z_j)+\mathcal{L}_1^j + \mathcal{L}_2^j+\mathcal{L}_3^j,\;\, j=1,2,\cdots,M
\label{int-ut-app-mp}
\end{align}
where,
\begin{equation}\label{def-I1to3}
\left.\begin{array}{ccc}
\mathcal{L}_1^j&:=&\int_{D_j}{W^j\cdot \left(\int_{0}^{1}{\nabla_x V^t(z_j+t(x-z_j))\cdot (x-z_j) dt}\right)  dx}
\\
&&
\\
\mathcal{L}_2^j&:=&\int_{D_j}W^j\cdot \left(\int_{D_j}{\omega^2(\alpha(y)G^{\omega}(x,y)-\alpha(z_j)\Gamma^0(x,y))\cdot U^t(y)dy}\right) dx
\\
&&
\\
\mathcal{L}_3^j&:=&\int_{D_j}W^j\cdot \left(\sum_{\substack{m=1 \\ m\neq j}}^M\int_{D_m}\alpha(y)\,\omega^2G^{\omega}(x,y)\cdot U^t(y)dy\right) dx
\end{array}
\right\}.
\end{equation}

In order to estimate the terms $\mathcal{L}_1^j, \;\mathcal{L}_2^j$ and $\mathcal{L}_3^j$ in \eqref{int-ut-app-mp}, we make use of similar steps as done in our previous work \cite[Proposition 2.2]{challa-divya-sini-2024} and rightly use Lemma \ref{lemma-series-convergent-for-n-not-n0-single-paricle} to derive the following identities:\footnote{We use $\langle\,; \rangle_j$ to represent the $(L^2(D_j))^3$ inner product}
\begin{eqnarray}
\int_{D_j}{
\hspace{-0.1cm}W_k^j dx}=\dfrac{1}{(1-\alpha(z_j)\omega^2 \lambda_{n_0}^j)}{\sum_{l=1}^{l_{\lambda_{n_0}^j}}\langle{\mathtt{e_k}\,;\,e_{n_0,l}^j}\rangle_j\;\int_{D_j}{e_{n_0,l}^j(x)dx}}+O( a^3), \quad k=1,2,3,\label{int-W-k-j-app-1}\\
\int_{D_j}{W_k^j dx}\,=\,O\bigg(a^{3-h}\bigg)\;\, \mbox{ and }\;\,\norm{W_k^j}_{(L^2({D_j}))^3}=O( a^{\frac{3}{2}-h_j}),\quad k=1,2,3. \label{int-W1-and-norm-w-app-mp}
\end{eqnarray}
\\

Using \eqref{int-W1-and-norm-w-app-mp}, \eqref{norm-u-app-mp} and \eqref{G-omega-in-terms-of-gamma-omega}, we approximate $\mathcal{L}_1^j$, $\mathcal{L}_2^j$ and $\mathcal{L}_3^j$ defined in \eqref{int-ut-app-mp} as follows:
\setcounter{mysubequations}{0} 
\begin{align}
 \mysubnumber \quad|\mathcal{L}_1^j|^2
\underset{CSI}{\leq} \quad\;\;&\hspace{-0.4cm}\int_{D_j}\sum_{l=1}^3|\int_{0}^{1}{\hspace{-0.2cm} \nabla V_l^t(z_j+t(x-z_j))\cdot (x-z_j)|dt}|^2dx\,\sum_{k=1}^3\norm{W_k^j}_{(L^2(D_j))^3}^2
\underset{\eqref{int-W1-and-norm-w-app-mp}}{=} O(a^{8-2h})\label{I_1-app-mp}   
\end{align}
\begin{align}
 \mysubnumber \;\;|\mathcal{L}_2^j|^2
\hspace{-0.1cm}\underset{CSI,\eqref{N-0-difference-T-omega-estimation}}{\leq}& O(a^2\norm{U^t}_{(L^2(D_j))^3}^2\sum_{k=1}^3\norm{W_k^j}^2_{(L^2(D_j))^3})\underset{\eqref{int-W1-and-norm-w-app-mp},\eqref{norm-u-app-mp}}{=}O(Ma^{8-4h})\qquad\qquad\qquad\qquad\;
\label{I-2-app-mp}
\end{align}
\mysubnumber \, To approximate $\mathcal{L}_3^j$, 
 first apply Taylor series expansion for $\alpha(y)G^\omega(x,y)$ about $x\in D_j$ near $z_j$ and again about $y\in D_m, m\neq j$, near $z_m$, in the definition \eqref{def-I1to3} of $\mathcal{L}_3^j$, to get
\begin{align}\label{L-3-j-def}
 \;\;\;\mathcal{L}_3^j=\sum_{\substack{m=1 \\ m\neq j}}^M\int_{D_j}W^jdx\cdot \left(\alpha(z_m)G^\omega(z_j,z_m)\cdot\int_{D_m}\omega^2U^t(y)dy\right)
+\mathcal{Q}_1^j+\mathcal{Q}_2^j\qquad\qquad\qquad\qquad\quad\quad
\end{align}
where
\begin{align}
   & \mathcal{Q}_1^j(k):=\sum_{l=1}^3\int_{D_j}\hspace{-0.2cm}W^j_{kl}\cdot \underset{\substack{m=1 \\ m\neq j}}{\sum^M}\int_{D_m}\hspace{-0.4cm}\omega^2\int\limits_{0}^1\nabla_y(\alpha(z_m+t(y-z_m)) G^\omega_l(z_j,z_m+t(y-z_m)))\cdot (y-z_m)dt\cdot U^t(y)dy dx
    \nonumber\\
&\mathcal{Q}_2^j(k):=\sum_{l=1}^3\int_{D_j}\hspace{-0.2cm}W^j_{kl}\cdot\underset{\substack{m=1 \\ m\neq j}}{\sum^M}  \int_{D_m}\hspace{-0.4cm}\omega^2\int\limits_{0}^1\nabla_x(\alpha(y) G^\omega_l(z_j+t(x-z_j),\,y))\cdot (x-z_j)dt\cdot U^t(y)dy\, dx,\;\, k=1,2,3.
    \nonumber
\end{align}
Also we can observe that $\mathcal{Q}_1^j$ and $\mathcal{Q}_2^j$ behaves as 
\begin{eqnarray}\label{Q-1-2-j-estimation}
\mathcal{Q}_1^j=O\left(Ma^{6-2h}\frac{a}{d}+Ma^{3-2h}\frac{a^2}{d^2}\right) 
& \mbox{ and }& \mathcal{Q}_2^j=O(Ma^{3-2h}\frac{a^2}{d^2}).
\end{eqnarray}

For, making use of  \eqref{int-W1-and-norm-w-app-mp}, \eqref{A-j-estimation-mp} and the $CSI$ appropriately, we can estimate $\mathcal{Q}_1^j$ using following computations
\begin{align}
&|\mathcal{Q}_1^j|^2=\sum_{k=1}^3|\sum_{l=1}^3\int_{D_j}\hspace{-0.2cm}W^j_{kl}\cdot \underset{\substack{m=1 \\ m\neq j}}{\sum^M}\int_{D_m}\hspace{-0.4cm}\omega^2\int\limits_{0}^1\nabla_y(\alpha(z_m+t(y-z_m)) G^\omega_l(z_j,z_m+t(y-z_m)))\cdot (y-z_m)dt\cdot U^t(y)dy dx|^2
\nonumber\\
&\leq \sum_{k=1}^3\norm{W_k^j}_{(L^2(D_j))^3}\sum_{l=1}^3\int_{D_j}|\underset{\substack{m=1 \\ m\neq j}}{\sum^M}\int_{D_m}\hspace{-0.4cm}\omega^2\int\limits_{0}^1\nabla_y(\alpha(z_m+t(y-z_m)) G^\omega_l(z_j,z_m+t(y-z_m)))\cdot (y-z_m)dt\cdot U^t(y)dy |^2dx
\nonumber\\
&\leq \sum_{k=1}^3\norm{W_k^j}_{(L^2(D_j))^3}\hspace{-0.1cm}\sum_{l=1}^3\sum_{i=1}^3\underset{\substack{m=1 \\ m\neq j}}{\sum^M}\hspace{-0.1cm}\int_{D_j}\int_{D_m}\hspace{-0.4cm}|\omega^2|^2\bigg(\int\limits_{0}^1\hspace{-0.1cm}|(\nabla_y(\alpha(z_m+t(y-z_m)) G^\omega_l(z_j,z_m+t(y-z_m))))_i|| (y-z_m)|dt\bigg)^2\hspace{-0.1cm} dy dx
\nonumber\\
&\qquad \underset{\substack{m=1 \\ m\neq j}}{\sum^M}\norm{U^t(y)}_{(L^2(D_m))^3}^2
\nonumber\\
&=O\left(\sum_{k=1}^3\norm{W_k^j}^2_{(L^2(D_j))^3}\sum_{\substack{m=1 \\ m\neq j}}^{M}\norm{U^t}^2_{(L^2(D_m))^3}a^6(M-1)\left(\frac{a^2}{d^2}+\frac{a^{-4}}{d^4}a^2\right)\right)
\underset{\eqref{int-W1-and-norm-w-app-mp}, \eqref{norm-u-app-mp}}{=}O\left(M^2a^{12-4h}\left(\frac{a^2}{d^2}+\frac{a^{-2}}{d^{4}}\right)\right).
\nonumber
\end{align}
Similarly, $\mathcal{Q}_2^j$ behavior  can be observed using \eqref{int-W1-and-norm-w-app-mp}, \eqref{norm-u-app-mp} and \eqref{G-omega-in-terms-of-gamma-omega}. 
Thus, by using \eqref{Q-1-2-j-estimation}  in \eqref{L-3-j-def}, we have the estimate for 
\begin{align}
    \mathcal{L}_3^j=\sum_{\substack{m=1 \\ m\neq j}}^M\int_{D_j}W^jdx\cdot \left(\alpha(z_m)G^\omega(z_j,z_m)\cdot\int_{D_m}\omega^2U^t(y)dy\right)+O\left(\max\{Ma^{6-2h}\frac{a}{d},Ma^{3-2h}\frac{a^2}{d^2}\}\right)\label{I3-app-mp}.
\end{align}
Substituting \eqref{I_1-app-mp}, \eqref{I-2-app-mp} and \eqref{I3-app-mp} in \eqref{int-ut-app-mp} we get the desired result \eqref{int-ut-app-Dj-multiple-particle-lemma-mp}, i.e.,
\begin{align}
\int_{D_j}{\hspace{-0.4cm}U^t(x)dx}\hspace{-0.05cm}=\hspace{-0.15cm}& 
\int_{D_j}\hspace{-0.4cm}W^jdx\cdot \hspace{-0.1cm}\left[ V^t(z_j)+\hspace{-0.1cm}\sum_{\substack{m=1 \\ m\neq j}}^M\hspace{-0.2cm}\omega^2\alpha(z_m)G^\omega(z_j,z_m)\cdot\hspace{-0.2cm}\int_{D_m}\hspace{-0.45cm}U^t(y)dy\right]
\hspace{-0.15cm}+\hspace{-0.1cm}O\left(\hspace{-0.1cm}\max\{a^{4-h},M^{\frac{1}{2}}a^{4-2h},Ma^{6-2h}\frac{a}{d},Ma^{3-2h}\frac{a^{2}}{d^2}\}\hspace{-0.1cm}\right)
\nonumber\\
\underset{\eqref{int-W-k-j-app-1}}{=}&\dfrac{1}{(1-\alpha(z_j)\omega^2 \lambda_{n_0}^j)}E_{n_0}^{D_j}\cdot V^t(z_j)+ \dfrac{1}{(1-\alpha(z_j)\omega^2 \lambda_{n_0}^j)}E_{n_0}^{D_j}{\sum_{\substack{m=1 \\ m\neq j}}^M \omega^2\alpha(z_m)G^\omega(z_j,z_m)\cdot\int_{D_m}\hspace{-0.3cm}U^t(y)dy}
\nonumber\\
&+O(a^3)+O\left({Ma^{3-h}\frac{a}{d}}\right)+O\left(\max\{a^{4-h},M^{\frac{1}{2}}a^{4-2h},Ma^{6-2h}\frac{a}{d}, Ma^{3-2h}\frac{a^2}{d^2}\}\right)\label{int-ut-algebraic-system-mp}
\end{align}
 Observing that the system \eqref{int-ut-algebraic-system-mp} is invertible whenever $(M-1)\frac{a}{d}<a^{h}$
 and the second term of the approximation \eqref{int-ut-algebraic-system-mp} behaves as ${O(Ma^{3-2h}\frac{a}{d}})$, we will get the desired estimate \eqref{int-ut-app-Dj-multiple-particle-lemma-mp}.
\end{enumerate}
\end{proof}

\begin{lemma}\label{lemma-series-convergent-for-n-not-n0-single-paricle}
  {The series}$\underset{n\neq {n_0}}{\sum}\dfrac{1}{(1-(\rho_{1}-\rho_{0})\omega^2 \lambda_n^{D_1})}\left(\mathtt{e_k}\cdot\int_{D_1}{e_n^{D_1}(x)\, dx}\right)\int_{D_1}{e_n^{D_1}(x)\,dx}\cdot \mathtt{e_m}\;$  for  $\;m=1,2,3$ is convergent and behaves as $O( a^3)$. 
\end{lemma}
\begin{proof}  We have that $N_B^0:(L^2(B))^3\to (L^2(B))^3$ is a self-adjoint compact operator with complete orthonormal eigen system $(\lambda_n^{B},e_n^{B})$ of $N_B^0$.
Now, let $ \Lambda$ and $\tau$ be scalars such that $\Lambda=\underset{n}{\max}\{\lambda_n^{B}\}$ 
and $\tau -\Lambda \gg 1$. Also, let $W_\tau$ be a matrix satisfying $(N^0_{B}-\tau I)W_{\tau}=I$ and denote its $m^{th}$ row by ${W_\tau}_m$ and its $(m,k)^{th}$ entry by ${W_\tau}_{mk}$ for $m,k=1,2,3$ . 
Then, observe that
\begin{align}
\hspace{-2cm}\setcounter{mysubequations}{0}
 \mysubnumber\quad \int_{B}{e_n^{B}(\xi)\,d\xi}=\int_{B}I\cdot e_{n}^{B}(\xi)\,d\xi
\,=\int_{B}(N^0_{B}-\tau I)W_\tau \cdot e_{n}^{B}(\xi)\,d\xi
\,=(\lambda_n^{B}-\tau )\int_{B}W_\tau(\xi) \cdot e_{n}^{B}(\xi)\, d\xi,
\nonumber
\end{align}
which gives
\begin{eqnarray}
\langle \mathtt{e_m}, e_n^{B}\rangle_{({L}^2(B))^3}=\int_{B}{\hspace{-0.1cm}e_n^{B}(\xi)\,d\xi}\cdot \mathtt{e_m}\,=\,(\lambda_n^{B}-\tau )\int_{B}{\hspace{-0.1cm}W_\tau}_m(\xi) \cdot e_{n}^{B}\,d\xi=(\lambda_n^{B}-\tau )\,\langle{W_\tau}_m, \, e_{n}^{B}\rangle_{({L}^2(B))^3},\, m=1,2,3.\quad\label{em-inner-product-eigen-vector-n-not-n0-single-inclusion}
\end{eqnarray}
  {
\begin{align}
\hspace{-3.5cm} \mysubnumber\quad
 \norm{W_{\tau m}}^2_{(L^2(B))^3}=\sum_{n}|\langle{W_\tau}_m, e_n^{B}\rangle_{({L}^2(B))^3}|^2
\underset{\eqref{em-inner-product-eigen-vector-n-not-n0-single-inclusion}}{=}\sum_{n}|\dfrac{\langle \mathtt{e_m}, e_n^{B}\rangle_{({L}^2(B))^3}}{(\lambda_n^{B}-\tau)}|^2, \quad\, m=1,2,3.
\label{norm-W-tau-m}
\end{align}
Now by making use of the Lemma \ref{lemma-properties} we observe that
\begin{align*}
\sum_{n\neq {n_0}}\hspace{-0.05cm}|\dfrac{1}{(1-(\rho_{1}-\rho_{0}(z_1))\omega^2 \lambda_n^{D_1})}\langle \mathtt{e_k}, e_n^{D_1}\rangle_{(L^2(D_1))^3}\langle \mathtt{e_m}, e_n^{D_1}\rangle_{(L^2(D_1))^3}|
\leq & \dfrac{a^{3}}{\sigma_1}\sum_{n\neq {n_0}}\hspace{-0.05cm}|\langle \mathtt{e_k}, e_n^{B}\rangle_{(L^2(B))^3}\langle \mathtt{e_m}, e_n^{B}\rangle_{(L^2(B))^3}|
\\
\leq\dfrac{a^3}{\sigma_1}\bigg(\sum_{n\neq n_0}|&\langle \mathtt{e_k}, e_n^{B}\rangle_{(L^2(B))^3}|^2\bigg)^{1/2\hspace{-0.2cm}}\left(\sum_{n\neq n_0}|\langle \mathtt{e_m}, e_n^{B}\rangle_{(L^2(B))^3}|^2\right)^{1/2}
\nonumber\\
&\hspace{-2.2cm}\underset{\eqref{norm-W-tau-m}}{\leq}  \,\dfrac{a^3}{\sigma_1}|\tau|^2\norm{W_{\tau m}}_{(L^2(B))^3} \norm{W_{\tau k}}_{(L^2(B))^3}
=O(a^3).
\end{align*}
}
\end{proof}

\begin{lemma}\label{lemma-mixed-reciprocity-elastic}
The total field $V^t$ and the Green's matrix $G^{\omega}$ enjoy the following mixed reciprocity relations: 
 \begin{eqnarray}\label{lemma-mixed-reciprocity-elastic-p}
4\pi(\lambda+2\mu)\beta_1 G^{\omega,\infty}_\mathtt{p}(\hat{x},y)\cdot  {U}= (U\cdot \hat{x})V^{t,\mathtt{p}}(y,-\hat{x});
\end{eqnarray}
and
\begin{eqnarray}
    G^{\omega,\infty}_\mathtt{s}(\hat{x},y)\cdot  {U}=\frac{(U\cdot\hat{x}^{\perp_h})}{4\pi\mu\beta_{2_h}} V^{t,\mathtt{s}_h}(y,-\hat{x})+\frac{(U\cdot\hat{x}^{\perp_v})}{4\pi\mu\beta_{2_v}} V^{t,\mathtt{s}_v}(y,-\hat{x});\label{lemma-mixed-reciprocity-relation-s-incident-wave}
\end{eqnarray}
where $U$ is any constant vector {(w.r.t $x$ and $y$)},
$\beta_{2_h}:=\frac{\beta_2\alpha_1}{\sqrt{\alpha_1^2+\alpha_2^2}}$ and  $\beta_{2_v}:=\frac{\beta_2\alpha_2}{\sqrt{\alpha_1^2+\alpha_2^2}}$.
\end{lemma}
\begin{proof}
We prove this result in two cases: $y\in\mathbb{R}^3\setminus{\Omega}$ and $y\in\Omega$, respectively.
\begin{enumerate}
\item[]\textbf{\hspace{-.6cm}Case 1 - $y\in\mathbb{R}^3\setminus{\Omega}$ :}
First, observe that by using the total field $V^t$ satisfies \eqref{background-problem-without-bubble-variable-density}, the Green's tensor $G^\omega$ satisfies \eqref{Greens-function-variable-rho} and the fundamental matrix $\Gamma^\omega$ satisfies \eqref{fundamental-homogeneous-background-massdensity-tildrho}, we can derive
 \begin{align}
    &(i)\; V^t(y,\theta) = V^i(y,\theta) + \int_{\Omega} \omega^2(\rho_0(\zeta) - \tilde{\rho_0})\Gamma^\omega(y,\zeta) \cdot V^t(\zeta,\theta) \, d\zeta, \quad y \in \mathbb{R}^3 \setminus \bar{\Omega}; \label{L-S-equation-in-terms-Vt-elastic}\\
  \mbox{and }  &(ii)\; G^\omega(x,y) = \Gamma^\omega(x,y) + \int_{\Omega} \omega^2(\rho_0(\zeta) - \tilde{\rho_0})\Gamma^\omega(y,\zeta) G^\omega(x,\zeta) \, d\zeta, \quad x, y \in \mathbb{R}^3 \setminus \bar{\Omega}; \label{scatteredfield-elastic-model-before-step}
\end{align}
Now, we prove the result in three parts. 

\begin{itemize}
\item[(1)]  Proof of \eqref{L-S-equation-in-terms-Vt-elastic}: From \eqref{background-problem-without-bubble-variable-density}, we have
\begin{eqnarray}
\Delta^eV^t(\zeta )+\omega^2\rho_0(\zeta )V^t(\zeta )=0, \quad \zeta \in \Omega.\nonumber
\end{eqnarray}
Multiplying this equation by  $\Gamma^\omega(y,\zeta )$ for $y\in\mathbb{R}^3\setminus \bar{\Omega}$ and integrating over $\Omega$, we get
\begin{eqnarray}
\int_{\Omega}\Gamma^\omega(y,\zeta )\cdot \Delta^eV^t(\zeta )\,d\zeta +\omega^2\int_{\Omega}\rho_0(\zeta )\Gamma^\omega(y,\zeta )\cdot V^t(\zeta )\,d\zeta =0,\quad y\in\mathbb{R}^3\setminus \bar{\Omega}.\nonumber
\end{eqnarray}
Making use of Betti's third identity \cite{LOW-FREQUENCY} in the above and using \eqref{fundamental-homogeneous-background-massdensity-tildrho}, we deduce that,
\begin{eqnarray}
  \int_{\partial \Omega}\hspace{-0.23cm}[\Gamma^\omega(y,\zeta )\cdot T_\nu V^t(\zeta )-T_\nu\Gamma^\omega(y,\zeta )\cdot V^t(\zeta )]\,ds(\zeta ) 
=\hspace{-0.1cm}\int_{\Omega}\hspace{-0.1cm}\omega^2(\tilde{\rho_0}-\rho_0(\zeta ))
\Gamma^\omega(y,\zeta )\cdot V^t(\zeta )\,d\zeta, \quad y\in\mathbb{R}^3\setminus \bar{\Omega}.\quad \label{lippmann-in-terms-of -Vs-with-boundary-terms-elastic}
\end{eqnarray}
Considering $V^t=V^i+V^s$, where $V^i$ satisfies the background equation $\Delta^eV^i+\omega^2\tilde{\rho_0}V^i=0 \,\, \text{in } \mathbb{R}^3$ and   $V^s$ is the radiating solution, and making use of representation formulas (see \cite{COLTON-IE, COLTON-KRESS} for instance), we can rewrite the above equation \eqref{lippmann-in-terms-of -Vs-with-boundary-terms-elastic} for the incident direction $\theta\in\mathbb{S}^2$ as:
\begin{eqnarray}
-V^s(y,\theta)=\int_{\Omega}\omega^2(\tilde{\rho_0}-\rho_0(\zeta ))\Gamma^\omega(y,\zeta )\cdot V^t(\zeta,\theta )
\,d\zeta , 
\end{eqnarray}
which is nothing but \eqref{L-S-equation-in-terms-Vt-elastic}.
\item[(2)] Proof of \eqref{scatteredfield-elastic-model-before-step}: Since, the Green's tensor $G^\omega$ satisfies \eqref{Greens-function-variable-rho}, we have
\begin{eqnarray}
\Delta^e_{\zeta }G^\omega(\zeta ,y)+\omega^2\rho_0(\zeta )G^\omega(\zeta ,y)=-\delta(\zeta ,y)I, \quad \mbox{ on }\mathbb{R}^3.\label{Green-tensor-elastic}
\end{eqnarray}
For  $y\in \mathbb{R}^3\setminus\bar{\Omega}$, multiplying equation \eqref{Green-tensor-elastic} with $\Gamma^\omega(x,\zeta )$ for  $x\in \mathbb{R}^3\setminus \bar{\Omega}$ and integrating over $\Omega$, we get
\begin{eqnarray}
&&\int_{\Omega}\Gamma^\omega(x,\zeta )\Delta^e_{\zeta }G^\omega(\zeta ,y)\,d\zeta +\int_{\Omega}\omega^2\rho_0(\zeta )\Gamma^\omega(x,\zeta )G^\omega(\zeta ,y)\,d\zeta =0, \quad x, y \in \mathbb{R}^3\setminus\bar{\Omega},
\nonumber
\end{eqnarray}
and then by employing Betti's third identity \cite{LOW-FREQUENCY} and\eqref{fundamental-homogeneous-background-massdensity-tildrho}, we deduce 
\begin{eqnarray}
&&\int_{\Omega}\omega^2(\rho_0(\zeta )-\tilde{\rho_0})\Gamma^\omega(x,\zeta )G^\omega(\zeta ,y)\,d\zeta +\int_{\partial \Omega}\left[\Gamma^\omega(x,\zeta )T_{\nu(\zeta )}G^\omega(\zeta ,y)-T_{\nu(\zeta )}\Gamma^\omega(x,\zeta )\,G^\omega(\zeta,y)\right]\,ds(\zeta ){=}0.
\nonumber
\end{eqnarray}
This equation can be rewritten further as 
\begin{align}
&\int_{\Omega}\omega^2(\rho_0(\zeta )-\tilde{\rho_0})\Gamma^\omega(x,\zeta )G^\omega(\zeta ,y)\,d\zeta+\int_{\partial \Omega}\Big(\Gamma^\omega(x,\zeta )T_{\nu(\zeta )}\Gamma^\omega(\zeta ,y)-T_{\nu(\zeta )}\Gamma^\omega(\zeta ,y)\, \Gamma^\omega(x,\zeta )\Big)\,ds(\zeta )\nonumber\\
& +\int_{\partial \Omega}\Big(\Gamma^\omega(x,\zeta )T_{\nu(\zeta )}[G^\omega(\zeta ,y)-\Gamma^\omega(\zeta ,y)]-T_{\nu(\zeta )}\Gamma^\omega(x,\zeta )\,[G^\omega(\zeta ,y)-\Gamma^\omega(\zeta ,y)]\Big)\,ds(\zeta )=0.
\label{lippmann-with-two-boundary-terms-elastic}
\end{align}
{Observing that, for  $x\in\mathbb{R}^3\setminus\bar{\Omega}$,  employing the facts that $G^\omega-\Gamma^\omega$ is a radiating solution and $\Gamma^\omega$ is a fundamental tensor, along with the representation formulas for elastic wave propagation, 
 we have} 
\begin{align*}
 &\int_{\partial \Omega}[T_{\nu(\zeta )}\Gamma^\omega(x,\zeta )\Gamma^\omega(\zeta ,y)-\Gamma^\omega(x,\zeta )T_{\nu(\zeta )}\Gamma^\omega(\zeta ,y)]\,ds(\zeta )=0\qquad \mbox{  and  }
 \\
& \int_{\partial \Omega}[\Gamma^\omega(x,\zeta )T_{\nu(\zeta )}(G^\omega(\zeta ,y)-\Gamma^\omega(\zeta ,y))-T_{\nu(\zeta )}\Gamma^\omega(x,\zeta )(G^\omega(\zeta ,y)-\Gamma^\omega(\zeta ,y))]\,ds(\zeta )=-G^\omega(x,y)+\Gamma^\omega(x,y).
\end{align*}
Substituting the above into \eqref{lippmann-with-two-boundary-terms-elastic} gives us
\begin{eqnarray}
{G^\omega(x,y)=\Gamma^\omega(x,y)+\int_{\Omega}\omega^2(\rho_0(\zeta )-\tilde{\rho_0})\Gamma^\omega(x,\zeta )G^\omega(\zeta ,y)\,d\zeta }.\nonumber \label{lippmann-without-boundary-terms-elastic-before-symmetricity-applying}
\end{eqnarray}
Now, by making use of the symmetry of $G^\omega$ and $\Gamma^\omega$ in the above equation, we get the desired \eqref{scatteredfield-elastic-model-before-step}.
\medskip

\item[(3)] Derivation of the mixed reciprocity relations: By multiplying \eqref{scatteredfield-elastic-model-before-step} by  {$U$} ({$U$ as constant vector w.r.t x)}, we obtain
\begin{eqnarray}\label{G-Gamma-dot-U-Int-eqn}
    G^\omega(x,y)\cdot  {U}=\Gamma^{ {\omega}}(x,y)\cdot  {U}+\omega^2\int_{\Omega}(\rho_0(\zeta)-\tilde{\rho}_0)\Gamma^\omega(y,\zeta)G^\omega(\zeta,x)\cdot  {U}d\zeta\label{G-omega-Gamma-omega-related-integral-euation}
\end{eqnarray}
From the asymptotic expansion \eqref{assymptotic-expansion-Gamma-omega} of $\Gamma^\omega$, we have
\begin{align}
    \Gamma^\omega(x,y)\cdot {U}=\dfrac{1}{4\pi(\lambda+2\mu)}(\hat{x}\cdot {U})\hat{x}e^{-\mathbf{\mathtt{i}}\kappa_{\mathtt{p}}\hat{x}\cdot y}\dfrac{e^{\mathtt{\mathbf{i}}\kappa_{\mathtt{p}}|x|}}{|x|}+\dfrac{1}{4\pi\mu}({U}-({U}\cdot\hat{x})\hat{x})e^{-\mathbf{\mathtt{i}}\kappa_{\mathtt{s}}\hat{x}\cdot y}\dfrac{e^{\mathtt{\mathbf{i}}\kappa_{\mathtt{s}}|x|}}{|x|}+O(|x|^{-2}),\; |x|\to \infty,\nonumber
\end{align}
and the $\mathtt{p}$ and $\mathtt{s}$ farfields of $\Gamma^\omega$ denoted by $\Gamma^{\omega,\infty}_{\mathtt{p}}$  and $\Gamma^{\omega,\infty}_{\mathtt{s}}$ are given respectively as,  
\begin{eqnarray}\label{p,s-parts-of farfield-gamma}
{\Gamma^{\omega,\infty}_{\mathtt{p}}(\hat{x},y)\,=\, \dfrac{1}{4\pi(\lambda+2\mu)}\hat{x}\,\hat{x}^{\top}e^{-\mathbf{\mathtt{i}}\kappa_{\mathtt{p}}\hat{x}\cdot y}}\mbox{ and } {\Gamma^{\omega,\infty}_{\mathtt{s}}(\hat{x},y)\,=\,\dfrac{1}{4\pi\mu}(I-\hat{x}\,\hat{x}^{\top})e^{-\mathbf{\mathtt{i}}\kappa_{\mathtt{s}}\hat{x}\cdot y}}.
\end{eqnarray}
Similarly, making use the asymptotic expansion \eqref{assymptotic-expansion-Green} of $G^\omega$,  we have
\begin{eqnarray}
    G^\omega(x,y)\cdot  {U}=G_\mathtt{p}^{\omega,\infty}(\hat{x},y)\cdot  {U}\dfrac{e^{\mathtt{\mathbf{i}}\kappa_{\mathtt{p}}|x|}}{|x|}+G_\mathtt{s}^{\omega,\infty}(\hat{x},y)\cdot {U}\dfrac{e^{\mathtt{\mathbf{i}}\kappa_{\mathtt{s}}|x|}}{|x|}+O(|x|^{-2}),\;\, |x|\to\infty,
    \label{G-farfiled-dot-U}
\end{eqnarray}
where  the $\mathtt{p}$ and $\mathtt{s}$-parts $G_\mathtt{p}^{\omega,\infty}(\cdot,y)$ and $G_\mathtt{p}^{\omega,\infty}(\hat{x},y)$ of farfield of $G^\omega(\cdot,y)$ respectively satisfies the following, from \eqref{G-Gamma-dot-U-Int-eqn}:
\begin{eqnarray}\label{farfield-expansion-Gomega-p-part-Tnt-eqn}\label{farfield-elastic-G-Gamma-p-part}
G^{\omega,\infty}_{\mathtt{p}}(\hat{x},y)\cdot  {U}=\dfrac{1}{4\pi(\lambda+2\mu)}(\hat{x}\cdot {U})\hat{x}e^{-\mathbf{\mathtt{i}}\kappa_{\mathtt{p}}\hat{x}\cdot y}+\int_{\Omega}\omega^2 (\rho_0(\zeta)-\tilde{\rho}_0)\Gamma^\omega(y,\zeta )G^{\omega,\infty}_{\mathtt{p}}(\hat{x},\zeta )\cdot  {U}\,d\zeta,
\end{eqnarray}
and  
\begin{eqnarray}\label{farfield-expansion-Gomega-s-part-Tnt-eqn-1}
G^{\omega,\infty}_{\mathtt{s}}(\hat{x},y)\cdot  {U}=\dfrac{1}{4\pi\mu}( {U}-(\hat{x}\cdot  {U})\hat{x})e^{-\mathbf{\mathtt{i}}\kappa_{\mathtt{s}}\hat{x}\cdot y}+\int_{\Omega}\omega^2 (\rho_0(\zeta)-\tilde{\rho}_0)\Gamma^\omega(y,\zeta )G^{\omega,\infty}_{\mathtt{s}}(\hat{x},\zeta )\cdot  {U}\,d\zeta.
\end{eqnarray}
Since $U$ is a constant vector and $\hat{x}$ is a unit vector, we can find an orthonormal basis of $\mathbb{R}^3$, i.e., $\{\hat{x},\hat{x}^{\perp_h},\hat{x}^{\perp_v}\}$. Using this fact in \eqref{farfield-expansion-Gomega-s-part-Tnt-eqn-1}, we get
\begin{eqnarray}\label{farfield-expansion-Gomega-s-part-Tnt-eqn}
G^{\omega,\infty}_{\mathtt{s}}(\hat{x},y)\cdot {U}=\dfrac{1}{4\pi\mu}({(U\cdot\hat{x}^{\perp_h})\hat{x}^{\perp_h}+(U\cdot\hat{x}^{\perp_v})\hat{x}^{\perp_v}})e^{-\mathbf{\mathtt{i}}\kappa_{\mathtt{s}}\hat{x}\cdot y}+\int_{\Omega}\omega^2 (\rho_0(\zeta)-\tilde{\rho}_0)\Gamma^\omega(y,\zeta )G^{\omega,\infty}_{\mathtt{s}}(\hat{x},\zeta )\cdot {U}\,d\zeta.\quad
\end{eqnarray}
 Considering the incident wave as the $\mathtt{p}$-incident wave, the $\mathtt{s}_h$ incident wave, and the $\mathtt{s}_v$ incident wave, with the incidence direction being $-\hat{x}$, we have the following respectively, from \eqref{L-S-equation-in-terms-Vt-elastic}:
 
\begin{align}
&V^{t,\mathtt{p}}(y,-\hat{x})=V^i_{\mathtt{p}}(y,-\hat{x})+\int_{\Omega}\omega^2(\rho_0(\zeta )-\tilde{\rho_0})\Gamma^\omega(y,\zeta )
\cdot V^{t,\mathtt{p}}(\zeta ,-\hat{x})\,d\zeta, \label{lippmann-equation-Vt-Vi-terms-incident-direction-elastic}\\
&V^{t,\mathtt{s}_h}(y,-\hat{x})=V^i_{\mathtt{s}_h}(y,-\hat{x})+\int_{\Omega}\omega^2(\rho_0(\zeta )-\tilde{\rho_0})\Gamma^\omega(y,\zeta )
\cdot V^{t,\mathtt{s}_h}(\zeta ,-\hat{x})\,d\zeta, \label{lippmann-equation-Vt-Vi-terms-Ui-s-h}\\
\mbox{and }&V^{t,\mathtt{s}_v}(y,-\hat{x})=V^i_{\mathtt{s}_v}(y,-\hat{x})+\int_{\Omega}\omega^2(\rho_0(\zeta )-\tilde{\rho_0})\Gamma^\omega(y,\zeta )
\cdot V^{t,\mathtt{s}_v}(\zeta ,-\hat{x})\,d\zeta. \label{lippmann-equation-Vt-Vi-terms-Ui-s-v}
\end{align}
By adding \eqref{lippmann-equation-Vt-Vi-terms-Ui-s-h} times $\frac{(U\cdot\hat{x}^{\perp_h})}{4\pi\mu\beta_{2_h}}$ and \eqref{lippmann-equation-Vt-Vi-terms-Ui-s-v} times $\frac{(U\cdot\hat{x}^{\perp_v})}{4\pi\mu\beta_{2_v}}$, we obtain
\begin{eqnarray}
\frac{(U\cdot\hat{x}^{\perp_h})}{4\pi\mu\beta_{2_h}}V^{t,\mathtt{s}_h}(y,-\hat{x})+\frac{(U\cdot\hat{x}^{\perp_v})}{4\pi\mu\beta_{2_v}}V^{t,\mathtt{s}_v}(y,-\hat{x})=\frac{(U\cdot\hat{x}^{\perp_h})}{4\pi\mu\beta_{2_h}}V^i_{\mathtt{s}_h}(y,-\hat{x})+\frac{(U\cdot\hat{x}^{\perp_v})}{4\pi\mu\beta_{2_v}}V^i_{\mathtt{s}_v}(y,-\hat{x})\nonumber\\
+\int_{\Omega}\omega^2(\rho_0(\zeta )-\tilde{\rho_0})\Gamma^\omega(y,\zeta )
\cdot \left[ \frac{(U\cdot\hat{x}^{\perp_h})}{4\pi\mu\beta_{2_h}}V^{t,\mathtt{s}_h}(\zeta ,-\hat{x})+\frac{(U\cdot\hat{x}^{\perp_v})}{4\pi\mu\beta_{2_v}} V^{t,\mathtt{s}_v}(\zeta ,-\hat{x})\right]\,d\zeta. \label{lippmann-equation-Vt-Vi-terms-Ui-s-h-adding-U-i-sv}
\end{eqnarray}

Making use the uniqueness of the solution to the Lippmann Schwinger equation \eqref{lippmann-equation-Vt-Vi-terms-incident-direction-elastic} and by comparing \eqref{lippmann-equation-Vt-Vi-terms-incident-direction-elastic}  with \eqref{farfield-elastic-G-Gamma-p-part}, we derive the required mixed reciprocity relation for $y\in \mathbb{R}^3\setminus \bar{\Omega}$, i.e.,
\begin{eqnarray}
{4\pi(\lambda+2\mu)\beta_1 G^{\omega,\infty}_\mathtt{p}(\hat{x},y)\cdot  {U}= (\hat{x}\cdot U)V^{t,\mathtt{p}}(y,-\hat{x}), \quad y\in \mathbb{R}^3\setminus \bar{\Omega}, \quad \hat{x}\in \mathbb{S}^2.\label{mixed-reciprocity-elastic-lemma-statement-y-outside-omega-1-p-field}}
\end{eqnarray}
Similarly, by comparing \eqref{lippmann-equation-Vt-Vi-terms-Ui-s-h-adding-U-i-sv} and \eqref{farfield-expansion-Gomega-s-part-Tnt-eqn}, we derive
\begin{eqnarray}
     G^{\omega,\infty}_\mathtt{s}(\hat{x},y)\cdot {U}=\frac{(U\cdot\hat{x}^{\perp_h})}{4\pi\mu\beta_{2_h}} V^{t,\mathtt{s}_h}(y,-\hat{x})+\frac{(U\cdot\hat{x}^{\perp_v})}{4\pi\mu\beta_{2_v}} V^{t,\mathtt{s}_v}(y,-\hat{x}),\quad y\in \mathbb{R}^3\setminus \bar{\Omega}, \quad \hat{x}^{\perp_h} \mbox{ and }\hat{x}^{\perp_v}\in \mathbb{S}^2.\label{mixed-reciprocity-elastic-lemma-statement-y-outside-omega-1-s-field}
\end{eqnarray}
The relations \eqref{mixed-reciprocity-elastic-lemma-statement-y-outside-omega-1-p-field} and \eqref{mixed-reciprocity-elastic-lemma-statement-y-outside-omega-1-s-field} hold true as $y$ approaches $\partial \Omega$.
\begin{remark}  Employing similar arguments as above, we can observe that corresponding conormal derivatives also satisfy the mixed reciprocity relation for $y\in\mathbb{R}^3\setminus\Omega$. Consequently, we obtain the following equations satisfied for $y\in\mathbb{R}^3\setminus\Omega$:
\begin{equation}
\left\{
\begin{array}{c c}
\frac{ (\hat{x}\cdot U)}{4\pi(\lambda+2\mu)\beta_1}V^{t,\mathtt{p}}(y,-\hat{x})=G^{\omega,\infty}_\mathtt{p}(\hat{x},y)\cdot U,& y\in \mathbb{R}^3\setminus {\Omega},\\
\frac{ (\hat{x}\cdot U)}{4\pi(\lambda+2\mu)\beta_1}T_\nu V^{t,\mathtt{p}}(y,-\hat{x})=T_{\nu}G^{\omega,\infty}_\mathtt{p}(\hat{x},y)\cdot U,& y\in \mathbb{R}^3\setminus {\Omega},
\end{array}\right.
\label{mixed-reciprocity-elastic-lemma-statement-y-outside-omega-conormal}
\end{equation} 
and 
\begin{equation}
\left\{
\begin{array}{c c}
     G^{\omega,\infty}_\mathtt{s}(\hat{x},y)\cdot  {U}=\frac{(U\cdot\hat{x}^{\perp_h})}{4\pi\mu\beta_{2_h}} V^{t,\mathtt{s}_h}(y,-\hat{x})+\frac{(U\cdot\hat{x}^{\perp_v})}{4\pi\mu\beta_{2_v}} V^{t,\mathtt{s}_v}(y,-\hat{x}), & y\in \mathbb{R}^3\setminus {\Omega},
     \\
     T_{\nu} G^{\omega,\infty}_\mathtt{s}(\hat{x},y)\cdot  {U}=\frac{(U\cdot\hat{x}^{\perp_h})}{4\pi\mu\beta_{2_h}} T_{\nu}V^{t,\mathtt{s}_h}(y,-\hat{x})+\frac{(U\cdot\hat{x}^{\perp_v})}{4\pi\mu\beta_{2_v}} T_{\nu}V^{t,\mathtt{s}_v}(y,-\hat{x}), &y\in \mathbb{R}^3\setminus {\Omega}.
\end{array}
\right.\label{mixed-reciprocity-elastic-lemma-statement-y-outside-omega-conormal-s-h-sv}
\end{equation}
\end{remark}
\end{itemize}
\item[]\textbf{Case 2 - $y\in{\Omega}$ :}
Now we prove the result for the second case, i.e., for $y\in\Omega$.\\
In this case, i.e., for $y\in\Omega$, firstly observing that $G^\omega$ is Green's matrix satisfying \eqref{Greens-function-variable-rho}, and using the representation formula associated with the interior problem, we obtain
  \begin{eqnarray}
  \int_{\partial\Omega}[G^\omega(\zeta,y)\,T_\nu G^\omega(x,\zeta)-T_\nu\,G^\omega(\zeta,y)G^\omega(x,\zeta)]\,ds(\zeta)=G^\omega(x,y), \quad y\in\Omega.\nonumber
  \end{eqnarray}
Making use the asymptotic expansion \eqref {assymptotic-expansion-Green} of $G^\omega$ for $|x|\to\infty$, we have 
\begin{align}
&\int_{\partial\Omega}[G^\omega(\zeta,y)\,T_\nu G^{\omega,\infty}_{\mathtt{s}}(\hat{x},\zeta)-T_\nu\,G^\omega(\zeta,y)G^{\omega,\infty}_{\mathtt{s}}(\hat{x},\zeta)]\,ds(\zeta)=G^{\omega,\infty}_{\mathtt{s}}(\hat{x},y),\label{G-s-infty-boundary-equation}
\\
\mbox{and }&\int_{\partial\Omega}[G^\omega(\zeta,y)\,T_\nu G^{\omega,\infty}_{\mathtt{p}}(\hat{x},\zeta)-T_\nu\,G^\omega(\zeta,y)G^{\omega,\infty}_{\mathtt{p}}(\hat{x},\zeta)]\,ds(\zeta)=G^{\omega,\infty}_{\mathtt{p}}(\hat{x},y).\label{boundary-integral-term-involving-Gomega-for-y-outside}
\end{align}
which further results in the following:
 \begin{align}
G^{\omega,\infty}_{\mathtt{p}}(\hat{x},y)\cdot U
  \underset{\eqref{mixed-reciprocity-elastic-lemma-statement-y-outside-omega-conormal}}{=}& \frac{1}{4\pi(\lambda+2\mu)\beta_1}\int_{\partial\Omega}(\hat{x}\cdot U)[G^\omega(\zeta,y)\cdot T_\nu V^{t,\mathtt{p}}(\zeta,-\hat{x})-T_\nu\,G^\omega(\zeta,y)\cdot V^{t,\mathtt{p}}(\zeta,-\hat{x})]\,ds(\zeta)\label{G-farfield-for-y-in-Omega},\\
G^{\omega,\infty}_{\mathtt{s}}(\hat{x},y)\cdot U\underset{\eqref{mixed-reciprocity-elastic-lemma-statement-y-outside-omega-conormal-s-h-sv}}{=}&\int_{\partial\Omega}[G^\omega(\zeta,y)\,\left[\frac{(U\cdot\hat{x}^{\perp_h})}{4\pi\mu\beta_{2_h}} T_{\nu}V^{t,\mathtt{s}_h}(\zeta,-\hat{x})+\frac{(U\cdot\hat{x}^{\perp_v})}{4\pi\mu\beta_{2_v}} T_{\nu}V^{t,\mathtt{s}_v}(\zeta,-\hat{x})\right]ds(\zeta)
\nonumber\\
&-\int_{\partial\Omega}T_\nu\,G^\omega(\zeta,y)\left[ \frac{(U\cdot\hat{x}^{\perp_h})}{4\pi\mu\beta_{2_h}} V^{t,\mathtt{s}_h}(y,-\hat{x})+\frac{(U\cdot\hat{x}^{\perp_v})}{4\pi\mu\beta_{2_v}} V^{t,\mathtt{s}_v}(y,-\hat{x})\right]\,ds(\zeta).\label{G-s-infty-boundary-equation-after-mixed-recip-rel}
\end{align}
 On the other hand, considering that $V^{t,\mathtt{p}}$, $V^{t,\mathtt{s}_h}$ and $V^{t,\mathtt{s}_v}$ satisfy the equation $\Delta^e V^t(y)+\rho_0(y)\omega^2 V^{t}(y)=0$ in $\Omega$ and $G^\omega$ is Green's matrix satisfying \eqref{Greens-function-variable-rho}, we get the following, for $y\in\Omega$:
  \begin{align}
 V^{t,\mathtt{p}}(y,-\hat{x})=&\int_{\partial\Omega} [G^\omega(\zeta,y)\cdot T_\nu V^{t,\mathtt{p}}(\zeta,-\hat{x})-T_\nu G^\omega(\zeta,y)\cdot V^{t,\mathtt{p}}(\zeta,-\hat{x})]\,ds(\zeta), 
 \label{lippmann-Vt-Green-for-y-in-Omega}\\
 V^{t,\mathtt{s}_h}(y,-\hat{x})=&\int_{\partial\Omega} [G^\omega(\zeta,y)\cdot T_\nu V^{t,\mathtt{s}_h}(\zeta,-\hat{x})-T_\nu G^\omega(\zeta,y)\cdot V^{t,\mathtt{s}_h}(\zeta,-\hat{x})]\,ds(\zeta),
 \label{lippmann-Vt-Green-for-y-in-Omega-sh-incident}\\
 V^{t,\mathtt{s}_v}(y,-\hat{x})=&\int_{\partial\Omega} [G^\omega(\zeta,y)\cdot T_\nu V^{t,\mathtt{s}_v}(\zeta,-\hat{x})-T_\nu G^\omega(\zeta,y)\cdot V^{t,\mathtt{s}_v}(\zeta,-\hat{x})]\,ds(\zeta)
 .\label{lippmann-Vt-Green-for-y-in-Omega-sv-incident}
\end{align}

Now, as it was done in Case 1, by appropriately comparing the equations (\ref{G-farfield-for-y-in-Omega}-\ref{G-s-infty-boundary-equation-after-mixed-recip-rel}) and (\ref{lippmann-Vt-Green-for-y-in-Omega}-\ref{lippmann-Vt-Green-for-y-in-Omega-sv-incident}), we achieve the required mixed reciprocity relations (\ref{lemma-mixed-reciprocity-elastic-p}-\ref{lemma-mixed-reciprocity-relation-s-incident-wave}) for $y\in \Omega$.
\end{enumerate}
 \end{proof}


\section{Appendix} \label{section-appendix} 

\subsection{Appendix A: Reconstruction of the mass density using $\mathtt{s}$-parts of the farfields}


 The reconstruction of mass density is possible by using the $\mathtt{p}$-incident wave, as explained in section \ref{section-Reconstruction-mass-density}. Given this, the immediate question one might have is whether reconstruction is achievable using only the $\mathtt{s}$-incident wave. However, to pursue this, one may need to solve the following nonlinear algebraic system (\ref{non-linear-system-eqn-1}--\ref{non-linear-system-eqn-6}).
 Let us define $\mathbf{X_{1}(\theta)}:=V^{t,\mathtt{s}_h}(z,\theta)\cdot \hat{x}$, $\mathbf{X_{2}(\theta)}:=V^{t,\mathtt{s}_h}(z,\theta)\cdot \hat{x}^{\perp_h}$, $\mathbf{X_{3}(\theta)}:=V^{t,\mathtt{s}_h}(z,\theta)\cdot \hat{x}^{\perp_v}$, $\mathbf{Y_1(\theta)}:=V^{t,\mathtt{s}_v}(z,\theta)\cdot\hat{x}$, $\mathbf{Y_2(\theta)}:=V^{t,\mathtt{s}_v}(z,\theta)\cdot\hat{x}^{\perp_h}$, and $\mathbf{Y_3(\theta)}:=V^{t,\mathtt{s}_v}(z,\theta)\cdot\hat{x}^{\perp_v}$
 
 \begin{align}
     &\frac{\omega^2\omega_{n_0}^2}{\omega_{n_0}^2-\omega^2} \frac{K\,a^3\alpha(z)}{4\pi\mu\beta_{2}}\left[ \mathbf{X_{2}}(\theta)\mathbf{X_{1}}(-\hat{x})+ \mathbf{X_3}(\theta)\mathbf{Y_1}(-\hat{x})\right] = O(a^{\min\{1,2-2h\}})\label{non-linear-system-eqn-1}
     \\
 &\nonumber\\
 &\frac{\omega^2\omega_{n_0}^2}{\omega_{n_0}^2-\omega^2} \frac{K\,a^3\alpha(z)}{4\pi\mu\beta_{2}}\left[ \mathbf{Y_2}(\theta)\mathbf{X_1}(-\hat{x})+\mathbf{Y_3}(\theta)\mathbf{Y_1}(-\hat{x})\right] = O(a^{\min\{1,2-2h\}})\label{non-linear-system-eqn-2}.
\end{align}
\begin{align}
U^{\infty,\mathtt{s}_h}_{\mathtt{s}}(\hat{x},\theta)\cdot \hat{x}^{\perp_h}=&V^{\infty,\mathtt{s}_h}_{\mathtt{s}}(\hat{x},\theta)\cdot \hat{x}^{\perp_h}+\frac{\omega^2\omega_{n_0}^2}{\omega_{n_0}^2-\omega^2} \frac{a^3K\alpha(z)}{4\pi\mu\beta_{2}}\left[\mathbf{X_2}(\theta)\mathbf{X_2}(-\hat{x}) 
     +\mathbf{X_3}(\theta)\mathbf{
     Y_2}(-\hat{x})\right] + O(a^{\min\{1,2-2h\}}),\label{non-linear-system-eqn-3}
\\
U^{\infty,\mathtt{s}_v}_{\mathtt{s}}(\hat{x},\theta)\cdot \hat{x}^{\perp_v}=&V^{\infty,\mathtt{s}_v}_{\mathtt{s}}(\hat{x},\theta)\cdot \hat{x}^{\perp_v}+\frac{\omega^2\omega_{n_0}^2}{\omega_{n_0}^2-\omega^2} \frac{a^3K\alpha(z)}{4\pi\mu\beta_{2}}\left[\mathbf{Y_2}(\theta)\mathbf{X_3}(-\hat{x})+\mathbf{Y_3}(\theta)\mathbf{Y_3}(-\hat{x})\right] + O(a^{\min\{1,2-2h\}})\label{non-linear-system-eqn-4}
\\
     U^{\infty,\mathtt{s}_h}_{\mathtt{s}}(\hat{x},\theta)\cdot \hat{x}^{\perp_v}=&V^{\infty,\mathtt{s}_h}_{\mathtt{s}}(\hat{x},\theta)\cdot \hat{x}^{\perp_v}+\frac{\omega^2\omega_{n_0}^2}{\omega_{n_0}^2-\omega^2} \frac{a^3K\alpha(z)}{4\pi\mu\beta_{2}}\left[\mathbf{X_2}(\theta)\mathbf{X_3}(-\hat{x}))+ \mathbf{X_3}(\theta)\mathbf{Y_3}(-\hat{x})\right] + O(a^{\min\{1,2-2h\}})\label{non-linear-system-eqn-5}
\\
U^{\infty,\mathtt{s}_v}_{\mathtt{s}}(\hat{x},\theta)\cdot \hat{x}^{\perp_h}=&V^{\infty,\mathtt{s}_v}_{\mathtt{s}}(\hat{x},\theta)\cdot \hat{x}^{\perp_h}+\frac{\omega^2\omega_{n_0}^2}{\omega_{n_0}^2-\omega^2} \frac{a^3K\alpha(z)}{4\pi\mu\beta_{2}}\left[\mathbf{Y_2}(\theta)\mathbf{X_2}(-\hat{x})+\mathbf{Y_3}(\theta)\mathbf{Y_2}(-\hat{x})\right] + O(a^{\min\{1,2-2h\}})\label{non-linear-system-eqn-6}
\end{align}
  The above nonlinear 
  algebraic system (\ref{non-linear-system-eqn-1}--\ref{non-linear-system-eqn-6}) is derived from \eqref{s-part-farfield-exp-in-h-incident} and \eqref{s-part-farfield-exp-in-v-incident}, knowing that $s$ parts of the farfields $U^\infty_{\mathtt{s}}$ and $V^\infty_{\mathtt{s}}$ are orthogonal to the observational direction, i.e., $U^\infty_\mathtt{s}\cdot \hat{x}=0$ and $V^\infty_{\mathtt{s}}\cdot \hat{x}=0$ (see \cite{alves2002far} for instance), and considering their components in the direction of $\hat{x}^{\perp_h}$ and $\hat{x}^{\perp_h}$. By considering the backscattered measurements in (\ref{non-linear-system-eqn-1}--\ref{non-linear-system-eqn-6}), one observe that we need to estimate six unknowns $\mathbf{X_i}(-\hat{x})$, $\mathbf{Y_i}(-\hat{x})$, $i=1,2,3$ associated to spherical shaped $D$, in which case  $E_{n_0}^{D_1}$ simplifies to $E_{n_0}^{D_1}=a^3KI_{3\times 3}$, where $K$ is a constant independent of $a$. So far, it is unclear to us how to proceed further.

\subsection{Appendix B: Estimation of the Green's function singularities.}

 \begin{lemma}\label{lemma-fundamental-Gree-properties-inside-D} For $x,y$ in $D$, we have the following properties. For any domain $D$:
 \begin{enumerate}
 \item \label{lemma-Gamma-x,y-bounded-x,y-in-D} $(\Gamma^\omega-\Gamma^0)(x,y) $ is bounded.
 \item $G^\omega(x,y)-\Gamma^0(x,y)$ is bounded .
 \end{enumerate}
\end{lemma}
\begin{proof}
\begin{enumerate}
\item Making use of the representations \eqref{entrywise-FM} and \eqref{entrywise-FM-zerof} of the fundamental matrix $\Gamma^{\omega}(x,y)$ and the Kelvin matrix $\Gamma^{0}(x,y)$ respectively, the $ij^{\mbox{th}}$ entry of $(\Gamma^\omega-\Gamma^0)$ is given by,
\begin{eqnarray}\label{series-repFS-dif-Gw-G0}
(\Gamma^{\omega}-\Gamma^0)_{ij}(x,y)&=&
\sum_{n=0} a_n + \sum_{n=0} b_n
\end{eqnarray}
\begin{eqnarray}
\mbox{with  }\quad a_n&:=&\frac{1}{4\pi\tilde{\rho_0}}\left[{\dfrac{\mathbf{\mathtt{i}}^{n+1}}{(n+3)(n+1)!}\left(\dfrac{n+2}{c_{\mathtt{s}}^{n+3}}+\dfrac{1}{c_{\mathtt{p}}^{n+3}}\right)\omega^{n+1}\delta_{ij}|x-y|^{n}}\right], \mbox{ and }\nonumber\\
 {b_n}&:=&\frac{-1}{4\pi\tilde{\rho_0}}\left[{\dfrac{\mathbf{\mathtt{i}}^{n+1}(n)}{(n+3)(n+1)!}\left(\dfrac{1}{c_{\mathtt{s}}^{n+3}}-\dfrac{1}{c_{\mathtt{p}}^{n+3}}\right)\omega^{n+1}|x-y|^{n-2}(x_i-y_i)(x_j-y_j)}\right]\nonumber.
\end{eqnarray}
Hence to show  {that} $(\Gamma^{\omega}-\Gamma^0)_{ij}(x,y)$ is bounded for $x,y\in D $ it is sufficient to show $\sum|{a_n}|$ and $\sum|{b_n}|$ are convergent, and it holds  true under the assumption $\max\left\{\frac{1}{2}\kappa_\mathtt{s}|x-y|,\frac{1}{2}\kappa_\mathtt{p}|x-y|\right  \}<1$.
{Infact, we can get the uniform bound for the $(\Gamma^{\omega}-\Gamma^0)_{ij}(x,y)$ for $x, y$ in $D$, as follows;}

From \eqref{series-repFS-dif-Gw-G0}, we have  
\begin{eqnarray}
|(\Gamma^\omega-\Gamma^0)_{ij}(x,y)|&\leq&\frac{1}{4\pi\tilde{\rho_0}}\sum_{n=1}^{\infty}{\dfrac{1}{(n+2)n!}\left(\dfrac{n+1}{c_{\mathtt{s}}^{n+2}}+\dfrac{1}{c_{\mathtt{p}}^{n+2}}\right)\omega^{n}\delta_{ij}|x-y|^{n-1}}
\nonumber\\
&&+\frac{1}{4\pi\tilde{\rho_0}}\sum_{n=1}^{\infty}{\dfrac{n-1}{(n+2)n!}\left(\dfrac{1}{c_{\mathtt{s}}^{n+2}}+\dfrac{1}{c_{\mathtt{p}}^{n+2}}\right)\omega^{n}|x-y|^{n-3}|x-y||x-y|}
\nonumber\\
&\leq & \dfrac{1}{4\pi\tilde{\rho_0}}\left[\dfrac{\kappa_{\mathtt{s}}^3}{\omega^2}\sum_{n=1}^{\infty}\frac{2}{(n+2)(n-1)!} (\kappa_{\mathtt{s}}|x-y|)^{n-1}+\dfrac{\kappa_{\mathtt{p}}^3}{\omega^2}\sum_{n=1}^{\infty}\frac{1}{(n+2)(n-1)!}(\kappa_{\mathtt{p}}|x-y|)^{n-1}\right]\nonumber\quad\quad\label{Gamma-w-Gamma-0-ij-mod-before-app-lemma}\\
&\leq  & H_4 \label{Gamma-Gamma0-bounded-for-x-y-in-D}
\end{eqnarray}
\mbox{where, } $H_4:=\dfrac{1}{4\pi\tilde{\rho_0}}\left[\dfrac{2\kappa_{\mathtt{s}}}{c_{\mathtt{s}}^2}\dfrac{1}{1-\frac{1}{2}\kappa_{\mathtt{s}}\,diam(D)}+\dfrac{\kappa_{\mathtt{p}}}{c_{\mathtt{s}}^2}\dfrac{1}{1-\frac{1}{2}\kappa_{\mathtt{p}}\,diam(D)}\right].$ The last inequality in the above holds true under the assumption that `$\frac{1}{2}\max\{\kappa_{\mathtt{s}},\kappa_{\mathtt{p}}\}\,diam(D)<1$' and fact that `$n!\geq 2^{n-1}$, $\forall n\in\mathbb{N}$'. 
\medskip
\item We can prove this results following the procedure adopted from \cite[Lemma 2.3]{alsaedi2016extraction},  \cite[Lemma 3.11]{challa2017mathematical}, \cite{al2016equivalent} and \cite{hahner2001new}).
Write, 
\begin{eqnarray}
\mathcal{H}^\omega(x,y):=G^\omega(x,y)-\Gamma^\omega(x,y).\label{G-omega-in-terms-of-gamma-omega}
\end{eqnarray}
Then, from \eqref{Greens-function-variable-rho} and \eqref{fundamental-homogeneous-background-massdensity-tildrho}, we can observe that $\mathcal{H}^\omega(x,y)$ satisfies the  following equation on $D$,
\begin{eqnarray}
\Delta^e\mathcal{H}^\omega(x,y)+\omega^2\rho_0(y)\mathcal{H}^\omega(x,y)=\omega^2 (\tilde{\rho_0}-\rho_0(y))\Gamma^\omega(x,y),\label{H-omega-in-terms-Green-function-G-var-mass}
\end{eqnarray}
along with the \textit{K.R.C.} 

Since, $\Gamma^\omega(x,y)$ is in $(L^2_{loc}(\mathbb{R}^3))^3$ {, and hence}   $\omega^2(\tilde{\rho_0}-\rho_0(y))\Gamma^\omega(x,y)$ is in $(L^2_{loc}(\mathbb{R}^3))^3$, 
  {then making} use of the interior regularity properties of the Lam\'e system we deduce that $\mathcal{H}^\omega(x,y)$ is in $(W_{loc}^{2,2}(\mathbb{R}^3))^3$ and then by Sobolev embedding that $(W^{2,2}_{loc}(\mathbb{R}^3))^3\, \hookrightarrow \,L^\infty$, we obtain $\mathcal{H}^\omega\in L^\infty$. Thus, $G^\omega-\Gamma^\omega$ is bounded by a constant, say $H_3$, on $D$. Combining this with \eqref{Gamma-Gamma0-bounded-for-x-y-in-D}, we get
\begin{eqnarray}
|G^\omega(x,y)-\Gamma_z^0(x,y)|\leq (H_3+H_4).\label{G-omega-Gamma-omega-bounded-for-x,y-in-D}
\end{eqnarray}
\end{enumerate}
\end{proof}

\begin{remark}\label{lemma-grad-G-for-x-away-from-omega}
For  $x$ away from $\Omega$ and for $y\in D$, we observe that $G_{ij}^\omega(x,y)$, and also $\int_{0}^{1}|\nabla_yG^{\omega}(x,z_1+t(y-z_1))_{ij}|\,dt$, behaves as $O(1)$. 
\end{remark}
\begin{remark}\label{lemma-int-int-D-|y-z|-Gamma-bounded} The term
{$\int_{D_1}\int_{D_1}|y-z|^2|[\Gamma_z^0(x,y)]_{ij}|^2\,dy\,dx$ behaves as $a^6$, precisely,}
\begin{align}
\int_{D_1}\int_{D_1}|y-z|^2|[\Gamma_z^0(x,y)]_{ij}|^2\,dy\,dx
&\leq {diam(B)}^2a^6 C_3, \mbox{ where }{C_3\hspace{-.1cm}}:=\hspace{-.1cm}\left[\dfrac{|\gamma_1|+|\gamma_2|}{4\pi}\right]^2\hspace{-.1cm}\int_{B}\int_B\dfrac{1}{|\xi_1-\xi_2|^2}\,d\xi_1\,d\xi_2. \label{int-D-int-D-|gamma-0-x,y-square}
\end{align}
\end{remark}

\bibliographystyle{abbrv}

\begin{thebibliography}{10}

\bibitem{al2016equivalent}
F.~Al-Musallam, D.~P. Challa, and M.~Sini.
\newblock The equivalent medium for the elastic scattering by many small rigid
  bodies and applications.
\newblock {\em IMA J. Appl. Math.}, 81(6):1020--1050, 2016.

\bibitem{alsaedi2016extraction}
A.~Alsaedi, F.~Alzahrani, D.~P. Challa, M.~Kirane, and M.~Sini.
\newblock Extraction of the index of refraction by embedding multiple small
  inclusions.
\newblock {\em Inverse Problems}, 32(4):045004, 18, 2016.

\bibitem{alves2002far}
C.~J.~S. Alves and R.~Kress.
\newblock On the far-field operator in elastic obstacle scattering.
\newblock {\em IMA J. Appl. Math.}, 67(1):1--21, 2002.

\bibitem{AMMARI-BIOMEDICAL-IMAGING}
H.~Ammari.
\newblock {\em An introduction to mathematics of emerging biomedical imaging},
  volume~62 of {\em Math\'{e}matiques \& Applications (Berlin) [Mathematics \&
  Applications]}.
\newblock Springer, Berlin, 2008.

\bibitem{AMMARI-MATHEMATICAL-METHODS-IN-ELASTIC-IMAGING}
H.~Ammari, E.~Bretin, J.~Garnier, H.~Kang, H.~Lee, and A.~Wahab.
\newblock {\em Mathematical methods in elasticity imaging}.
\newblock Princeton Series in Applied Mathematics. Princeton University Press,
  Princeton, NJ, 2015.

\bibitem{AMMARI}
H.~Ammari, P.~Calmon, and E.~Iakovleva.
\newblock Direct elastic imaging of a small inclusion.
\newblock {\em SIAM J. Imaging Sci.}, 1(2):169--187, 2008.

\bibitem{Ammar-all-bubble}
H.~Ammari, D.~P. Challa, A.~P. Choudhury, and M.~Sini.
\newblock The point-interaction approximation for the fields generated by
  contrasted bubbles at arbitrary fixed frequencies.
\newblock {\em J. Differential Equations}, 267(4):2104--2191, 2019.

\bibitem{Ammari-all-Volume-potential}
H.~Ammari, D.~P. Challa, A.~P. Choudhury, and M.~Sini.
\newblock The equivalent media generated by bubbles of high contrasts:
  volumetric metamaterials and metasurfaces.
\newblock {\em Multiscale Model. Simul.}, 18(1):240--293, 2020.

\bibitem{Ammari-F-G-L-Z}
H.~Ammari, B.~Fitzpatrick, D.~Gontier, H.~Lee, and H.~Zhang.
\newblock Minnaert resonances for acoustic waves in bubbly media.
\newblock {\em Ann. Inst. H. Poincar\'{e} C Anal. Non Lin\'{e}aire},
  35(7):1975--1998, 2018.

\bibitem{LAYER-POTENTIAL-TECHNIQUES}
H.~Ammari, H.~Kang, and H.~Lee.
\newblock {\em Layer potential techniques in spectral analysis}, volume 153 of
  {\em Mathematical Surveys and Monographs}.
\newblock American Mathematical Society, Providence, RI, 2009.

\bibitem{Ammari-all-2023}
H.~Ammari, B.~Li, and J.~Zou.
\newblock Mathematical analysis of electromagnetic scattering by dielectric
  nanoparticles with high refractive indices.
\newblock {\em Trans. Amer. Math. Soc.}, 376(1):39--90, 2023.

\bibitem{Bal-Review}
G.~Bal.
\newblock Hybrid inverse problems and internal functionals.
\newblock In {\em Inverse problems and applications: inside out. {II}},
  volume~60 of {\em Math. Sci. Res. Inst. Publ.}, pages 325--368. Cambridge
  Univ. Press, Cambridge, 2013.

\bibitem{Bal-all-review-2}
G.~Bal, E.~Bonnetier, F.~Monard, and F.~Triki.
\newblock Inverse diffusion from knowledge of power densities.
\newblock {\em Inverse Probl. Imaging}, 7(2):353--375, 2013.

\bibitem{Bal-all-Review-3}
G.~Bal and G.~Uhlmann.
\newblock Reconstruction of coefficients in scalar second-order elliptic
  equations from knowledge of their solutions.
\newblock {\em Comm. Pure Appl. Math.}, 66(10):1629--1652, 2013.

\bibitem{CGS-2023}
X.~Cao, A.~Ghandriche, and M.~Sini.
\newblock The electromagnetic waves generated by a cluster of nanoparticles
  with high refractive indices.
\newblock {\em J. Lond. Math. Soc. (2)}, 108(4):1531--1616, 2023.

\bibitem{challa2017mathematical}
D.~P. Challa, A.~P. Choudhury, and M.~Sini.
\newblock Mathematical imaging using electric or magnetic nanoparticles as
  contrast agents.
\newblock {\em Inverse Probl. Imaging}, 12(3):573--605, 2018.

\bibitem{challa-divya-sini-2024}
D.~P. Challa, D.~Gangadaraiah, and M.~Sini.
\newblock Elastic fields generated by multiple small inclusions with high mass
  density at nearly resonant frequencies.
\newblock {\em J. Math. Anal. Appl.}, 538(2):Paper No. 128442, 2024.

\bibitem{DPC-MS-InvProblems-2012}
D.~P. Challa and M.~Sini.
\newblock Inverse scattering by point-like scatterers in the foldy regime.
\newblock {\em Inverse Problems}, 28(12):125006, nov 2012.

\bibitem{Choulli-book}
M.~Choulli.
\newblock {\em Une introduction aux probl\`emes inverses elliptiques et
  paraboliques}, volume~65 of {\em Math\'{e}matiques \& Applications (Berlin)
  [Mathematics \& Applications]}.
\newblock Springer-Verlag, Berlin, 2009.

\bibitem{COLTON-IE}
D.~Colton and R.~Kress.
\newblock {\em Integral equation methods in scattering theory}, volume~72 of
  {\em Classics in Applied Mathematics}.
\newblock Society for Industrial and Applied Mathematics (SIAM), Philadelphia,
  PA, 2013.

\bibitem{COLTON-KRESS}
D.~Colton and R.~Kress.
\newblock {\em Inverse acoustic and electromagnetic scattering theory},
  volume~93 of {\em Applied Mathematical Sciences}.
\newblock Springer, Cham, [2019] \copyright 2019.

\bibitem{DGS-2021}
A.~Dabrowski, A.~Ghandriche, and M.~Sini.
\newblock Mathematical analysis of the acoustic imaging modality using bubbles
  as contrast agents at nearly resonating frequencies.
\newblock {\em Inverse Probl. Imaging}, 15(3):555--597, 2021.

\bibitem{LOW-FREQUENCY}
G.~Dassios and R.~Kleinman.
\newblock {\em Low frequency scattering}.
\newblock Oxford Mathematical Monographs. The Clarendon Press, Oxford
  University Press, New York, 2000.
\newblock Oxford Science Publications.

\bibitem{Engl-all-book}
H.~W. Engl, M.~Hanke, and A.~Neubauer.
\newblock {\em Regularization of inverse problems}, volume 375 of {\em
  Mathematics and its Applications}.
\newblock Kluwer Academic Publishers Group, Dordrecht, 1996.

\bibitem{F-M-S2003}
E.~Fear, P.~Meaney, and M.~Stuchly.
\newblock Microwaves for breast cancer detection.
\newblock {\em IEEE Potentials}, 22(1):12--18, 2003.

\bibitem{GS-2022-3}
A.~Ghandriche and M.~Sini.
\newblock An introduction to the mathematics of the imaging modalities using
  small scaled contrast agents.
\newblock {\em ICCM Not.}, 10(1):28--43, 2022.

\bibitem{GS-2022}
A.~Ghandriche and M.~Sini.
\newblock Mathematical analysis of the photo-acoustic imaging modality using
  resonating dielectric nano-particles: the 2{$D$}{TM}-model.
\newblock {\em J. Math. Anal. Appl.}, 506(2):Paper No. 125658, 64, 2022.

\bibitem{GS-2022-2}
A.~Ghandriche and M.~Sini.
\newblock Photo-acoustic inversion using plasmonic contrast agents: the full
  {M}axwell model.
\newblock {\em J. Differential Equations}, 341:1--78, 2022.

\bibitem{Drosso-Mourad}
D.~Gintides and M.~Sini.
\newblock Identification of obstacles using only the scattered p-waves or the
  scattered s-waves.
\newblock {\em Inverse Problems and Imaging}, 6(1):39--55, 2012.

\bibitem{hahner2001new}
P.~H\"{a}hner and T.~Hohage.
\newblock New stability estimates for the inverse acoustic inhomogeneous medium
  problem and applications.
\newblock {\em SIAM J. Math. Anal.}, 33(3):670--685, 2001.

\bibitem{HAN-all}
X.~Han, K.~Xu, O.~Taratula, and K.~Farsad.
\newblock Applications of nanoparticles in biomedical imaging.
\newblock {\em Nanoscale}, 11(3):799--819, 2019.

\bibitem{Isakov-book}
V.~Isakov.
\newblock {\em Inverse problems for partial differential equations}, volume 127
  of {\em Applied Mathematical Sciences}.
\newblock Springer, Cham, third edition, 2017.

\bibitem{Kar-Mourad}
M.~Kar and M.~Sini.
\newblock On the inverse elastic scattering by interfaces using one type of
  scattered waves.
\newblock {\em Journal of Elasticity}, 118(1):15--38, 2015.

\bibitem{KIM-all}
J.~Kim, P.~Chhour, J.~Hsu, H.~I. Litt, V.~A. Ferrari, R.~Popovtzer, and D.~P.
  Cormode.
\newblock Use of nanoparticle contrast agents for cell tracking with computed
  tomography.
\newblock {\em Bioconjugate chemistry}, 28(6):1581--1597, 2017.

\bibitem{kupradze1967potential}
V.~D. Kupradze.
\newblock {\em Potential methods in the theory of elasticity}.
\newblock Israel Program for Scientific Translations, Jerusalem; Daniel Davey
  \& Co., Inc., New York, 1965.
\newblock Translated from the Russian by H. Gutfreund, Translation edited by I.
  Meroz.

\bibitem{kupradze1980three}
V.~D. Kupradze, T.~G. Gegelia, M.~O. Bashele\u{\i}shvili, and T.~V.
  Burchuladze.
\newblock {\em Three-dimensional problems of the mathematical theory of
  elasticity and thermoelasticity}, volume~25 of {\em North-Holland Series in
  Applied Mathematics and Mechanics}.
\newblock North-Holland Publishing Co., Amsterdam-New York, russian edition,
  1979.
\newblock Edited by V. D. Kupradze.

\bibitem{Lakes}
R.~S. Lakes.
\newblock Negative-poisson's-ratio materials: Auxetic solids.
\newblock {\em Annual review of materials research}, 47:63--81, 2017.

\bibitem{Chen-Li-2015}
W.~Li and X.~Chen.
\newblock Gold nanoparticles for photoacoustic imaging.
\newblock {\em Nanomedicine}, 10(2):299--320, 2015.

\bibitem{Lim}
T.-C. Lim.
\newblock {\em Auxetic Materials and Structures}, volume~11.
\newblock Springer, 2015.

\bibitem{MPS-2022-volume-potential}
A.~Mantile, A.~Posilicano, and M.~Sini.
\newblock On the origin of {M}innaert resonances.
\newblock {\em Journal de Math{\'e}matiques Pures et Appliqu{\'e}es},
  165:106--147, 2022.

\bibitem{Mitchel-all}
M.~J. Mitchell, M.~M. Billingsley, R.~M. Haley, M.~E. Wechsler, N.~A. Peppas,
  and R.~Langer.
\newblock Engineering precision nanoparticles for drug delivery.
\newblock {\em Nature Reviews Drug Discovery}, 20(2):101--124, 2021.

\bibitem{C-F-Q2009}
S.~Qin, C.~Caskey, and K.~Ferrara.
\newblock Corrigendum: Ultrasound contrast microbubbles in imaging and therapy:
  physical principles and engineering ultrasound contrast microbubbles in
  imaging and therapy: physical principles and engineering.
\newblock {\em Physics in Medicine and Biology}, 54:4621, 07 2009.

\bibitem{Q2007}
E.~Quaia.
\newblock Microbubble ultrasound contrast agents: An update.
\newblock {\em European radiology}, 17:1995--2008, 09 2007.

\bibitem{SSW-2023}
S.~Senapati, M.~Sini, and H.~Wang.
\newblock Recovering both the wave speed and the source function in a
  time-domain wave equation by injecting contrasting droplets.
\newblock {\em Discrete and Continuous Dynamical Systems}, 44(5):1446--1474,
  2024.

\bibitem{SW-2022}
M.~Sini and H.~Wang.
\newblock The inverse source problem for the wave equation revisited: A new
  approach.
\newblock {\em SIAM Journal on Mathematical Analysis}, 54(5):5160--5181, 2022.

\bibitem{Z-L-L-S-W19}
P.~Zhang, L.~Li, L.~Lin, J.~Shi, and L.~Wang.
\newblock In vivo superresolution photoacoustic computed tomography by
  localization of single dyed droplets.
\newblock {\em Light: Science \& Applications}, 8:36, 04 2019.

\end{thebibliography}



\end{document}